\documentclass[a4paper,10pt]{amsart}
\usepackage{verbatim}
\usepackage{amssymb}
\usepackage{amsmath}
\usepackage{amsthm}
\usepackage{amsfonts}
\usepackage{url}
\usepackage{enumerate}
\usepackage{mathtools}

\usepackage{diagrams}

\usepackage[all,cmtip]{xy}

\theoremstyle{plain}
\newtheorem{theorem}{Theorem}[section]
\newtheorem{lemma}[theorem]{Lemma}
\newtheorem{prop}[theorem]{Proposition}
\newtheorem{proposition}[theorem]{Proposition}
\newtheorem{conjecture}[theorem]{Conjecture}
\newtheorem{fact}[theorem]{Fact}
\newtheorem{corollary}[theorem]{Corollary}

\theoremstyle{definition}
\newtheorem{definition}[theorem]{Definition}
\newtheorem{defn}[theorem]{Definition}
\newtheorem{notation}[theorem]{Notation}

\newtheorem{remark}[theorem]{Remark}
\newtheorem{remarks}[theorem]{Remarks}

\newtheorem{question}[theorem]{Question}

\newcommand{\abar}{{\ensuremath{\bar{a}}}}

\newcommand{\xbar}{{\ensuremath{\bar{x}}}}

\DeclareMathOperator{\rk}{rk}

\DeclareMathOperator{\acl}{acl}   

\DeclareMathOperator{\loc}{Loc}   

\newcommand{\Qab}{\Q^{\mathrm{ab}}}  
\newcommand{\alg}{\ensuremath{\mathrm{alg}}} 

\DeclareMathOperator{\ldim}{ldim}  
\DeclareMathOperator{\Mat}{Mat}  

\DeclareMathOperator{\ecl}{ecl} 

\newcommand{\N}{\ensuremath{\mathbb{N}}}
\newcommand{\Z}{\ensuremath{\mathbb{Z}}}
\newcommand{\Q}{\ensuremath{\mathbb{Q}}}
\newcommand{\R}{\ensuremath{\mathbb{R}}}
\newcommand{\C}{\ensuremath{\mathbb{C}}}

\newcommand{\Rexp}{\ensuremath{\mathbb{R}_{\mathrm{exp}}}}
\newcommand{\Cexp}{\ensuremath{\mathbb{C}_{\mathrm{exp}}}}

\newcommand{\Cat}{\ensuremath{\mathcal{C}}} 
\newcommand{\Loo}{\ensuremath{L_{\omega_1,\omega}}} 
\newcommand{\Looq}{\ensuremath{\Loo(Q)}}

\newcommand{\ga}{\ensuremath{\mathbb{G}_\mathrm{a}}}   
\newcommand{\gm}{\ensuremath{\mathbb{G}_\mathrm{m}}}  
\newcommand{\GL}{\ensuremath{\mathrm{GL}}}  

\renewcommand{\phi}{\varphi}
\renewcommand{\le}{\ensuremath{\leqslant}}
\renewcommand{\ge}{\ensuremath{\geqslant}}

\newcommand{\class}[2]{\ensuremath{\left\{ #1 \,\left|\, #2 \right.\right\}}}

\newcommand{\iso}{\cong}

\newcommand{\into}{\hookrightarrow}
\newcommand{\onto}{\twoheadrightarrow}

\newcommand{\subs}{\subseteq} 
\newcommand{\sups}{\supseteq} 

\newcommand{\minus}{\ensuremath{\smallsetminus}}
\newcommand{\powerset}{\ensuremath{\mathcal{P}}} 

\newcommand{\strong}{\ensuremath{\lhd}} 
\newcommand{\nstrong}{\ensuremath{\not\kern-4pt\lhd\;}} 

\newcommand{\gen}[1]{\ensuremath{\left\langle #1 \right\rangle}} 
\newcommand{\hull}[1]{\ensuremath{\lceil #1\rceil}}

\newcommand{\cross}{\ensuremath{\times}}

\newbox\noforkbox \newdimen\forklinewidth
\forklinewidth=0.3pt \setbox0\hbox{$\textstyle\smile$}
\setbox1\hbox to \wd0{\hfil\vrule width \forklinewidth depth-2pt
  height 10pt \hfil}
\wd1=0 cm \setbox\noforkbox\hbox{\lower 2pt\box1\lower
2pt\box0\relax}
\def\unionstick{\mathop{\copy\noforkbox}\limits}

\def\nonfork_#1{\unionstick_{\textstyle #1}}

\setbox0\hbox{$\textstyle\smile$} \setbox1\hbox to \wd0{\hfil{\sl
/\/}\hfil} \setbox2\hbox to \wd0{\hfil\vrule height 10pt depth
-2pt width
                \forklinewidth\hfil}
\newbox\doesforkbox
\setbox\doesforkbox\hbox{\lower 2pt\box1 \lower
2pt\box2\lower2pt\box0\relax}
\def\nunionstick{\mathop{\copy\doesforkbox}\limits}

\def\fork_#1{\nunionstick_{\textstyle #1}}

\newcommand{\ra}[3]{\ensuremath{#1 \stackrel{#2}{\longrightarrow} #3}}

\newcommand{\leteq}{\mathrel{\mathop:}=}

\newarrowmiddle{normal}\lhtriangle\rttriangle\uhtriangle\dhtriangle

\newarrowmiddle{iso}{\cong}{\cong}{}{}   

\newarrow{Partial}{-}{-}{-}{-}{harpoon}     

\newarrow{Equal}=====     

\newarrow{Strong}{hooka}-{normal}->     

\newarrow{Str}{}{}{normal}{}{}     

\newarrow{Iso}--{iso}->     

\newarrow{NT}====>     

\newcommand{\DEQ}{(DEQ)}   
\newcommand{\DE}{\DEQ}

\newcommand{\struct}[1]{\ensuremath{\left\langle #1 \right\rangle}}

\providecommand{\Cexp}{\mathbb{C}_{\exp}}
\newcommand{\B}{\ensuremath{\mathbb{B}}}

\newcommand{\trd}{\operatorname{trd}}

\newcommand{\td}{\trd}

\newcommand{\pd}{\operatorname{\eta}} 

\providecommand{\C}{\mathbb{C}}
\providecommand{\R}{\mathbb{R}}
\providecommand{\Q}{\mathbb{Q}}
\providecommand{\Z}{\mathbb{Z}}
\providecommand{\N}{\mathbb{N}}
\providecommand{\G}{\mathbb{G}}
\providecommand{\Ga}{\mathbb{G}_{\text{a}}}
\providecommand{\Gm}{\mathbb{G}_{\text{m}}}
\providecommand{\OO}{\ensuremath{\mathcal{O}}}

\providecommand{\powerset}{\mathcal{P}}

\providecommand{\Union}{\bigcup}

\newcommand{\isom}{\cong}

\newcommand{\tuple}[1]{\langle #1 \rangle}

\newcommand{\tp}{\operatorname{tp}}
\newcommand{\qftp}{\operatorname{qftp}}

\providecommand{\acl}{\operatorname{acl}}
\providecommand{\Gal}{\operatorname{Gal}}

\newcommand{\theorystyle}[1]{\operatorname{#1}}
\newcommand{\ACF}{\theorystyle{ACF}}

\providecommand{\cl}{\operatorname{cl}}
\providecommand{\ecl}{\operatorname{ecl}}

\providecommand{\G}{\mathbb{G}}
\providecommand{\End}{\operatorname{End}}
\providecommand{\Tor}{\operatorname{Tor}}

\newcommand{\epar}{\ar@{->>}}
\newcommand{\embar}{\ar@{^{(}->}}
\newcommand{\lstrar}{\ar@{^{(}->}_{\triangleleft}}
\newcommand{\hstrar}{\ar@{^{(}->}^{\triangleleft}}
\newcommand{\hstrarl}{\ar@{_{(}->}_{\triangleleft}}

\newcommand{\epexar}{\ar@{-->>}}
\newcommand{\embexar}{\ar@{^{(}-->}}
\newcommand{\lstrexar}{\ar@{^{(}-->}_{\triangleleft}}
\newcommand{\hstrexar}{\ar@{^{(}-->}^{\triangleleft}}

\newcommand{\Cc}{\mathcal{C}}

\newcommand{\Cfg}{\Cc^{\mathrm{fg}}}
\newcommand{\Ccount}{\Cc^{\le\aleph_0}}
\newcommand{\Cfullcount}{\Cc^{\mathrm{full},\le\aleph_0}}
\newcommand{\Cfull}{\Cc^{\mathrm{full}}}
\newcommand{\Cfgfull}{\Cc^{\mathrm{fg\text{-}full}}}
\newcommand{\Cfullfg}{\Cfgfull}
\newcommand{\Calg}{\Cc^{\mathrm{alg}}}

\newcommand{\Ctrans}{\mathcal{C}_{\Gamma\text{-}\mathrm{tr}}}
\newcommand{\Ctrfg}{\mathcal{C}_{\Gamma\text{-}\mathrm{tr}}^{\mathrm{fg}}}
\newcommand{\Ctransfull}{\Ctrans^{\mathrm{full}}}

\newcommand{\K}{\ensuremath{\mathcal{K}}}

\newcommand{\Afull}{A^{\mathrm{full}}}
\newcommand{\Bfull}{B^{\mathrm{full}}}

\DeclareMathOperator{\Aut}{Aut}

\newcommand{\bk}{\ensuremath{K_0}} 
\providecommand{\er}{\mathcal{O}} 
\providecommand{\ek}{{k_\er}} 
\newcommand{\kO}{\ek}

\newcommand{\Fbase}{{F_\mathrm{base}}} 

\newcommand{\MM}{\ensuremath{\mathbb{M}}}  
\newcommand{\Mtr}{\ensuremath{M_{\Gamma\text{-}\mathrm{tr}}}}  

\providecommand{\ldim}{\operatorname{ldim}}
\providecommand{\trd}{\operatorname{trd}}

\newcommand{\closed}{\strong_{\cl}} 

\renewcommand{\Cat}{\ensuremath{\mathcal{K}}}

\newcommand{\A}{\ensuremath{\mathbb{A}}} 

\providecommand{\Q}{\mathbb{Q}}
\providecommand{\Qalg}{\Q^{\mathrm{alg}}}

\providecommand{\maps}{\to} 

\providecommand{\isom}{\cong}

\providecommand{\End}{\operatorname{End}}

\providecommand{\Tor}{\operatorname{Tor}}
\providecommand{\Gal}{\operatorname{Gal}}
\providecommand{\Zar}{\operatorname{Zar}}

\renewcommand{\hat}{\widehat}

\newcommand{\Lqe}{\ensuremath{L^{\mathrm{QE}}}}
\newcommand{\Gammadim}{\operatorname{\Gamma dim}}
\newcommand{\Gdim}{\Gammadim}

\newcommand{\Fraisse}{Fra\"iss\'e}

\newcommand{\ECFskccp}{\ensuremath{\mathrm{ECF}_\mathrm{SK,CCP}}}
\newcommand{\PCFskccp}{\ensuremath{\mathrm{\wp CF}_\mathrm{SK,CCP}}}

\newcommand{\GCF}{\ensuremath{\mathrm{\Gamma CF}}} 
\newcommand{\GCFccp}{\ensuremath{\GCF_{\mathrm{CCP}}}} 

\newcommand{\Gammacl}{\operatorname{\Gamma cl}}
\newcommand{\Gcl}{\Gammacl}

\newcommand{\Tate}{\ensuremath{\widehat{T}}}  

\newcommand{\strongembed}{\xhookrightarrow{\strong}}  

\title
{Pseudo-exponential maps, variants, and quasiminimality}
\date{Version 3.0, March 3 2017}

\author{Martin Bays \and Jonathan Kirby} 

\thanks{JK was supported by EPSRC grant EP/L006375/1}
\address{Martin Bays, Institut f\"ur Mathematische Logik und Grundlagenforschung,
Fachbereich Mathematik und Informatik,
Universit\"at M\"unster,
Einsteinstrasse 62,
48149 M\"unster, Germany}
\email{m.bays@math.uni-muenster.de}
\address{Jonathan Kirby, School of Mathematics, University of East Anglia, Norwich Research Park, Norwich NR4 7TJ, UK}
\email{jonathan.kirby@uea.ac.uk}

\keywords{Exponential fields, Predimension, Categoricity, Schanuel
Conjecture, Ax-Schanuel, Zilber-Pink, Quasiminimality, Kummer Theory}
\subjclass[2010]{Primary: 03C65; Secondary: 12L12, 03C75}

\begin{document}

\begin{abstract}

We give a construction of quasiminimal fields equipped with pseudo-analytic maps, generalising Zilber's pseudo-exponential function. In particular we construct pseudo-exponential maps of simple abelian varieties, including pseudo-$\wp$-functions for elliptic curves. We show that the complex field with the corresponding analytic function is isomorphic to the pseudo-analytic version if and only the appropriate version of Schanuel's conjecture is true and the corresponding version of the strong exponential-algebraic closedness property holds. 
Moreover, we relativize the construction to build a model over a fairly arbitrary countable subfield and deduce that the complex exponential field is quasiminimal if it is exponentially-algebraically closed. This property asks only that the graph of exponentiation have non-trivial intersection with certain algebraic varieties but does not require genericity of these points. Furthermore Schanuel's conjecture is not required as a condition for quasiminimality.

\end{abstract}

\maketitle

\tableofcontents


\section{Introduction}

\subsection{Exponential fields}
The field $\C$ of complex numbers is well-known to be \emph{strongly minimal}, that is, any subset of $\C$ definable in the ring language is either finite or cofinite. Consequently the model theory of $\C$ is very tame: there is a very well-understood behaviour of the models (one model of each uncountable cardinality, known as uncountable categoricity) and of the definable sets (they have finite Morley rank and we can understand them geometrically in terms of  algebraic varieties). The other most important mathematical field, the field $\R$ of real numbers, is \emph{o-minimal} which means that although the class of models is not well-behaved (not classifiable) there is a very good geometric understanding of the definable sets (they are the semialgebraic sets). Remarkably, Wilkie showed that when the real exponential function $e^x$ is adjoined, the structure $\Rexp$ is still o-minimal \cite{Wilkie96}. Adjoining the complex exponential function $e^z$ to $\C$ gives the structure $\Cexp$ which cannot be well-behaved in terms of the class of models or the definable sets because it interprets the ring $\Z$. However, Zilber suggested that in the model $\Cexp$ itself, the influence of $\Z$ might only extend to the countable subsets of $\C$. He made the following conjecture.

\begin{conjecture}[Zilber's weak quasiminimality conjecture]\label{weak qm conj}
The complex exponential field $\Cexp = \langle \C;+,\cdot,\exp\rangle$ is \emph{quasiminimal}, that is, every subset of $\C$ definable in $\Cexp$ is either countable or co-countable.
\end{conjecture}
A slightly stronger version of the conjecture which avoids reference to definable sets is that every automorphism-invariant subset is countable or co-countable. As far as we are aware, all known approaches to the conjecture would give this stronger result anyway. If the conjecture is true then the solution sets of exponential polynomial equations, which we can call complex exponential varieties, would be expected to  have good geometric properties similar to those of algebraic varieties, provided we avoid some exceptional cases like $\Z$. If the conjecture is false, another possibility is that $\R$ is definable as a subset of $\C$. The field $\R$ with $\Z$ as a definable subset is so-called \emph{second-order arithmetic}, and the definable sets are extremely wild, with no geometric properties in general.

As one approach to his conjecture, Zilber \cite{Zilber00fwpe}, \cite{Zilber05peACF0} showed how to construct a quasiminimal exponential field we call $\B$ using a variant of Hrushovski's predimension method from \cite{Hru93}. He called $\B$ a \emph{pseudo-exponential field} with the idea that the exponential map is a \emph{pseudo-analytic} function.

Zilber's approach was to prove that a certain list of axioms $\ECFskccp$ in the infinitary logic $\Looq$ behaves in an analogous way to a strongly minimal first-order theory. In particular, all its models are quasiminimal and it is uncountably categorical.
\begin{theorem}\label{main exp theorem}
Up to isomorphism there is exactly one model of the axioms $\ECFskccp$ of each uncountable cardinality, and it is quasiminimal.
\end{theorem}
This theorem appears in \cite{Zilber05peACF0}. Some gaps in the proof were filled in the unpublished note \cite{BKfix}, which this paper supersedes. In this paper we give a new construction of $\B$ and hence a complete proof of Theorem~\ref{main exp theorem}.  The theorem suggests a stronger form of the quasiminimality conjecture which evidently implies Conjecture~\ref{weak qm conj}.
\begin{conjecture}[Strong quasiminimality conjecture]\label{strong qm conj}
$\Cexp$ is isomorphic to the unique model $\B$ of $\ECFskccp$ of cardinality continuum.
\end{conjecture}
The axioms in $\ECFskccp$ will be explained in Section~\ref{axiomatization section}, but briefly there are two algebraic axioms which are obviously true in $\Cexp$ and then three more axioms: Schanuel's conjecture, strong exponential-algebraic closedness, and the countable closure property.
Schanuel's conjecture is a conjecture of transcendental number theory which can be seen as saying that certain systems of exponential polynomial equations do not have solutions. Strong exponential-algebraic closedness roughly says that a system of equation has solutions (even generic over any given finite set) unless that would contradict Schanuel's conjecture. The countable closure property says roughly that such systems of equations which are \emph{balanced}, in the sense of having the same number of equations as variables, have only countably many solutions. Zilber proved the countable closure property for \Cexp, so we have the following reformulation.
\begin{theorem}\label{SC + SEAC}
Conjecture~\ref{strong qm conj} is true if and only if Schanuel's conjecture is true and $\Cexp$ is strongly exponentially-algebraically closed.
\end{theorem}

Theorems~\ref{main exp theorem} and~\ref{SC + SEAC} together imply that if $\Cexp$ satisfies Schanuel's conjecture and is strongly exponentially-algebraically closed then it is quasiminimal. Schanuel's conjecture is considered out of reach, since even the very simple consequence that the numbers $e$ and $\pi$ are algebraically independent is unknown. Proving strong exponential-algebraic closedness involves finding solutions of certain systems of equations and then showing they are generic, the latter step usually done using Schanuel's conjecture. A weaker condition is \emph{exponential-algebraic closedness} which requires the same systems of equations to have solutions, but says nothing about their genericity. We are able to remove the dependence on Schanuel's conjecture completely from Conjecture~\ref{weak qm conj}.
\begin{theorem}\label{eac implies qm}
If $\Cexp$ is exponentially-algebraically closed then it is quasiminimal.
\end{theorem}

\subsection{A more general construction: $\Gamma$-fields}

 Our construction is more general and we can use it to construct also a pseudo-analytic version of the Weierstrass $\wp$-functions, the exponential maps of simple abelian varieties, and more generally other pseudo-analytic subgroups of the product of two commutative algebraic groups. For example, we prove an analogous form of Theorems~\ref{main exp theorem} and~\ref{SC + SEAC} for $\wp$-functions. The list of axioms $\PCFskccp(E)$ and the other notions used in the statement of the theorem will be explained in section~\ref{pseudo-p section} of the paper.
\begin{theorem}\label{main p theorem}
Given an elliptic curve $E$ over a number field $K_0 \subs \C$, the list $\PCFskccp(E)$ of axioms is uncountably categorical and every model is quasiminimal. Furthermore, if $\wp$ is the Weierstrass function associated to $E(\C)$, so $\exp_E = [\wp:\wp':1] : \C \to E(\C)$ is the exponential map of $E(\C)$, then $\C_\wp \leteq \struct{\C;+,\cdot,\exp_E} \models \PCFskccp(E)$ if and only if the analogue of Schanuel's conjecture for $\wp$ holds and $\C_\wp$ is strongly $\wp$-algebraically closed.
\end{theorem}

In the most general form, we consider what we call $\Gamma$-fields, which are fields $F$ of characteristic zero equipped with a subgroup $\Gamma(F)$ of a product $G_1(F) \cross G_2(F)$ where $G_1$ and $G_2$ are commutative algebraic groups. The complete definition is given in section~\ref{Gamma field defn section}, where we also explain how the examples we consider fit into cases (EXP), generalizing the exponential and Weierstrass $\wp$-functions above, (COR), generalizing analytic correspondences between non-isogenous elliptic curves, and (DEQ), generalizing the solution sets of certain differential equations.

Hrushovski used \Fraisse's amalgamation method which produces countable structures. Zilber wanted uncountable structures so he instead framed his constructions in terms of existentially closed models within a certain category. He gave a framework of \emph{quasiminimal excellent classes} \cite{Zilber05qmec}, building on Shelah's notion of an excellent $\Loo$-sentence \cite{Sh87a}, to prove the uniqueness of the uncountable models. The second author showed \cite{OQMEC} that the quasiminimal excellence conditions can be checked just on the countable models, and with Hart, Hyttinen and Kes\"al\"a we proved in \cite{BHHKK14} that the most complicated of the conditions to check, excellence, follows from the other conditions. So in this paper we recast the construction in 4 stages.
\begin{enumerate}[1.]
\item We start with a suitable base $\Gamma$-field $\Fbase$, and describe a category $\Cc(\Fbase)$ of so-called \emph{strong extensions} of $\Fbase$.
\item We apply a suitable version of \emph{\Fraisse's amalgamation theorem} to the category to produce a countable model $M(\Fbase)$.
\item We check that $M(\Fbase)$ satisfies the conditions to be part of a \emph{quasiminimal class}, and deduce there is a unique model of cardinality continuum we denote by $\MM(\Fbase)$.
\item We give the \emph{axioms} $\GCFccp(\Fbase)$ describing the class.
\end{enumerate}

As a more general form of Theorem~\ref{main exp theorem} we prove:
\begin{theorem}\label{main Gamma theorem}
Given an essentially finitary $\Gamma$-field $\Fbase$ of type (EXP), (COR), or (DEQ), the list of axioms $\GCFccp(\Fbase)$ is uncountably categorical and every model is quasiminimal.
\end{theorem}

Our notion of $\Gamma$-fields is algebraic and not every example is related to an analytic prototype. However cases (EXP) and (COR) do have many analytic examples, given in Definitions~\ref{analytic EXP defn} and~\ref{analytic COR defn}. We call these \emph{analytic $\Gamma$-fields}. For these we are able to prove the countable closure property, extending Zilber's result for $\Cexp$ and the equivalent result in \cite{JKS16} for $\wp$-functions.
\begin{theorem}\label{CCP for analytic Gamma-fields}
Let $\C_\Gamma$ be an analytic $\Gamma$-field. Then $\C_\Gamma$ satisfies the countable closure property.
\end{theorem}

One of the key ideas of this paper is that the amalgamation construction is done over a base $\Gamma$-field $\Fbase$, and that everything is done relative to that base. Pushing this idea further, we can also work over a base which is closed with respect to the quasiminimal pregeometry on the model $\Fbase$. This involves modifying the amalgamation construction so we only consider extensions of $\Fbase$ in which $\Fbase$ remains closed with respect to the pregeometry. This idea comes from differential fields, where the base would be the field of constants and one often wants to consider differential field extensions with no new constants. One advantage of this approach for us is that Ax's versions of Schanuel's conjecture then apply to say that Schanuel's conjecture is true relative to the base in the analytic $\Gamma$-fields. 

In the paper \cite{ECFCIT} it was shown that, assuming the Conjecture on Intersections with Tori (CIT, also known as the multiplicative Zilber-Pink conjecture), any exponential field satisfying Schanuel's conjecture and exponential-algebraic closedness is actually strongly exponentially-algebraically closed.
In this paper we are able to adapt that idea to show unconditionally that the difference between $\Gamma$-closedness and strong $\Gamma$-closedness (the analogues of exponential-algebraic closedness and strong exponential-algebraic closedness) disappears if we consider the generic version, meaning relative to a closed base. Instead of the CIT we use a theorem which we call the \emph{horizontal semiabelian weak Zilber-Pink}, a theorem about intersections of families of algebraic varieties with cosets of algebraic subgroups of semiabelian varieties. 
 This method allows us to prove Theorem~\ref{eac implies qm} and a more general version:
\begin{theorem}\label{Gcl implies qm}
Let $\C_\Gamma$ be an analytic $\Gamma$-field. If $\C_\Gamma$ is $\Gamma$-closed then it is quasiminimal.
\end{theorem}

\subsection{An overview of the paper}

In section~\ref{alg background section} we explain our conventions on viewing algebraic varieties and their profinite covers in a model-theoretic way. We also explain the relationship between subgroups and endomorphisms of the commutative algebraic groups we study. 

In section~\ref{Gamma fields section} we define our $\Gamma$-fields and their finitely generated extensions. 
We prove that finitely generated extensions of suitable (so-called \emph{essentially finitary}) $\Gamma$-fields are determined by good bases, and that these good bases exist and are determined by finite data from a countable range of possibilities. This is the key step in proving the form of $\aleph_0$-stability which is essential for the existence of quasiminimal models. The main tool here is Kummer theory over torsion for abelian varieties.

In section~\ref{predim section} we introduce the predimension notion and use it to define which extensions of $\Gamma$-fields are strong. We also use it to define a pregeometry on $\Gamma$-fields. Then we show that there is a unique \emph{full-closure} of an essentially finitary $\Gamma$-field, and classify the strong finitely generated extensions of $\Gamma$-fields and of full $\Gamma$-fields. This completes stage 1 of the construction as described above.

Section~\ref{amalg section} covers stage 2 of the construction. We recall a category-theoretic version of \Fraisse's amalgamation theorem which is suitably general for us. Then, starting with a suitable base $\Gamma$-field $\Fbase$, we consider the category $\Cc(\Fbase)$ of strong extensions of $\Fbase$ and apply the amalgamation theorem to get a countable \Fraisse\ limit $M(\Fbase)$. We also consider a variant amalgamating only the $\Gamma$-algebraic extensions and another variant where we consider only extensions which are purely $\Gamma$-transcendental over $\Fbase$.

In section~\ref{categoricity section} we show that the  \Fraisse\ limit  models we have produced are quasiminimal pregeometry structures, and hence give rise to uncountably categorical classes. In this way we get the uncountable models, in particular the model $\MM(\Fbase)$ of cardinality continuum. This is stage 3.

In section~\ref{rotund section} we give a classification of the finitely generated strong extensions of $\Gamma$-fields, and in section~\ref{axiomatization section} we use that to give  axiomatizations of our models and prove Theorem~\ref{main Gamma theorem}. This completes stage 4.

In section~\ref{applications section} we consider specific instances of our $\Gamma$-fields including pseudo-exponentiation, pseudo Weierstrass $\wp$-functions, and others, and prove Theorem~\ref{main exp theorem} and half of Theorem~\ref{main p theorem}.

 In section~\ref{analytic comparison} we compare our models to the complex analytic prototypes. For Weierstrass $\wp$-function we relate the Schanuel property to the Andr\'e-Grothendieck conjecture on the periods of 1-motives, using work of Bertolin, finishing the proof of Theorem~\ref{main p theorem}. We briefly discuss the literature on steps towards proving the strong $\Gamma$-closedness and $\Gamma$-closedness properties for analytic $\Gamma$-fields. Then we prove Theorem~\ref{CCP for analytic Gamma-fields}.
 
In section~\ref{GGC section} we consider $\Gamma$-fields which may not be $\Gamma$-closed but are generically so. These are the $\Gamma$-fields produced by the variant construction in which the base $\Fbase$ remains closed with respect to the pregeometry. We state and prove the horizontal semiabelian weak Zilber-Pink, and then prove Theorems~\ref{eac implies qm} and~\ref{Gcl implies qm}.

\subsection*{Acknowledgements}

Some of this work was carried out in the Max Planck Institute for Mathematics, Bonn during the programme on Model Theory and Applications, Spring 2012. Further work was carried out at the Mathematical Sciences Research Institute, Berkeley, during the programme on Model Theory, Arithmetic Geometry and Number Theory, Spring 2014, which was supported by the NSF under Grant No.\ 0932078 000. Our thanks to both institutions for their support and hospitality.
We would like to thank Boris Zilber and Misha Gavrilovich for useful discussions. Particular thanks are due to Juan Diego Caycedo, whose suggestions on building a pseudo-$\wp$-function were part of the initial foundations of this project, and who further contributed to the project in its earlier stages.


\section{Algebraic background}\label{alg background section}

\subsection{Algebraic varieties and groups}

We will use the standard model-theoretic foundations for algebraic varieties and algebraic groups, as described by Pillay \cite{Pillay98}, roughly following Weil. In particular, we will work in the theory $\ACF_0$ with parameters for a field $\bk$. Any variety $V$ is considered as a definable set, and using elimination of imaginaries it is in definable bijection with a constructible subset of affine space. We always assume we have chosen such a bijection, although we will not mention it explicitly. Given any field extension $F$ of $\bk$, we write $V(F)$ for the points of $V$ all of whose coordinates lie in $F$. In this way, $V$ is a functor from the category of field extensions of its field of definition to the category of sets. Similarly, given any subset $A \subs V(F)$, we can form the subfield of $F$ which is generated by (the co-ordinates of) the points in $A$.

In the same way, a commutative algebraic group $G$, defined over $\bk$, is considered as a functor from the category of field extensions of $\bk$ to the category of abelian groups. If $G$ is an algebraic $\OO$-module, that is, the ring $\OO$ acts on $G$ via regular endomorphisms, defined over $\bk$, we can also consider it as a functor to the category of $\OO$-modules.

$\ga$ denotes the additive group, $\ga(F) = \tuple{F;+}$,
and $\gm$ denotes the multiplicative group, $\gm(F) = \tuple{F^\times;\cdot}$.

An algebraic group is \emph{connected} if it has no proper finite index algebraic subgroups.
The \emph{connected component} $G^o$ of an algebraic group $G$ is the largest connected algebraic subgroup.

We write $G[m]$ for the $m$-torsion subgroup of an algebraic group $G$.

If $G$ is a commutative algebraic group over a field of characteristic 0,
we write $LG$ for the (commutative) Lie algebra of $G$,
the tangent space at the identity considered as an algebraic group. So $LG \iso \ga^{\dim(G)}$.
If $\theta : G \maps G'$ is an algebraic group homomorphism,
then $L\theta : LG \maps LG'$ is the derivative at the identity.
For algebraic groups over $\C$, these definitions agree with the usual definitions for complex Lie groups.

\subsection{Subgroups and endomorphisms}

Any connected algebraic subgroup $H$ of a power $\gm^n$ of the multiplicative group can be defined by a system of monomial equations,
$H = \ker(M)$ for some integer square matrix $M \in \Mat_n(\Z)$ acting multiplicatively.
Then $LH \leq L\gm^n = \ga^n$ is the kernel of the same matrix acting additively.

As we observe in the following lemma,
the picture is almost the same when we replace $\gm$ with an abelian variety $G$ and $\Z$ with its endomorphism ring $\End(G)$:
up to finite index, subgroups are defined by $\OO$-linear equations,
namely those which define the corresponding Lie subalgebra.
With a few self-contained exceptions, this lemma is essentially all we will use of the theory of abelian varieties.

\begin{lemma} \label{subgroupsEndomorphisms}
  Suppose $G$ is $\gm$ or an abelian variety over a field of characteristic 0,
  and $\OO = \End(G)$ is its endomorphism ring. Then
  \begin{enumerate}[(i)]
    \item any connected algebraic subgroup $H \leq G^n$
      is the connected component of the kernel of an endomorphism
      $\eta \in \End(G^n) \isom \Mat_n(\OO)$,
        \[ H = \ker(\eta)^o; \]
    \item $LH \leq LG^n$ is then the kernel of $L\eta \in \End(LG^n)$.
  \end{enumerate}
\end{lemma}
\begin{proof}
 \begin{enumerate}[(i)]
  \item By Poincar\'e's complete reducibility theorem \cite[p173]{MumfordAbVars},
there exists an algebraic subgroup $H'$ such that the summation map $\Sigma:H\times H' \maps G^n$ is an isogeny, that is, a surjective homomorphism with finite kernel. Let $m$ be the exponent of the kernel of $\Sigma$.
Then $\theta(\Sigma(h,h')) := (mh,mh')$ defines an isogeny
$\theta : G^n \to H\times H'$.
      Let $\pi_2 : H\times H' \maps H'$ be the projection.
      Then $(\pi_2\circ\theta\circ\Sigma)(h,h')=mh'$,
      so $\ker(\pi_2\circ\theta)^o=(\Sigma(H\times H'[m]))^o=(H+H'[m])^o=H$.
    \item
      Since $H \subs \ker(\eta)$, the derivative $L\eta$ of $\eta$ at 0 vanishes on $LH$, so $LH \le \ker(L\eta)$. Also $\ker(L \eta) = L\ker(\eta)$, so we have
      \begin{eqnarray*}
      \dim(\ker(L \eta)) &=& \dim (L \ker(\eta)) \\
      & = & \dim (\ker(\eta)) \quad \text{since $0$ is a smooth point}\\
      & = & \dim H \quad \text{since $H$ has finite index in $\ker(\eta)$} \\
      & = & \dim LH \quad \text{again since $0$ is a smooth point.}
      \end{eqnarray*}
      So $LH$ has finite index in $\ker(L\eta)$, but $\ker(L\eta) \le LG^n$ which is torsion-free, so $\ker(L\eta)$ is connected, so $LH = \ker(L\eta)$.
  \end{enumerate}
\end{proof}

\subsection{Division points and the profinite cover}

\begin{defn}
Let $G$ be a commutative group and let $a \in G$.  A \emph{division point} of $a$ in $G$ is any $b \in G$ such that, for some $m \in \N^+$, $m b = a$.

A \emph{division sequence} for $a$ in $G$ is a sequence $(a_m)_{m \in \N^+}$ in $G$ such that $a_1 = a$ and for all $m,n \in \N^+$ we have $n a_{nm} = a_m$.
\end{defn}

If $(a_m)_{m \in \N^+}$ is a division sequence for $a$ in $G$ we can define a group homomorphism $\theta: \Q \to G$ by $\theta(\frac rm) =r a_m$ for $r \in \Z$ and $m \in \N^+$. This gives a bijective correspondence between division sequences for $a$ in $G$ and group homomorphisms $\theta: \Q \to G$ such that $\theta(1) = a$. 

\begin{defn}
The \emph{profinite cover} $\hat G$ of a commutative group $G$ is the group of all  homomorphisms $\Q \to G$, with the group structure defined pointwise in $G$. We write $\rho_G : \hat G \to G$ for the evaluation homomorphism given by $\rho_G(\theta) = \theta(1)$.
\end{defn}
Thus the set of division sequences for $a$ in $G$ is in bijective correspondence with $\rho_G^{-1}(a)$, and we think of elements of $\hat G$ both as homomorphisms from $\Q$ and as division sequences.

The group $\hat G$ itself is divisible and torsion-free. The image of $\rho_G$ is the subgroup of divisible points of $G$, and $\rho_G$ is injective if and only if $G$ is torsion-free. In general, $\ker(\rho_G)$ is a profinite group built from the torsion of $G$ (in fact it is, the product over primes $l$ of the $l$-adic Tate modules of $G$).

For an element $a \in G$, we will often use the notation $\hat a$ for a chosen element of $\hat G$ such that $\rho_G(\hat a) = a$. Of course $\hat a$ is determined by $a$ only when $\rho_G$ is injective, that is, when $G$ is torsion-free.

If $f : G \to H$ is a group homomorphism, we can lift it to a homomorphism $\hat f : \hat G \to \hat H$ defined by $\theta \mapsto f \circ \theta$.  In particular, if $G \subs H$ is a subgroup then $\hat G$ is naturally a subgroup of $\hat H$. (In category-theoretic language, $\left.\hat{\ }\right.$ is a covariant representable functor and in fact $\rho_G : \hat G \to G$ is the universal arrow from the category of divisible, torsion-free abelian groups into $G$.)

When $G$ is a commutative algebraic group we think of $\hat G$ also as a functor, so we write $\hat G(F)$ rather than $\hat{G(F)}$ for the group of division sequences of the group $G(F)$.

Model-theoretically we think of $\hat G$ as the set of division sequences from $G$, which is a set of infinite tuples satisfying the divisibility conditions. It can be seen as an inverse limit of definable sets, sometimes called a pro-definable set \cite{Kamensky07}.

\begin{remark}
Suppose that $G$ is a Lie group, for example the complex points of a complex algebraic group, and $\exp: LG \to G$ is the exponential map. For $a \in LG$, the sequence $(\exp(a/m))_{m \in \N^+}$ is a division sequence in $G$. In fact, the division sequences which arise this way are precisely those which converge topologically to the identity of $G$ \cite[Remark~2.20]{BHP14}.
\end{remark}

\section{$\Gamma$-fields}\label{Gamma fields section}

\subsection{$\Gamma$-fields}\label{Gamma field defn section}
In this section we describe the analytic examples we are studying and give the definition of a $\Gamma$-field which is intended to capture and generalize the model-theoretic algebra of the examples.

\begin{defn}[Analytic $\Gamma$-fields of type (EXP)]\label{analytic EXP defn}
The graph of the usual complex exponential function is a subgroup of $\ga(\C)\cross \gm(\C)$. Similarly, if $A(\C)$ is a complex abelian variety (or more generally any commutative complex algebraic group) of dimension $d$, then the graph $\Gamma$ of the exponential map of $A$ is a subgroup of $LA(\C) \cross A(\C)$. Here $LA(\C)$ is the Lie algebra of $A$ and we can identify it with the group $\ga^d(\C)$. 
In this paper we will only consider the cases when $A$ is $\gm$ or $A$ is a simple abelian variety of dimension $d$. We combine these by saying $A$ is a \emph{simple semiabelian variety}.

We write $\OO$ for the ring $\End(A)$ of algebraic endomorphisms of $A$. In many cases $\OO = \Z$, but sometimes, for example if $A$ is an elliptic curve with \emph{complex multiplication}, then \OO\ properly extends $\Z$. Any $\eta \in \OO$ acts on $LA$ as the derivative $d\eta$, a linear map. Thus $\OO$ naturally acts on $LA(\C)$ as a subring of $\GL_d(\C)$, and $\Gamma$ is an $\OO$-submodule of $LA(\C) \cross A(\C)$.

In the case where $A$ is an elliptic curve $E$, embedded in projective space $\mathbb P^2$ in the usual way via its Weierstrass equation, then the exponential map of $E(\C)$ is written in homogeneous coordinates as $z \mapsto (\wp(z) : \wp'(z) : 1)$, where $\wp$ is the Weierstrass $\wp$-function associated with $E$.

We call all of these examples \emph{analytic $\Gamma$-fields of type (EXP)}.
\end{defn}

\begin{defn}[Analytic $\Gamma$-fields of type (COR)]\label{analytic COR defn}
The exponential map of a complex elliptic curve factors through $\gm(\C)$, giving an analytic map $\theta:\gm(\C) \to E(\C)$. More generally there are analytic correspondences between semiabelian varieties. We take $G_1$ and $G_2$ both to be simple complex semiabelian varieties of the same dimension $d$, and assume $G_1$ and $G_2$ are not isogenous. Suppose $\End(G_1)$ and $\End(G_2)$ are both isomorphic to a ring $\OO$, and furthermore there is a $\C$-vector space isomorphism $\psi : LG_1 \maps LG_2$ which respects the actions of $\OO$. We choose such a $\psi$ and take $\Gamma$ to be the image
of the graph of $\psi$ under $\exp_{G_1\cross G_2} : LG_1(\C) \times LG_2(\C) \maps G_1(\C) \times G_2(\C)$. Then $\Gamma$ is an $\OO$-submodule of $G_1(\C) \times G_2(\C)$, and a complex Lie-subgroup. The graph of the map $\theta$ is an example of such a $\Gamma$, but more generally $\Gamma$ will not be the graph of a function.

We call these examples \emph{analytic $\Gamma$-fields of type (COR)}. By an \emph{analytic $\Gamma$-field} we mean one of type (EXP) or type (COR).
\end{defn}

\begin{defn}[Differential equation examples]
If $f(t)$ is a holomorphic function in a neighbourhood of $0 \in \C$, then the pair $(x,y) = (f(t), \exp(f(t)))$ satisfies the differential equation $Dx = \frac{Dy}y$, where $D = \frac{\mathrm d}{\mathrm d t}$. We can consider the set $\Gamma$ of solutions of the differential equation not just in a field of functions but in a differentially closed field $F$. Then $\Gamma$ is a subgroup of $\ga(F) \cross \gm(F)$. The paper \cite{TEDESV} studies this situation for the differential equations satisfied by the exponential maps of semiabelian varieties $S$. While these $S$ do not have to be simple, they do have to be defined over the constant field $C$ of $F$. In these cases the group $\Gamma$ is quite closely related to the graph of the exponential map and can be analysed via a similar amalgamation construction.
\end{defn}

We capture all of these three types of examples in the notion of a $\Gamma$-field. We next give the assumptions we will use on the algebraic groups, and then define $\Gamma$-fields. The assumptions we make are not the most general possible, but they are what we will use throughout this paper.

\begin{defn}[Conventions for $\bk$,  $G_2$, $\OO$, and $\kO$]
We take $\bk$ to be a countable field of characteristic 0, which must be a number field except in case (DEQ) below. Let $G_2$ be a simple semiabelian variety defined over $\bk$.

We write $\OO$ for the ring $\End(G_2)$ of algebraic (that is, regular) group endomorphisms of $G_2$ and assume that they are also all defined over $\bk$. Let $\kO$ denote the ring $\Q \otimes_{\Z} \OO$. 
\end{defn}

\begin{remarks}
The ring $\OO$ has no zero divisors because $G_2$ is simple. So \OO\ embeds in $\kO$. If $\OO = \Z$ then $\kO$ is just $\Q$. Every non-zero algebraic group endomorphism of a simple abelian variety is an isogeny, so becomes invertible in $\kO$. Hence $\kO$ is a division ring, and the $\OO$-torsion of any $\OO$-module is exactly the $\Z$-torsion.
\end{remarks}

\begin{defn}[Conventions for $G_1$, $G$, and the torsion]
We consider two cases for the choice of $G_1$, corresponding to the above analytic examples.
\begin{description}
\item[Case (EXP)] We take $G_1=\Ga^d$, where $d=\dim G_2$. We identify $G_1$ with the Lie algebra $LG_2$, that is, the tangent space at the identity of $G_2$. As in the analytic case, this identification makes $G_1$ into an algebraic $\OO$-module, that is, an $\OO$-module in which every element of $\OO$ acts as a regular map.
  
\item[Case (COR)]  $G_1$ is also a simple semiabelian variety defined over $\bk$, and with all its algebraic endomorphisms defined over $\bk$. We assume $G_1$ is not isogenous to $G_2$, but $\End(G_1) \iso \OO$ and we choose an isomorphism, so $G_1 \cross G_2$ becomes an algebraic $\OO$-module over $\bk$.
\end{description}
 
Let $G = G_1 \cross G_2$, and write $\pi_i : G \to G_i$ for the projection maps of the product, for $i=1,2$. We will write the groups $G_1$, $G_2$ and $G$ additively.

For $i=1,2$ the torsion of $G_i$ is contained in $G_i(\bk^\alg)$, and hence is bounded. We write $\Tor_i$ for the torsion of $G_i(F)$ for any $F$ such that $G(F)$ contains $\Tor(G(\bk^\alg))$. The torsion of $G(F)$ is written $\Tor(G)$. It is equal to $(\Tor_1\cross\Tor_2) \cap G(F)$.
\end{defn}

\begin{remarks}
Note that for any algebraically closed field $F$ extending $\bk$ the groups $G_i(F)$ and $G(F)$ are divisible $\OO$-modules. Furthermore, $G(F)/\Tor(G)$ is divisible and torsion-free, and hence is a $\kO$-vector space.
\end{remarks}

\begin{defn}[$\Gamma$-fields]
A \emph{$\Gamma$-field} (with respect to the $\OO$-module $G$) is a field extension $A$ of $\bk$ equipped with a divisible $\OO$-submodule $\Gamma(A)$ of $G(A)$ such that
\begin{enumerate}
\item $A$ is generated as a field by $\Gamma(A)$.
\item The projection $\pi_i(\Gamma(A))$ in $G_i(A)$ contains $\Tor_i$ for $i=1,2$.
\end{enumerate}
We  write $\Gamma_i(A)$ for the projections $\pi_i(\Gamma(A))$. The $\Gamma$-field $A$ is \emph{full} if, in addition, $A$ is algebraically closed and the projections $\Gamma_i(A)$ are equal to $G_i(A)$.

The \emph{kernels} of a $\Gamma$-field $A$ are defined to be
\[
\ker_1(A) := \class{ x \in G_1(A)}{(x,0) \in \Gamma(A)},\]
\[\ker_2(A) := \class{ y \in G_2(A)}{(0,y) \in \Gamma(A) } .
\]
\end{defn}

When $A$ is full and $\ker_2(A)$ is trivial, $\Gamma(A)$ will be the graph of a surjective $\OO$-module homomorphism from $G_1(F)$ to $G_2(F)$ with kernel $\ker_1(F)$ as in the analytic examples of type (EXP). However the case (EXP) for our $\Gamma$-fields is more general.
 
 \medskip
 
The most difficult part of this paper uses the Kummer theory of semiabelian varieties over number fields. This is not needed for the differential equations examples, or more generally in the following variant.
\begin{defn}[Case \DEQ] A $\Gamma$-field in case \DEQ\ is the same as above except that we require the full torsion group $\Tor(G)$ to be contained in $\Gamma(A)$, and we relax the assumption that $K_0$ is a number field so it can be any countable field of characteristic 0. 
 \end{defn}

\begin{defn}[Extensions of $\Gamma$-fields]
An \emph{extension} of a $\Gamma$-field $A$ is a $\Gamma$-field $B$ together with an inclusion of fields $A \subs B$ over $\bk$ such that $\Gamma(A) \subs \Gamma(B)$. We also say that $A$ is a \emph{$\Gamma$-subfield} of $B$.
We say an extension $A \subs B$ \emph{preserves the kernels}, and that $A$ and $B$ \emph{have the same kernels}, if $\ker_i(A) = \ker_i(B)$, for $i=1,2$.
\end{defn}

In this paper, we will only consider extensions of $\Gamma$-fields which preserve the kernels.

\begin{remarks}[$\Gamma$-fields as model-theoretic structures] \label{Gamma extension remarks} \
\begin{enumerate}
\item Model-theoretically, we consider a $\Gamma$-field as a structure in the 1-sorted first-order language $L_\Gamma = \langle +,\cdot,-,\Gamma, (c_a)_{a\in \bk} \rangle$, where $\Gamma$ is a relation symbol of appropriate arity to denote a subset of the group $G$, and we have parameters for the field $\bk$. Later we will also be adding parameters for a base $\Gamma$-field $\Fbase$.

\item However, our notion of $\Gamma$-field extension corresponds to an injective $L_\Gamma$-homomorphism, not necessarily an $L_\Gamma$-embedding. Specifically, it is not necessary in an extension $A \into B$ of $\Gamma$-fields that $\Gamma(B) \cap G(A) = \Gamma(A)$, although in most cases we will consider later that will be true. 

\item Although we use the 1-sorted language with the sort being that of the underlying field, we will also refer to elements of $\Gamma$ as being from the sort $\Gamma$, rather than from the definable set $\Gamma$. Model theorists used to working with $L^{eq}$ will see there is no important difference.

\item By definition, the $\Gamma$-field $A$ is determined by the submodule $\Gamma(A)$ of $G(F)$. Furthermore, an extension $A \into B$ is determined by the inclusion of submodules $\Gamma(A) \into \Gamma(B)$. Thus, if $F$ is a monster model of $\ACF_0$, the category of $\Gamma$-fields is equivalent to the category of divisible $\OO$-submodules of $G(F)$ (whose projections contain $\Tor_1$ and $\Tor_2$), with embeddings, in a first-order language with relation symbols for all of the Zariski-closed subsets of $G(F)$ which are defined over $\bk$. This is more-or-less Zilber's setting in \cite{Zilber05peACF0}.
\end{enumerate}
\end{remarks}

\begin{remark}
 One might also consider the case that $G_1$ and $G_2$ are equal (or,
   which comes to essentially the same thing, isogenous). $\Gamma$ can
   then be considered as the graph of a new (quasi)endomorphism of $G_1$. The
   situation is complicated by the need to consider the extension
   of the algebraic endomorphism ring generated by $\Gamma$. Analytic
   examples include raising to a complex power on $\Gm$, which is
   analysed with a different setup in \cite{Zilber03powers} and \cite{Zilber11tes}. 
   
   In an earlier draft of this paper we tried to incorporate this into our setup, and in fact produced an example where $\Gamma$ was the graph of a multivalued endomorphism $\theta$ on $\gm$, lifting to a generic action of the ring $\Q[\theta,\theta^{-1}]$ on the profinite cover $\hat\gm$. However this is subtly different from giving an action of the field $\Q(\theta)$ which is what occurs for complex powers. 

While we expect that such $\Gamma$ can be treated along the lines of this paper, much as we expect that the simplicity assumption on the semiabelian variety could be relaxed, these elaborations are left to future work.
\end{remark}


\subsection{Finitely generated extensions}
\begin{definition}
Let $B$ be a $\Gamma$-field, and let $\class{A_j}{j \in J}$ be a set of $\Gamma$-subfields of $B$, each with the same kernels as $B$. We define $\bigwedge_{j\in J} A_j$ to be the $\Gamma$-subfield $A$ of $B$ such that $\Gamma(A) = \bigcap_{j \in J}\Gamma(A_j)$.
\end{definition}
\begin{lemma}\label{intersection is Gamma-subfield}
$\bigwedge_{j\in J} A_j$ is a $\Gamma$-subfield of $B$.
\end{lemma}
The proof is straightforward, but we give the details because they show exactly where all the hypotheses of the definitions are used.
\begin{proof}
Let $A = \bigwedge_{j\in J} A_j$. Since $\Gamma(A)$ is defined as the intersection of a set of $\OO$-submodules of $\Gamma(B)$, it is also an $\OO$-submodule of $\Gamma(B)$. $A$ is defined as the subfield of $B$ generated by the coordinates of the points in $\Gamma(A)$, so $\Gamma(A)$ is an $\OO$-submodule of $G(A)$.

If $a \in \ker_1(B)$, then $(a,0) \in \Gamma(A_j)$ for all $j\in J$ because $\ker_1(A_j) = \ker_1(B)$, so $(a,0) \in \Gamma(A)$. So $\ker_1(A) = \ker_1(B)$ and similarly $\ker_2(A) = \ker_2(B)$.

If $a \in \Tor_1 = \Tor_1(B)$ then there is $b \in G_2(B)$ such that $(a,b) \in \Gamma(B)$. Furthermore for any $b' \in G_2(B)$ we have $(a,b') \in \Gamma(B)$ if and only if $b'-b \in \ker_2(B)$. For each $j \in J$, $A_j$ is a $\Gamma$-subfield of $B$, so $\Gamma_1(A_j)$ contains $\Tor_1(B)$, so there is $b'$ such that $(a,b') \in \Gamma(A_j)$. But $\ker_2(A_j) = \ker_2(B)$ by assumption, so we have $(a,b) \in \Gamma(A_j)$, and since this holds for all $j$ we have $(a,b) \in \Gamma(A)$. Thus $\Gamma_1(A)$ contains $\Tor_1(B)$, and similarly $\Gamma_2(A)$ contains $\Tor_2(B)$.

In particular, $\Gamma(A)$ contains all the torsion from $\Gamma(B)$, so since it is the intersection of divisible $\OO$-submodules, it is itself divisible as an $\OO$-submodule of $G(A)$. Hence $A$ is a $\Gamma$-subfield of $B$.
\end{proof}

\begin{definition}
Let $B$ be a $\Gamma$-field and $X \subs \Gamma(B)$ a subset. We say that 
\[A = \bigwedge \class{A' \text{, a $\Gamma$-subfield of $B$ with the same kernels as $B$}}{X \subs \Gamma(A')}\]
 is the $\Gamma$-subfield \emph{generated by $X$}, and write it as $\gen{X}_B$ or, more usually with the $B$ suppressed, as $\gen{X}$.
 
 We say $A$ is a \emph{finitely generated} $\Gamma$-field if $\Gamma(A)$ is of finite rank as an $\OO$-module, or equivalently as a $\Z$-module. Equivalently, $A$ is generated by a finite subset and $\ker_1(A)$ and $\ker_2(A)$ are of finite rank.

Note that a finitely generated $\Gamma$-field will not usually be finitely generated as a field, because we insist that $\Gamma(A)$ is a divisible $\OO$-submodule.

If $Y$ is a subset of $\Gamma(A)$, we say that $A$ is \emph{finitely generated over $Y$} if there is a finite subset $X$ of $\Gamma(A)$ such that $A$ is the $\Gamma$-subfield of itself generated by $X \cup Y$. In particular, for $Y$ a $\Gamma$-subfield of $A$, we have the notion of a finitely-generated extension of $\Gamma$-fields. It is easy to see that an extension $A \into B$ of $\Gamma$-fields is finitely generated if and only if $\ldim_\kO(\Gamma(B)/\Gamma(A))$ is finite.
\end{definition}

\begin{definition}
The intersection of full $\Gamma$-subfields of $B$ (with the same kernels as $B$) is again a full $\Gamma$-subfield. Thus we can define a full $\Gamma$-field $A$ to be \emph{finitely generated as a full $\Gamma$-field} if there is a finite subset $X$ of $A$ such that 
\[A = \bigwedge \class{A' \text{, a full $\Gamma$-subfield of $A$ with the same kernels as $A$}}{X \subs \Gamma(A')}\]
Likewise there is the notion of being \emph{finitely generated as a full $\Gamma$-field extension}.
\end{definition}
Except in trivial cases, a finitely generated full $\Gamma$-field will not be finitely generated as a $\Gamma$-field, and a finitely generated full $\Gamma$-field extension will not be finitely generated as a $\Gamma$-field extension.

\begin{defn}
Recall that an $\OO$-submodule $H$ of $G$ is \emph{pure} in $G$ if
whenever $x\in G$ and $nx \in H$ for some $n\in\N^+$, then $x \in
H$.
\end{defn}

\begin{lemma}
If $A$ is the $\Gamma$-subfield of $B$ generated by $X$, then
   $\Gamma(A)$ is the pure $\OO$-submodule of $\Gamma(B)$
   generated by $X \cup \pi_1^{-1}(\Tor_1) \cup \pi_2^{-1}(\Tor_2)$.
\end{lemma}
\begin{proof}
   This pure $\OO$-submodule together with the field it generates is a
   $\Gamma$-subfield of $B$ with the same kernels as $B$, so it
   suffices to see that it is contained in $\Gamma(A_j)$ for any $A_j$
   in the definition.

   $\Gamma(A_j)$ contains $X$ by definition, and since
   $\pi_i(\Gamma(A_j)) = \Tor_i$ and $A_j$ has the same kernels as $B$,
   it also contains $\pi_i^{-1}(\Tor_i)$. Hence it also contains
   $\Tor(G) \cap \Gamma(B)$. Since it is divisible, it follows that it
   is pure in $\Gamma(B)$.
\end{proof}

\subsection{Good bases}

Let $A$ be a $\Gamma$-field, and $B$ a finitely generated $\Gamma$-field extension of $A$. So the linear dimension $\ldim_\ek(\Gamma(B)/\Gamma(A))$ is finite. Thus we can find a \emph{basis} for the extension, by which we mean a tuple $b = (b_1,\ldots,b_n) \in \Gamma(B)^n$ of minimal length $n$ such that $b \cup \Gamma(A)$ generates $\Gamma(B)$, or equivalently such that $b_1+\Gamma(A),\ldots,b_n + \Gamma(A)$ is a basis for the quotient $\kO$-vector space $\Gamma(B) / \Gamma(A)$.

 We consider the locus $\loc(b/A)$ of $b$, that is, the smallest Zariski-closed subset of $G$, defined over $A$ and containing $b$.

\begin{definition}
A basis $b \in \Gamma(B)^n$ for a finitely generated extension $A \into B$ of $\Gamma$-fields is \emph{good} if the isomorphism type of the extension is determined up to isomorphism by the locus $\loc(b/A)$. That is, whenever $B'$ is another extension of $A$ which is generated by a basis $b'$ such that $\loc(b'/A) = \loc(b/A)$ then there is an isomorphism of $\Gamma$-fields $B \iso B'$ fixing $A$ pointwise, which takes $b$ to $b'$.
\end{definition}

\begin{prop}\label{good bases (D)}
Suppose we are in case \DE, that is $\Tor(G) \subs \Gamma$. Let $A \into B$ be a finitely generated extension of $\Gamma$-fields. Then every basis of the extension is good.
\end{prop}
\begin{proof}
Suppose $b$ is a basis of $B$ over $A$, and we have another extension $B'$ of $A$ with basis $b'$ such that $\loc(b'/A) = \loc(b/A)$. There is a (not necessarily unique) field isomorphism $\theta: B \iso B'$ over $A$ which takes $b$ to $b'$. Now for an element $c \in G(B)$ we have $c \in \Gamma(B)$ if and only if there is $m \in \N$ such that $mc$ is in the $\OO$-linear span of $\Gamma(A)$ and $b$, because $\Gamma(B)$ is divisible and contains all the torsion of $G$. It follows that $c \in \Gamma(B)$ if and only if $\theta(c) \in \Gamma(B')$, so $\theta$ is an isomorphism of $\Gamma$-field extensions. So $b$ is a good basis.
\end{proof}
In the proof it is critical that $\Tor(G) \subs \Gamma$ since otherwise some division points of the basis will be in $\Gamma$ but others will not. In general we can specify an extension $B$ of $A$ by specifying a choice of division sequence $\hat b$ for a basis $b$, such that $\hat b \in \hat\Gamma(B)$.

\begin{defn}
A $\Gamma$-field is \emph{essentially finitary} if it is finitely generated or if it is a finitely generated extension of a countable full $\Gamma$-field.
\end{defn}

\begin{prop}[Existence of good bases]\label{good bases exist} \ \\
  Let $A$ be an essentially finitary $\Gamma$-field, and let $B$ be a finitely generated $\Gamma$-field extension of $A$ (with the same kernels as $A$). Let $b$ be a basis for the extension. Then there is $m \in \N^+$ such that any $m^\mathrm{th}$ division point of $b$ in $\Gamma(B)$ is a good basis. Furthermore in case \DE\ we may take $m=1$, so every basis is good, and we may even remove the assumption that $A$ is essentially finitary.
\end{prop}

The bulk of the proof is contained in the following Kummer-theoretic results.

\begin{defn}
For a commutative algebraic group $H$ we write $\Tate(H)$ for the kernel of the map $\rho_H : \hat H \to H$. So $\Tate(H)$ is the group of division sequences of the identity of $H$ (which is the product over primes $l$ of the $l$-adic Tate modules $T_l(H)$ of $H$, hence the notation).
\end{defn}

\begin{proposition}\label{prop:kummer}

  Let $H = A \times \G_m^r$ be the product of an abelian variety and an algebraic torus.

  Suppose that $A$ is defined over a number field $\bk$, and moreover that
  every endomorphism of $A$ is also defined over $\bk$.

  Let $D$ be either $\Tor(H)$ or $H(L)$ for an algebraically
  closed field extension $L$ of $\bk$ and let $K$ be a finitely generated field extension of $\bk(D)$.

  Let $a \in H(K)$ and suppose that $a$ is free in $H$ over $D$, that is, in no coset $H'+\gamma$ for a proper algebraic subgroup $H'$ of $H$
  and $\gamma\in D$.

Let $\hat a = (a_m)_{m \in \N^+} $ be a division sequence for $a$ in $\hat H(K^\alg)$ and consider the Kummer map $\xi_a : \Gal(K^\alg/K) \maps \Tate(H)$ given by
    \[ \xi_a(\sigma) = (\sigma (a_m) - a_m)_{m \in \N^+} .\]

  Then $\xi_a$ does not depend on the choice of division sequence $\hat a$, so is well-defined, and the image of $\xi_a$ is of finite index in $\Tate(H)$.
\end{proposition}
\begin{remark}
  For the groups $\Tate(H)$ which occur in this theorem, the finite index subgroups are precisely those which are open in the profinite topology, so the conclusion of the proposition is that $\xi_a(\Gal(K^\alg/K))$ is open in $\Tate(H)$. 
\end{remark}
\begin{proof}[Proof of Proposition~\ref{prop:kummer}]
It is straightforward that $\xi_a$ is well-defined.
\medskip

First suppose $D=\Tor(H)$ and $K$ is a {\em finite} extension of
      $\bk(D)$ - so by increasing $\bk$, we may assume $K=\bk(D)$.

      The result then follows from Kummer theory for abelian varieties. For the case $H=A$, we refer to \cite[Theorem~5.2]{Bertrand11}, and for the generalization to $H = A \cross \gm^n$ we refer to \cite[Proposition~A.9]{BHP14}.
      
\medskip
 Suppose now that $D=H(L)$ where $L$ is an algebraically closed field. In this case, the result has a Galois-theoretic proof given as \cite[Section~3,~Claim~2]{BGH14}. In the case that
    $K=\bk(D,a)$, the result follows directly from that claim; in general, it
    follows on noting that
	\[ \xi_a(\Gal(K^\alg/K)) \isom \Gal(K(\hat{a}) / K) \isom
	\Gal(\bk(D,\hat{a}) / K \cap \bk(D,\hat{a}))
	,\]
    and $K \cap \bk(D,\hat{a})$ is a finite extension of $\bk(D,a)$.

    See also references in the introduction of \cite{BGH14} for alternative
    proofs, and \cite[Theorem~5.3]{Bertrand11} for an analytic proof.

\medskip
 Finally, suppose $D=\Tor(H)$ and $K$ is a finitely generated extension
    of $\bk(D)$.

    The result in this case follows from the first two cases. This can be seen
    model-theoretically in the context of \cite{BHP14} as a matter of
    transitivity of atomicity, but we give here a direct argument.

    Say $B$ is the minimal algebraic subgroup of $H$ such
    that, writing $\theta: H \to H/B$ for the quotient map, we have $\theta(a)\in (H/B)(\Qalg)$. Let $K' = K \cap \Qalg$, so
    $K$ is a regular extension of $K'$ and $K'$ is a finite extension of $\bk(D)$.
    Consider the following diagram
   \begin{diagram}[height=2em]
	1 & & 0 \\
	\dTo & & \dTo\\
	\Gal(K^\alg/\Qalg(K)) & \rTo^{\xi_a} & \Tate(B) \\
	\dTo & & \dTo\\
	\Gal(K^\alg/K) & \rTo^{\xi_a} &\Tate(H) \\
	\dTo & & \dTo\\
	\Gal(\Qalg/K') & \rTo^{\xi_{\theta(a)}} &\Tate(H/B)  \\
	\dTo & & \dTo\\
	1 & &0 
	\end{diagram}
    where the middle horizontal map is the Kummer map for $a$, the top map is
    its restriction, and the bottom map is the Kummer map for $\theta(a)$.
    The vertical sequences are exact.

    Say $a\in a_B + H(\Qalg)$, where $a_B\in B$. By minimality of $B$, we
    have that $a_B$ is free in $B$ over $B(\Qalg)$. Now the top
    map agrees with the Kummer map $\xi_{a_B}$ in $B$, and so by the second
    case above, the map has finite index image.

    Now since $a$ is free in $H$ over $\Tor(H)$, we have that
    $\theta(a)$ is free in $H/B$ over $\Tor(H/B)$,
    so by the first case applied to $H/B$,
    the bottom map in the above diagram also has finite index image.

    It follows that the central map has finite index image, as required.
  \end{proof}

Now we prove that good bases exist.
\begin{proof}[Proof of Proposition~\ref{good bases exist}]

Let $\hat b \in \hat\Gamma(B)^n$ be a division sequence of the basis $b$ and write $\hat b = (b_m)_{m \in \N^+}$. Then $\Gamma(B)$ is precisely the $\OO$-linear span of $\Gamma(A)$ and the $b_m$, so to specify $B$ up to isomorphism it is enough to specify the $\ACF$-type of $\hat b$ over $A$. 
$A$ is an essentially finitary $\Gamma$-field, so it is either finitely generated or a finitely generated extension of a countable full $\Gamma$-field $A_0$. In the former case, let $D = \Tor(G)$ and and in the latter case let $D = G(A_0)$. For $i=1,2$, write $b_i = \pi_i(b)$ and $D_i = \pi_i(D)$, and let $a_i$ be a $\kO$-basis for $\pi_i(\Gamma(A))$ over $D_i$.

We consider the different cases in turn.

 \begin{enumerate}[Case (EXP)]
    \item[Case (EXP)]
      Since the extensions are kernel-preserving, $(a_2,b_2)$ is $\ek$-linearly independent over $D_2$,
      and so is free in $G_2^{n+k}$ over $D_2$.

      So, by Proposition~\ref{prop:kummer}, $\xi_{a_2,b_2}(\Gal(K_0(D,a,b)^\alg/K_0(D,a,b)))$ has finite index in $\Tate(G_2^{n+k})$. In particular, its intersection with $0 \times \Tate(G_2^n)$ is of finite index.

      Since $A$ is generated as a field by $K_0(D,a)$ and the division
      points of $a_2$, it follows that
	$\Xi := \xi_{b_2}(\Gal(A(b)^\alg/ A(b)))$ has finite index in $\Tate(G_2^n)$.
      So if $m$ is the exponent of the finite quotient $\Tate(G_2^n) / \Xi$,
      then $m\Tate(G_2^n)$ is a subgroup of $\Xi$.

      Hence, if $b'$ is an $m^\mathrm{th}$ division point of $b$ we have $\xi_{b'_2}(\Gal(A(b')^\alg/A(b'))) = \Tate(G_2^n)$.
      So all division sequences of $b'$ have the same $\ACF$-type over $A(b')$, and hence $b'$ is a good basis for $B$ over $A$.

    \item[Case (COR)]
 Again, since the extensions are kernel-preserving, $(a_i,b_i)$ is free over $D_i$ for $i=1,2$.

      Since $G_1$ and $G_2$ are simple and non-isogenous, every algebraic subgroup of $G^{k+n}$ is of the form $H_1 \cross H_2$ for $H_i$ a subgroup of $G_i^{k+n}$, so it follows that $(a,b)$ is free in $G^{k+n}$ over $D$.

      Since $A$ is generated as a field by $K_0(D)$ and the division points of
      $a_1$ and of $a_2$, we conclude as in case (EXP).

      \item[Case \DE] was covered in Proposition~\ref{good bases (D)}.
  \end{enumerate}
\end{proof}

\begin{corollary}\label{0-stability over ess finitary}
If $A$ is an essentially finitary $\Gamma$-field there are, up to isomorphism, only countably many finitely generated kernel-preserving extensions of $A$.
\end{corollary}
\begin{proof}
Each extension $B$ has a good basis $b$, and is determined by $\loc(b/A)$. Since $A$ is countable there are only countably many algebraic varieties defined over it.
\end{proof}

\section{Predimension and strong extensions}\label{predim section}

\subsection{Predimension}

We define a predimension function $\delta$ as follows.
\begin{definition}
Let $A \subs B$ be $\Gamma$-fields. For any $\Gamma$-subfield $X$ of $B$ that is finitely generated over $A$, let
\[
\delta(X/A) := \trd(X/A) - d \ldim_\kO(\Gamma(X)/\Gamma(A))
\]
where recall $d = \dim G_1 = \dim G_2$.

Note that since $X$ is assumed to be finitely generated over $A$, the linear dimension $\ldim_\kO(\Gamma(X)/\Gamma(A))$ is finite, and, since $\OO$ acts by $\bk$-definable functions and $X$ is the field generated by $\Gamma(X)$, $\trd(X/A)$ is also finite. Hence the predimension is well-defined. 

As a convention, for any finite $b \subset \Gamma(B)$, we set
\[
\delta(b/A) := \delta(X/A),
\]
where $X = \gen{Ab}$, the $\Gamma$-subfield of $B$ generated by $b \cup A$.
\end{definition}
Note that $\delta(b/A) = \trd(b/A) - d \ldim_\kO(b/\Gamma(A))$.

\begin{lemma}\label{delta properties}
Let $A \subs B$ be $\Gamma$-fields. 
\begin{enumerate}
\item (Finite character for $\delta$)\\
If $b\subs\Gamma(B)$ is finite, there is a finitely generated $\Gamma$-subfield $A_0$ of $A$ such that for any intermediate $\Gamma$-field $A_0 \subs A' \subs A$, we have $\delta(b/A) = \delta(b/A')$.

\item (Addition formula for $\delta$)\\
Let $X,Y$ be $\Gamma$-subfields of $B$ finitely generated over $A$ with $X \subs Y$. Then
\[
\delta(Y/A) =  \delta(Y/X) + \delta(X/A) .
\]

\item (Submodularity of $\delta$)\\
Suppose $X,Y$ are $\Gamma$-subfields of $B$ with $X$ finitely generated over $X \wedge Y$. Then abbreviating $\gen{X\cup Y}$ by $XY$ we have
\[ \delta(XY/ Y) \le \delta(X / X \wedge Y). \]
\end{enumerate}
\end{lemma}
\begin{proof}
\begin{enumerate}

\item Immediate since transcendence degree and $\kO$-linear dimension have finite character.

\item Note that the addition formula holds with transcendence degree or linear dimension in place of $\delta$, so it also holds for $\delta$ by linearity.

\item The submodularity condition is true when $\delta$ is replaced by transcendence degree. Linear dimension is modular, which means 
\[\ldim_\kO(\Gamma(XY)/\Gamma(Y)) = \ldim_\kO(\Gamma(X)/\Gamma(X \wedge Y)),\]
 so by subtracting we get the required submodularity of $\delta$. 
 
  \end{enumerate}
\end{proof}

\subsection{Strong extensions}

\begin{definition}
An extension $A \subs B$ of $\Gamma$-fields is said to be a \emph{strong extension} if
\begin{enumerate}
\item the extension preserves kernels; and
\item for every $\Gamma$-subfield $X$ of $B$ that is finitely generated over $A$, $\delta(X/A) \ge 0$.
\end{enumerate}
In this case, we also say that $A$ is a \emph{strong $\Gamma$-subfield} of $B$, and write $A \strong B$.

For arbitrary $\Gamma$-fields $A$, $B$, an embedding $A \hookrightarrow B$ is said to be a \emph{strong embedding} if the image of $A$ is a strong $\Gamma$-subfield of $B$. To denote that an embedding is strong we use the notation $A \strongembed B$.
\end{definition}

The method of predimensions and strong (also known as self-sufficient) extensions has been widely used since it was introduced by Hrushovski~\cite{Hru93}. We now give a few basic results which are well-known in general, but fundamental to the later development so it would be inappropriate to omit them. Some of the proofs are slightly more involved for this setting than the more well-known settings, especially those where no field is present.

\begin{lemma}\label{composite of strong is strong}
The  composition of strong embeddings is strong.
\end{lemma}
\begin{proof}
Suppose $A \strong B$ and $B \strong C$. Clearly the kernels of $C$ are the same as those of $A$, since both are the same as those of $B$. Let $X \subs C$ be finitely generated over $A$. Then $\delta(X/A) = \delta(X/X \wedge B) + \delta(X \wedge B /A)$ by the addition formula. We have $\delta(X/X\wedge B) \ge \delta(XB/B)$ by submodularity, and $\delta(XB/B) \ge 0$ because $B \strong C$. Also $\delta(X \wedge B /A) \ge 0$ because $A \strong B$. So $\delta(X/A) \ge 0$.
\end{proof}

Given a strong extension $A \strong B$ of $\Gamma$-fields, and an intermediate $\Gamma$-field $X$,  finitely generated over $A$, it follows that $A \strong X$ but it may not be the case that $X \strong B$. However, as $Y$ varies over finitely generated extensions of $X$ inside $B$, the predimension $\delta(Y/A)$ takes integer values bounded below by 0 because $A \strong B$. Thus we can replace $X$ by a finitely generated extension $X'$ of $X$, inside $B$, such that $\delta(X'/A)$ is minimal, and from the addition formula for $\delta$ it follows that $X' \strong B$.

The next lemma shows that we can find this $X'$ in a canonical way. It is crucial for understanding the finitely generated $\Gamma$-fields we will amalgamate, and it will allow us to understand the types in our models and prove there are only countably many of them.
\begin{lemma}\label{meet of strong is strong}
 Suppose $B$ is a $\Gamma$-field and for each $j \in J$, $A_j$ is a strong $\Gamma$-subfield of $B$. Then $\bigwedge_{j \in J}
  A_j$ is also strong in $B$.  
\end{lemma}
\begin{proof}
The kernels of $\bigwedge_{j\in J} A_j$ are the same as those of $B$ since they are for each $A_j$, so it remains to consider the predimension condition.

First we prove that if $A_1,A_2 \strong B$ then $A_1 \wedge A_2 \strong A_1$. So suppose $X$ is a finitely generated $\Gamma$-field extension of $A_1 \wedge A_2$, inside $A_1$. Then 
\[\delta(X / A_1 \wedge A_2) = \delta(X/X \wedge A_2) \ge \delta(X A_2/A_2) \ge 0\]
using submodularity and the fact that $A_2 \strong B$. So $A_1 \wedge A_2 \strong A_1$, but $A_1 \strong B$ so, by Lemma~
\ref{composite of strong is strong}, $A_1\wedge A_2 \strong B$.
It follows by induction that if $J$ is finite, $\bigwedge_{j \in J} A_j \strong B$.

Now suppose that $J$ is infinite and that $X$ is a $\Gamma$-subfield of $B$ which is finitely generated as an extension of $A = \bigwedge_{j \in J} A_j$.  Then we have
\[A = A\wedge X = \bigwedge_{j\in J}(A_j \wedge X).\]
Each $\Gamma$-field $A_j \wedge X$ is in the lattice of $\Gamma$-fields intermediate between $A$ and $X$. This lattice is isomorphic to the lattice of vector subspaces of the finite-dimensional vector space $\Gamma(X)/\Gamma(A)$ and so has no infinite chains. Thus there is a finite subset $J_0$ of $J$ such that writing $A_{J_0} = \bigwedge_{j \in J_0} A_j $ we have
\[A = \bigwedge_{j \in J} (A_j \wedge X) = \bigwedge_{j \in J_0} (A_j \wedge X) = A_{J_0} \wedge X.\]
 Now using the result for finite intersections we have that $A_{J_0} \strong B$, so using also submodularity we have 
\[ \delta(X/A) = \delta(X / A_{J_0} \wedge X) \ge \delta( A_{J_0}X/A_{J_0}) \ge 0 \]
and hence $A \strong B$ as required.
\end{proof}

%

Consider again a strong extension $A \strong B$ of $\Gamma$-fields, and an intermediate $\Gamma$-field $X$.
\begin{defn}\label{hull defn}
We define the \emph{hull} of $X$ in $B$, $\hull{X}_B$ (also known as the strong closure of $X$ or the self-sufficient closure of $X$) by 
\[\hull{X}_B = \bigwedge \class{Y \text{ a strong $\Gamma$-subfield of } B}{X \subs Y}.\]
\end{defn}
The previous lemma shows that $\hull{X}_B$ is indeed strong in $Y$, and we observe also that if $X$ is finitely generated as an extension of $A$ then so is $\hull{X}_B$. Furthermore, if $B \strong C$ then it is immediate that $\hull{X}_C = \hull{X}_B$, so often we will drop the subscript $B$.

\begin{lemma}\label{hull finite char}
The hull operator has finite character. That is, if $A \strong B$ and $X$ is an intermediate $\Gamma$-field,
\[ \hull{X}_B = \bigcup \class{\hull{X_0}_B}{X_0 \subs X \text{ and } X_0 \text{ is a finitely generated extension of } A}.\]
\end{lemma}
\begin{proof}
Let $U$ be the union in the statement of the lemma. It is immediate from the definition of the hull that if $X_0 \subs X$ then $\hull{X_0}_B \subs \hull{X}_B$. It follows that $U \subs \hull{X}_B$. Also $X \subs U$. Now $U$ is a directed union of strong $\Gamma$-subfields of $B$, and since $\delta$ has finite character, it follows that $U \strong B$. So $\hull{X}_B \subs U$, as required.
\end{proof}

Finally in this section we give a useful lemma giving a simple sufficient condition for an extension of a strong $\Gamma$-subfield also to be strong.
\begin{lemma}\label{delta=0 implies strong}
If $A \strong B$ and $A \subs A' \subs B$ with $\delta(A'/A) = 0$ then $A' \strong B$.
\end{lemma}
\begin{proof}
Let $X \subs B$ be a finitely generated extension of $A'$. Then
\[\delta(X/A') = \delta(X/A) - \delta(A'/A) = \delta(X/A) \ge 0.\]
\end{proof}


\subsection{Pregeometry}

In this section, $F$ is any full $\Gamma$-field strongly extending a $\Gamma$-subfield $\Fbase$. We will use the predimension function $\delta$ to define a pregeometry on $F$. We could drop the assumptions that $F$ is full and that $F$ strongly extends some $\Fbase$ and give a definition along the lines of that done for exponential fields in \cite{EAEF} and for Weierstrass $\wp$-functions in \cite{JKS16}. However, it is sufficient for our purposes and much more straightforward to do it this way.

\begin{defn}
A $\Gamma$-subfield $A$ of $F$, extending $\Fbase$, is \emph{$\Gamma$-closed in $F$}, written $A \strong_{\cl} F$, if for any $A \subs B \subs F$ with $B$ finitely generated over $A$ and $\delta(B/A) \le 0$ we have $B = A$.
\end{defn}
\begin{lemma}
\begin{enumerate}
\item If $A \strong_{\cl} F$ then $A \strong F$.
\item If $A \strong_{\cl} F$ then $A$ is a full $\Gamma$-subfield of $F$.
\item If $A_j \strong_{\cl} F$ for $j \in J$ and $A = \bigwedge_{j\in J} A_j$ then $A \strong_{\cl} F$.
\end{enumerate}
\end{lemma}
\begin{proof}
\begin{enumerate}
\item Immediate.
\item Suppose $a \in G_1(F)$ is algebraic over $A$. Since $F$ is full, there is $b \in G_2(F)$ with $(a,b) \in \Gamma(F)$. We have $\td((a,b)/A) = \td(b/A) \le \dim G_2 = d$, so $\delta((a,b)/A) \le d - d\ldim_\kO((a,b)/\Gamma(A))$. If $\ldim_\kO((a,b)/\Gamma(A)) =  1$ then $\delta((a,b)/A) \le 0$, so since $A$ is closed in $F$ we have $(a,b) \in \Gamma(A)$. Otherwise $\ldim_\kO((a,b)/\Gamma(A)) =  0$ so again $(a,b) \in \Gamma(A)$. Similarly if $b \in G_2(F)$ is algebraic over $A$. Since $G_1(A)$ contains all points of $G_1$ that are algebraic over $A$, $A$ is an algebraically closed field. Thus $A$ is a full $\Gamma$-field.
\item Suppose $\delta(B/A) \le 0$. By submodularity and Lemma~\ref{meet of strong is strong}, for each $j$ we have $\delta(B A_j/A_j) \le \delta(B/A_j \wedge B) = \delta(B/A) - \delta(A_j \wedge B/A) \le 0$, so $B \subs A_j$. Thus $B \subs A$.
\end{enumerate}
\end{proof}


This notion of $\Gamma$-closedness induces a closure operator on the field $F$.

\begin{defn}
If $A \subs F$ is any subset the \emph{$\Gamma$-closure of $A$ in $F$} is defined to be the smallest $\Gamma$-closed $\Gamma$-subfield containing $A$,
\[\Gammacl^F(A) = \bigwedge \class{B \strong_{\cl} F}{A \subs B}.\]
$\Gammacl^F(A)$ is a $\Gamma$-subfield of $F$, and in particular a subset of $F$, so $\Gammacl^F$ induces a map $\powerset F \rightarrow \powerset F$ which we also denote by $\Gammacl^F$.
\end{defn}




\begin{lemma}\label{closure as union}
For any $\Gamma$-subfield $A$ of $F$, we have $\Gammacl^F(A) = \bigcup \mathcal B$, where $\mathcal B$ is the set of all $\Gamma$-subfields $B \subs F$ such that $B$ is a finitely generated $\Gamma$-field extension of $\hull{A}_F$ and $\delta(B/\hull A_F) = 0$.
\end{lemma}
\begin{proof}
Since $\Gammacl^F(A) \strong F$ we have $\hull{A}_F \subs \Gammacl^F(A)$. So $\Gammacl^F(A) = \Gammacl^F(\hull{A}_F)$, and thus we may assume $A \strong F$. Let $C = \bigcup \mathcal B$. Using the submodularity of $\delta$ it is easy to see that the system $\mathcal B$ of $\Gamma$-subfields of $F$ is directed, so its union $C$ is a $\Gamma$-subfield of $F$.

Suppose that $b$ is a finite tuple from $\Gamma(F)$ such that $\delta(b/C) \le 0$. Then by the finite character of $\delta$ and directedness of the union defining $C$, there is a finitely generated extension $B$ of $A$ inside $C$ such that $\delta(B/A)=0$ and $\delta(b/B) = \delta(b/C)$. Using the addition formula, 
\[0 \ge \delta(b/B) = \delta(b/A) - \delta(B/A) = \delta(b/A) \ge 0.\]
 So $\delta(b/A) = 0$ and hence $b \in \Gamma(C)$. Thus $C$ is $\Gamma$-closed, so $\Gammacl^F(A) \subs C$.

Now suppose $B$ is a finitely generated $\Gamma$-field extension of $A$ with $\delta(B/ A) = 0$ and $A \subs D \strong_{\cl} F$.
Then 
\begin{eqnarray*}
\delta(BD/D) & \le & \delta(B/B \wedge D)\\
& = & \delta(B/A) - \delta(B \wedge D / A)\\
& \le &0
\end{eqnarray*}
because $A \strong F$ so $\delta(B \wedge D / A) \ge 0$ and so $B \subs D$. Hence $B \subs \Gammacl^F(A)$, and so $C \subs \Gammacl^F(A)$.
\end{proof}

The predimension function $\delta$ is a function depending on the sort $\Gamma$, but $\Gamma$-closure will be shown to be a pregeometry on the field sort. The next lemma allows us to move from one sort to the other.
\begin{lemma}\label{changing sorts}
If $A \strong F$ and $a \in F \minus \Gammacl^F(A)$ there is $\alpha \in \Gamma(F)$ such that $\pi_1(\alpha) \in G_1(F)$ is interalgebraic with $a$ over $A$ and $\delta(\alpha/A) = 1$, and the $\Gamma$-subfield  $\gen{A\alpha}$ of $F$ generated by $A$ and $\alpha$ satisfies  $\gen{A\alpha} \strong F$. We can choose $\alpha$ such that $a$ is rational over $A(\pi_1(\alpha))$, and if $A$ is essentially finitary also such that $\alpha$ is a good basis. Furthermore, the locus $\loc(\alpha,a/A)$ can be taken to be a particular algebraic curve defined over $\bk$, not depending on $A$ or $a$.
\end{lemma}
\begin{proof}
We have $\td(a/A) = 1$ (since otherwise $a \in \Gammacl^F(A)$).
We have fixed an identification of $G_1$ with a constructible subset of affine space $\mathbb A^N$ for some $N$, defined over \bk. Let $f$ be the projection map from $G_1$ to the first coordinate where the projection is dominant, and choose a constructible curve $X \subseteq  G_1$ by fixing the values of all the other coordinates to be values in \bk. Then $X$ and the map $f$ are defined over \bk.
Choose $\alpha_1 \in X(F)$ with $f(\alpha_1) = a$. Then $\alpha_1$ is interalgebraic with $a$ over $A$, with $a$ rational over $A(\alpha_1)$, and the locus $\loc(\alpha_1,a/A)$ is defined over $\bk$.

Since $F$ is full, there is $\alpha \in \Gamma(F)$ with $\pi_1(\alpha) = \alpha_1$. Then $\alpha \in \Gammacl^F(A a)$ and $a \in \Gammacl^F(A \alpha)$. Then $\ldim_\kO(\alpha/\Gamma(A)) = 1$ and so since $\delta(\alpha/A) \neq 0$ by Lemma~\ref{closure as union}, $\td(\alpha/A) = d+1$ and  $\delta(\alpha/A) = 1$. If there were $B \sups \gen{A\alpha}$ with $\delta(B/A\alpha) < 0$ then $\delta(B/A) \le 0$ which contradicts $a \notin \Gammacl^F(A)$. So $\gen{A\alpha} \strong F$.  If $A$ is essentially finitary then by Lemma~\ref{good bases exist} we can divide $\alpha$ by some $m \in \N^+$ to ensure it is a good basis. Since $\alpha_2$ is generic in $G_2(A)$ over $\alpha_1$, and $G_2$ is defined over $\bk$, we deduce that $\loc(\alpha,a/A)$ is defined over $\bk$.
\end{proof}

\begin{prop}\label{closed sets of pregeom}
The $\Gamma$-closed subsets of $F$ are the closed sets of a pregeometry on $F$.
\end{prop}
\begin{proof}
It is immediate that for any subsets $A \subs B$ of $F$ we have $A \subs \Gammacl^F(A)$, $\Gammacl^F(\Gammacl^F(A)) = \Gammacl^F(A)$ and $\Gammacl^F(A) \subs \Gammacl^F(B)$. 

For finite character, suppose $b \in \Gammacl^F(A)$.
By Lemma~\ref{closure as union} there is a finitely generated extension $\hull{A}_F \subs B$ in $F$ such that $\delta(B/\hull{A}_F) = 0$ and $b \in B$. Then there is a finite tuple $\beta \in \Gamma(F)$ with $b$ rational over $\beta$ and $\delta(\beta/A) = 0$. By finite character of $\delta$ from Lemma~\ref{delta properties}, there is a finitely generated $\Gamma$-subfield $A_0$ of $\hull{A}_F$ such that for any $A'$ with $A_0 \subs A' \subs \hull{A}_F$ we have $\delta(\beta/A') = 0$. So by Lemma~\ref{closure as union} again, $b \in \Gammacl^F(A_0)$.

The hull operator is a closure operator which by Lemma~\ref{hull finite char} has finite character. We have $A_0 \subs \hull{A}_F$, so there is a finite subset $A_{00}$ of $A$ such that $\hull{A_0}_F = \hull{A_{00}}_F$. Hence $b \in \Gammacl^F(A_{00})$, and so $\Gammacl^F$ has finite character.

For exchange, suppose $A \strong_{\cl} F$ and that $a,b \in F \minus A$ with $b \in \Gammacl^F(Aa)$. Using Lemma~\ref{changing sorts}, we choose $\alpha, \beta \in \Gamma(F)$ corresponding to $a$ and $b$ respectively.

Now $\beta \in \Gammacl^F(A\alpha)$ so there is a finitely generated $\Gamma$-field extension $A \subs B$ inside $F$ with $\beta,\alpha \in \Gamma(B)$ and $\delta(B/A\alpha) = 0$. Then we have
\begin{eqnarray*}
\delta(B/A\beta) & = & \delta(B/A) - \delta(\beta/A) \\
 & = & \delta(B/A) - 1\\
 & = & \delta(B/A) - \delta(\alpha/A)\\
 & = & \delta(B/A\alpha) = 0 \\
\end{eqnarray*}
so $\alpha \in \Gammacl^F(A\beta)$, or equivalently $a \in \Gammacl^F(Ab)$.
\end{proof}

We write $\Gammadim^F$ for the dimension with respect to the pregeometry $\Gammacl^F$. However, if $F_1$ and $F_2$ are both full $\Gamma$-fields with $F_1 \strong_{\cl} F_2$ and $A \subs F_1$ then $\Gammacl^{F_1}(A) = \Gammacl^{F_2}(A)$. So from now on we will usually drop the superscript $F$ and just write $\Gammacl$ and $\Gammadim$ except where it might cause confusion.

We have the usual connection between the dimension and the predimension function.
\begin{lemma}\label{dim and predim}
Suppose that $A \strong F$ and $B$ is a finitely generated $\Gamma$-field extension of $A$ in $F$. Then:
\begin{enumerate}
\item $\Gammadim(B/A) = \min\class{\delta(C/A)}{B \subs C \subs F}$, and
\item  $B \strong F$ if and only if \ $\Gammadim(B/A) = \delta(B/A)$.
\end{enumerate}
\end{lemma}
\begin{proof}
Since $\hull B_F \subs \Gammacl^F(B)$, we have $\Gammadim(B/A) = \Gammadim(\hull B_F/A)$. Now it follows from the addition formula that $\delta(\hull B_F/A) = \min\class{\delta(C/A)}{B \subs C \subs F}$, so statement 1 reduces to the left-to-right direction of statement 2. To prove that, first assume $B \strong F$.

Let $n = \Gammadim(B/A)$ and let $b_1,\ldots,b_n$ be a $\Gcl$-basis for $B$ over $A$. Applying Lemma~\ref{changing sorts}, we get $\beta_i \in \Gamma(F)$ in the closure of $B$ with $\beta_i$ corresponding to $b_i$. Let $D$ be the $\Gamma$-subfield of $F$ generated by $A$ and $\beta_1,\ldots,\beta_n$. Then $\delta(D/A) = n$ and $D \strong F$. Furthermore $\Gammacl^F(D) = \Gammacl^F(B)$.

Since $D \subs  \Gammacl^F(B)$, there is $C \sups B \cup D$ such that $\delta(C/B) = 0$. Then
\[\delta(B/A) = \delta(C/A) - \delta(C/B) = \delta(C/A) \ge \delta(D/A) = n\]
using that $D \strong F$. Reversing the roles of $B$ and $D$, the same argument shows that $\delta(B/A) \le \delta(D/A)$, and so $\delta(B/A) = n = \Gammadim(B/A)$ as required.

The right-to-left direction of statement 2 now follows from statement 1 and the addition property. 
\end{proof}

\begin{remarks}
\begin{enumerate}
\item In the sense of the pregeometry $\Gammacl$, the set $\Gamma(F)$ is $d$-dimensional. Thus when $d=1$ such as in pseudo-exponentiation and pseudo-$\wp$ we actually get a pregeometry directly on $\Gamma(F)$.
\item In the case of pseudo-exponentiation or a pseudo-$\wp$-function, $G_1(F) = \ga(F) = F$, and $\Gamma$ is the graph of a function $\exp$, so we have a bijection $\phi: F \to \Gamma(F)$ given by $x \mapsto (x,\exp(x))$. The predimension usually considered for exponentiation, for example in \cite{Zilber05peACF0}, is a function on tuples from the field sort, and in fact is just the composite $\delta \circ \phi$ of the predimension function described here with $\phi$.
\item It is possible to define a predimension function directly on the field sort, even in our generality. 
Given any subfield $A$ of $F$ we write $\Gamma(A)$ for $\Gamma(F) \cap G(A)$. For any subset $X$ of $F$ (in the field sort) we write $X^\alg$ for the field-theoretic algebraic closure of $\Fbase(X)$ in $F$.

Given subsets $X,Y$ of $F$, with $\td\left((X\cup Y )^\alg/X^\alg\right) < \infty$,  define 
\[\pd(Y/X) = \td\left((X\cup Y )^\alg/X^\alg\right) - \ldim_\kO\left(\Gamma((X\cup Y )^\alg)/\Gamma(X^\alg)\right)\]
which takes values in $\Z \cup \{-\infty\}$.

The predimension functions $\pd$ and $\delta$ are closely related, and we could write $\pd(Y/X) = \delta((X\cup Y)^\alg/X^\alg)$ except that $X^\alg$ will usually fail to be a $\Gamma$-subfield of $F$ by our definition, because as a field it will not usually be generated by the coordinates of the points in $\Gamma(X^\alg)$. It may not even be algebraic over the field generated by those points.

It is possible to define the notion of strong embeddings of $\Gamma$-fields using this predimension function instead. Some things are easier with this approach, because the predimension is defined on a 1-dimensional sort. However we choose to work in the sort $\Gamma$ because it is a vector space and hence has a modular geometry, which makes other things much easier.
\end{enumerate}
\end{remarks}


\subsection{Full closures}
The following theorem and proof follow \cite[Theorem~2.18]{FPEF}.

\begin{theorem}\label{full extension}
If $A$ is a $\Gamma$-field then there is a full $\Gamma$-field extension $\Afull$ of $A$ such that $A \strong \Afull$ and $\Afull$ is generated as a full $\Gamma$-field by $A$.
Furthermore if $A$ is essentially finitary then $\Afull$ is unique up to isomorphism as an extension of $A$.
\end{theorem}
\begin{proof}
First we prove existence. Embed $A$ in a large algebraically closed field $F$. 
Choose a point $a \in G_1(F)$ which is algebraic over $A$ and not in $\pi_1(\Gamma(A))$, if such exists. Choose a division sequence $\hat a \in \hat G_1(F)$ for $a$. Let $b \in G_2(F)$ be generic over $A$ and choose a division system $\hat b$ for it. (Up to field isomorphism over $A$, $\hat b$ is unique.) Let $A'$ be the field generated by $A$ and the division sequences $\hat a = (a_m)_{m \in \N^+}$ and $\hat b = (b_m)_{m\in \N^+}$, and define $\Gamma(A')$ to be the $\OO$-submodule of $G(A')$ generated by $\Gamma(A)$ and the points $(a_m,b_m)$ for $m \in \N^+$. Then $A'$ is a $\Gamma$-field extension of $A$. Since $\pi_1(\Gamma(A))$ already contains the torsion of $G_1$, the extension preserves the kernels. We have
\[\delta(A'/A) = \td(b/A) - d \ldim_\kO (b/A) = d-d = 0,\]
so it is a strong extension. Similarly if there is $b \in G_2(F)$ which is algebraic over $A$ but not in $\pi_2(\Gamma(A))$ we can form a similar strong extension. Iterating these constructions, a strong full extension $A^\mathrm{full}$ of $A$ is readily seen to exist.

Now we prove uniqueness under the additional hypothesis that $A$ is essentially finitary. Suppose that $B$ and $B'$ both satisfy the conditions for $\Afull$. Since $A$ is essentially finitary it is countable, and then the construction above shows that we can take $B$ to be countable as well. Enumerate $\Gamma(B)$ as $(s_n)_{n \in \N^+}$ such that for each $n$, either $\pi_1(s_n)$ or $\pi_2(s_n)$ is algebraic over $A \cup \{s_1,\ldots,s_{n-1}\}$. This is possible since $B$ is generated as a full $\Gamma$-field by $A$.

We will inductively construct a chain of strong $\Gamma$-subfields $A_n \strong B$, each a finitely generated $\Gamma$-field extension of $A$ such that $A_0 = A$ and $s_n \in \Gamma(A_n)$. We will also construct a chain of strong embeddings $\theta_n : A_n \strongembed B'$. Assume we have $A_n$ and $\theta_n$. Let $A_{n+1}$ be the $\Gamma$-subfield of $B$ generated by $A_n$ and $s_{n+1}$. As a field, $A_{n+1}$ is generated by $A_n$ and the division points of $s_{n+1}$. If $s_{n+1} \in \Gamma(A_n)$, then we have $A_{n+1} = A_n$ and can just take $\theta_{n+1} = \theta_n$. Otherwise, we have $\ldim_\kO(\Gamma(A_{n+1})/\Gamma(A_n)) \ge 1$. By hypothesis, one of $\pi_1(s_{n+1})$ or $\pi_2(s_{n+1})$ is algebraic over $A_n$, say $\pi_1(s_{n+1})$. Thus $\td(A_{n+1}/A_n) = \td(s_{n+1}/A_n) \le d$. By inductive hypothesis $A_n \strong B$, so we have $\delta(A_{n+1}/A_n) \ge 0$. It follows that $\ldim_\kO(\Gamma(A_{n+1})/\Gamma(A_n)) = 1$ and $\td(A_{n+1}/A_n)  = d$, so  $\delta(A_{n+1}/A_n) = 0$. Thus by Lemma~\ref{delta=0 implies strong}, $A_{n+1} \strong B$. Also $\pi_2(s_{n+1})$ is generic in $G_2$ over $A_n$.

Since $A_n$ is a finitely generated $\Gamma$-field extension of $A$, by Proposition~\ref{good bases exist} there is $m \in \N$ such that $\{s_{n+1}/m\}$ is a good basis for the extension $A_n \strong A_{n+1}$. Replacing $s_{n+1}$ by $s_{n+1}/m$, we may assume $m=1$. Now let $W$ be the locus of $\pi_1(s_{n+1})$ over $A_n$ (a variety of dimension 0, irreducible over $A_n$, but not necessarily absolutely irreducible), and let $w$ be any point in $W^{\theta_n}$, the corresponding subvariety of $G_1(B')$. Choose $v \in G_2(B')$ such that $(w,v) \in \Gamma(B')$. Since $w$ is algebraic over $\theta_n(A_n)$ which is strong in $B'$, the same predimension argument as above shows that $v$ is generic in $G_2$ over $\theta_n(A_n)$, so $\loc(w,v/\theta_n(A_n)) = W^{\theta_n} \cross G_2 = (\loc(s_{n+1}/A_n))^{\theta_n}$. Since $s_{n+1}$ is a good basis over $A_n$, we can extend $\theta_n$ to a field embedding $\theta_{n+1} : A_{n+1} \to B'$ with $\theta(s_{n+1}) = (w,v)$, and using again that $\theta_n(A_n) \strong B'$ we get that $\Gamma(\theta_{n+1}(A_{n+1}))$ is generated by $(w,v)$ over $\Gamma(\theta_n(A_n))$ and hence $\theta_{n+1}$ is a $\Gamma$-field embedding. 

Also $\delta(\theta_{n+1}(A_{n+1})/ \theta_n(A_n)) = 0$ so $\theta_{n+1}$ is a strong embedding by Lemma~\ref{delta=0 implies strong}.

Now $B = \Union \class{A_n}{n \in \N}$ and $\Union \class{\theta_n(A_n)}{n \in \N}$ is a full $\Gamma$-subfield of $B'$ containing $A$, so it must be $B'$. Hence $\Union \class{\theta_n}{n \in \N}$ is an isomorphism $B \iso B'$. So $\Afull$ is unique, up to isomorphism as an extension of $A$.
\end{proof}

\begin{prop}\label{0-stability for full}
Let $A$ be a countable full $\Gamma$-field. Then there are only countably many finitely generated strong full $\Gamma$-field extensions of $A$, up to isomorphism.
\end{prop}
\begin{proof}
Let $A \strong B$ be such an extension and let $b$ be a finite tuple generating $B$ over $A$ as a full $\Gamma$-field, such that $B_0 \leteq \gen{Ab} \strong B$, and of minimal length such. Then by Proposition~\ref{good bases exist} we may replace $b$ by $b/m$ for some $m \in \N^+$ to ensure that $b$ is a good basis for the extension $A \strong B_0$. Then $B = B_0^\mathrm{full}$ which by Proposition~\ref{full extension} is determined uniquely up to isomorphism by $B_0$, and by Corollary~\ref{0-stability over ess finitary} there are only countably many choices for $B_0$.
\end{proof}


\section{The canonical countable model}\label{amalg section}

\subsection{The amalgamation theorem}
We use the definition of \emph{amalgamation category} from \cite{TEDESV}, slightly extending Droste and G\"obel \cite{DG92} who were themselves abstracting from Fra\"iss\'e's amalgamation theorem. We restrict to the countable case. We will apply the general theory to various categories of $\Gamma$-fields with strong embeddings as morphisms. The notions of \emph{finitely generated}, \emph{universal} and \emph{saturated} all have category-theoretic translations which we give first.
\begin{defn}
Given a category $\Cat$, an object $A$ of $\Cat$ is said to be \emph{$\aleph_0$-small} if and only if for every $\omega$-chain $(Z_i,
\gamma_{ij})$ in $\Cat$ with direct limit $Z_\omega$, any arrow $\ra{A}{f}{Z_\omega}$ factors through the chain, that is, there is $i < \omega$ and $\ra{A}{f^*}{Z_i}$ such that $f = \gamma_{i\omega} \circ f^*$. We write $\Cat^{<\aleph_0}$ for the full subcategory of $\aleph_0$-small objects of $\Cat$ and $\Cat^{\le\aleph_0}$ for the full subcategory of the limits of $\omega$-chains of $\aleph_0$-small objects of $\Cat$. 
\end{defn}

\begin{defn}
Given a category $\Cat$ and a subcategory $\Cat'$, an object $U$ of $\Cat$ is said to be \emph{$\Cat'$-universal} if for every object $A$ of $\Cat'$ there is an arrow $\ra{A}{}{U}$ in $\Cat$.  $U$ is \emph{$\Cat'$-saturated} if for every arrow $\ra{A}{f}{B}$ in $\Cat'$ and every arrow $\ra{A}{g}{U}$ in $\Cat$, there is an arrow $\ra{B}{h}{U}$ in $\Cat$ such that $g=h \circ f$. $U$ is \emph{$\Cat'$-homogeneous} if for every object $A$ of $\Cat'$ and every pair of arrows $\ra{A}{f,g}{U}$ in $\Cat$, there exists an isomorphism $\ra{U}{h}{U}$ in $\Cat$ such that $g=h \circ f$.
\end{defn}
Some authors refer to $\Cat'$-saturation as \emph{richness} with respect to the objects and arrows from $\Cat'$.

\begin{definition}
A category $\Cat$ is an \emph{amalgamation category} if the following hold.
 \begin{enumerate}[{AC}1.]
  \item Every arrow in $\Cat$ is a monomorphism.
  \item $\Cat$ has direct limits (unions) of $\omega$-chains.
  \item $\Cat^{<\aleph_0}$ has at most $\aleph_0$ objects up to
    isomorphism.
  \item For each object $A \in \Cat^{<\aleph_0}$ there are at most
    $\aleph_0$ extensions of $A$ in $\Cat^{<\aleph_0}$, up to isomorphism.
  \item $\Cat^{<\aleph_0}$ has the \emph{amalgamation property} (AP),
    that is, any diagram of the form
    
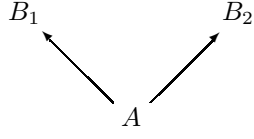
\begin{diagram}[height=2em,width=2em]
      B_1 &&&&B_2\\
      &\luTo&& \ruTo\\
      &&A
    \end{diagram}
    can be completed to a commuting square
    
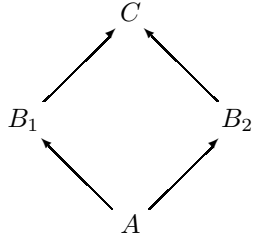
\begin{diagram}[height=2em,width=2em]
      &&C\\
      &\ruTo&&\luTo\\
      B_1 &&&&B_2\\
      &\luTo&& \ruTo\\
      &&A
    \end{diagram}
    in $\Cat^{<\aleph_0}$.
  \item $\Cat^{<\aleph_0}$ has the \emph{joint embedding property} (JEP),
    that is, for every $B_1,B_2 \in \Cat^{<\aleph_0}$ there is $C \in
     \Cat^{<\aleph_0}$ and arrows
    
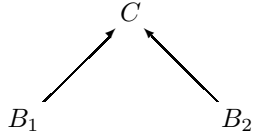
\begin{diagram}[height=2em,width=2em]
      &&C\\
      &\ruTo&&\luTo\\
      B_1 &&&&B_2\\
    \end{diagram}
    in $\Cat^{<\aleph_0}$.
  \end{enumerate}
\end{definition}
The point of the definition is that the following form of Fra\"iss\'e's amalgamation theorem holds.
\begin{theorem}[{\cite[Theorem~2.18]{TEDESV}}]\label{amalgamation theorem}
  If $\Cat$ is an amalgamation category then there is an object
  $U \in \Cat^{\le\aleph_0}$, the ``Fra\"iss\'e limit'', which is $\Cat^{\le\aleph_0}$-universal and
  $\Cat^{<\aleph_0}$-saturated. 
  
Furthermore, if $A \in \Cat^{\le\aleph_0}$ is $\Cat^{<\aleph_0}$-saturated then $A \iso U$.
\end{theorem}

\begin{remark}
  It follows from saturation and a back-and-forth argument that $U$ is also $\Cat^{<\aleph_0}$-homogeneous.
\end{remark}

\subsection{Amalgamation of $\Gamma$-fields}

We fix a $\Gamma$-field $\Fbase$ which is either finitely generated as a $\Gamma$-field, or is a countable full $\Gamma$-field.

The identity map on a $\Gamma$-field is obviously a strong embedding, hence from Lemma~\ref{composite of strong is strong} we have a category of strong $\Gamma$-field extensions of $\Fbase$, with strong embeddings as the arrows. We write $\Cc(\Fbase)$ for this category, but will usually abbreviate it to $\Cc$.
We also consider the following full subcategories of $\Cc$.
\begin{notation} \
\begin{itemize}
\item $\Cfull$ (or $\Cfull(\Fbase)$) consists of the full strong $\Gamma$-field extensions of $\Fbase$.
\item $\Cfg$ consists of the strong $\Gamma$-field extensions of $\Fbase$ which are finitely generated.
\item $\Cfgfull$ consists of the strong $\Gamma$-field extensions of $\Fbase$ which are full and finitely generated as full extensions.
\item $\Ccount$ consists of the strong $\Gamma$-field extensions of $\Fbase$ which are countable.
\item $\Cfullcount$ consists of the strong $\Gamma$-field extensions of $\Fbase$ which are full and countable.
\end{itemize}
\end{notation}
For our categories $\Cc$ and $\Cfull$, it is immediate that $\aleph_0$-small just means finitely generated in the appropriate sense, and a (full) $\Gamma$-field is the union of an $\omega$-chain of finitely generated (full) $\Gamma$-fields if and only if it is countable.

We will construct our canonical model as the Fra\"iss\'e limit of $\Cfg$. In fact it is also the Fra\"iss\'e limit of $\Cfgfull$.


In proving the amalgamation property we actually prove a stronger result, asymmetric amalgamation, which will be necessary when we come to axiomatize our models. However, the asymmetric property holds only in the case of full $\Gamma$-fields, not for $\Cfg$. We also observe that our amalgams are disjoint.
\begin{prop}\label{AAP}
The categories $\Cfull$ and $\Cfullcount$ have the disjoint asymmetric amalgamation property. That is, given full $\Gamma$-fields $A_0, A_L, A_R \in \Cfull$, an embedding $A_0 \into A_L$ and a strong embedding $A_0 \strongembed A_R$, there exist $A \in \Cfull$ and dashed arrows making the following diagram commute;
  
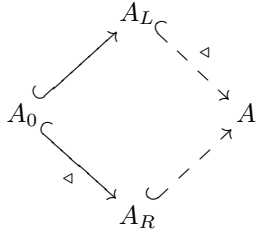
\[ \xymatrix{
      & A_L \hstrexar[dr] & \\
    A_0 \embar[ur] \lstrar[dr] &   & A \\
      & A_R \embexar[ur] &  \\
      }
      \]
      moreover, if the embedding $A_0 \into A_L$ is also strong, then so is the embedding $A_R \into A$;  furthermore, identifying $A_0$, $A_L$ and $A_R$ with their images in $A$, we have that $A_L \cap A_R = A_0$.
\end{prop}
\begin{proof}
Since $A_0$ is algebraically closed as a field, we may form the free amalgam $A_1$ of $A_L$ and $A_R$ over $A_0$ as fields, that is, the unique (up to isomorphism) field compositum of $A_L$ and $A_R$ in which they are algebraically independent over $A_0$. We identify $A_L$ and $A_R$ as subfields of $A_1$ so, in particular, $A_L \cap A_R = A_0$. We make $A_1$ into a $\Gamma$-field by defining $\Gamma(A_1)$ to be the $\OO$-submodule $\Gamma(A_L) + \Gamma(A_R)$ of $G(A_1)$. 

Then $\Gamma(A_L)$ and $\Gamma(A_R)$ are $\OO$-submodules of $\Gamma(A_1)$.

Suppose that $a \in \ker_1(A_1)$, that is, $(a,0) \in \Gamma(A_1)$. Then there are $(a_L,b_L) \in \Gamma(A_L)$ and $(a_R,b_R) \in \Gamma(A_R)$ such that $(a,0) = (a_L,b_L) + (a_R,b_R)$.
Then $b_L = - b_R$, so 
\[b_L,b_R \in \Gamma_2(A_L) \cap \Gamma_2(A_R) \subs G_2(A_L) \cap G_2(A_R) = G_2(A_0).\]
Since $A_0$ is a full $\Gamma$-field there is $a_0 \in G_1(A_0)$ such that $(a_0,b_L) \in \Gamma(A_0)$.  Then $a_L - a_0 \in \ker_1(A_L) = \ker_1(A_0)$, so $a_L \in \Gamma_1(A_0)$. Similarly, $a_R \in \Gamma_1(A_0)$, so $a \in \Gamma_1(A_0)$.

Thus $\ker_1(A_1) = \ker_1(A_0)$. The same argument shows that $\ker_2(A_1) = \ker_2(A_0)$, and hence the inclusions of $A_L$ and $A_R$ into $A_1$ preserve the kernels. 

Let us check that the inclusion $A_L \into A_1$ is strong. Let $X$ be a $\Gamma$-subfield of $A_1$ which is finitely generated over $A_L$. Choose a basis $b$ for the extension, say of length $n$. Translating by points in $\Gamma(A_L)$, we may assume that $b \in \Gamma(A_R)^n$. Now $\delta(b/A_0) \ge 0$ since $A_0 \strong A_R$, so $\td(b/A_0) \ge d \ldim_\kO(g/\Gamma(A_0) = dn$.

Since $A_R$ is $\ACF$-independent from $A_L$ over $A_0$, we have $\td(b/A_L) = \td(b/A_0)$, and we also have $\ldim_\kO(b/\Gamma(A_L)) = n$ by assumption, so
\[\delta(X/A_L) = \td(b/A_L) - d \ldim_\kO(b/\Gamma(A_L)) = \delta(b/A_0) \ge 0\]
as required. Thus $A_L \strong A_1$. The same argument shows that if the embedding $A_0 \into A_L$ is strong, then so is the embedding $A_R \into A_1$.

Now take $A = A_1^\mathrm{full}$, which exists and is a strong extension of $A_1$, by the existence part of Theorem~\ref{full extension}. Note that if $A_L$ and $A_R$ are countable then so is $A$.
\end{proof}

\begin{corollary}\label{AP}

The category $\Cfg$ has the amalgamation property. That is, given $A_0, A_L, A_R \in \Cfg$ and strong embeddings $A_0 \strongembed A_L$ and $A_0 \strongembed A_R$ as in the following diagram, there exist $A \in \Cfg$ and dashed arrows making the diagram commute.
  
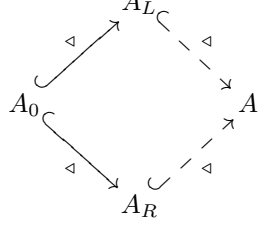
\[ \xymatrix{
      & A_L \hstrexar[dr] & \\
    A_0 \hstrar[ur] \lstrar[dr] &   & A \\
      & A_R \lstrexar[ur] &  \\
      }
      \]

\end{corollary}

\begin{proof}
Let $A_0, A_L, A_R$ be as in the statement. By the existence part of Theorem~\ref{full extension}, we can extend each of the three $\Gamma$-fields to its full closure.
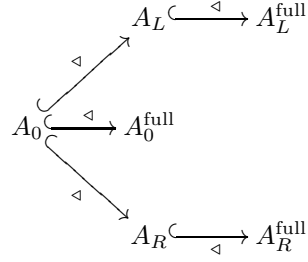
\[ 
\xymatrix{
      & A_L \hstrar[r] & A_L^\mathrm{full}\\
  A_0 \hstrar[ur] \lstrar[dr] \hstrar[r] & A_0^\mathrm{full}\\
      & A_R \lstrar[r] & A_R^\mathrm{full}\\
}
\]
Then because we have $A_0 \strong A_L^\mathrm{full}$ and $A_0 \strong A_R^\mathrm{full}$, by the uniqueness part of the same theorem there are embeddings as in the following diagram, which are strong by Lemma~\ref{delta=0 implies strong} and finite character of $\delta$.
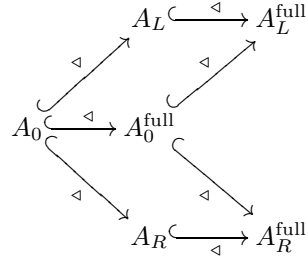
\[ 
\xymatrix{
      & A_L \hstrar[r] & A_L^\mathrm{full}\\
  A_0 \hstrar[ur] \lstrar[dr] \hstrar[r] & A_0^\mathrm{full}  \hstrar[ur] \lstrar[dr] \\
      & A_R \lstrar[r] & A_R^\mathrm{full}\\
}
\]
By Theorem~\ref{AAP}, we can complete the diagram to
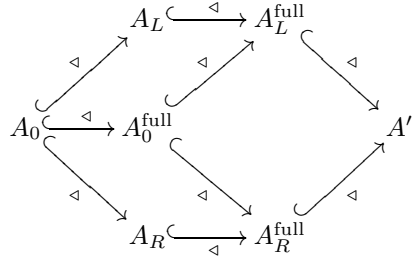
\[ 
\xymatrix{
      & A_L \hstrar[r] & A_L^\mathrm{full} \hstrar[dr] &\\
  A_0 \hstrar[ur] \lstrar[dr] \hstrar[r] & A_0^\mathrm{full}  \hstrar[ur] \lstrar[dr] && A'\\
      & A_R \lstrar[r] & A_R^\mathrm{full} \lstrar[ur] &\\
}
\]
and then we can take $A$ to be the $\Gamma$-subfield of $A'$ generated by $A_L \cup A_R$, which is in $\Cfg$.
\end{proof}


\begin{theorem}\label{amalg category theorem}
The categories $\Cc$ and $\Cfull$ are amalgamation categories, with the same Fra\"iss\'e limit.
\end{theorem}
\begin{proof}
Strong embeddings are injective functions, so monomorphisms. Hence AC1 holds. It is clear that the union of a chain of (full) $\Gamma$-fields is a (full) $\Gamma$-field, so AC2 holds. AC4 is given by Corollary~\ref{0-stability over ess finitary} for $\Cc$ and Proposition~\ref{0-stability for full} for $\Cfullfg$. The amalgamation property AC5 is proved in Proposition~\ref{AAP} and Corollary~\ref{AP}. Since every $\Gamma$-field in $\Cc$ is an extension of $\Fbase$, and every full $\Gamma$-field in $\Cfull$ is an extension of $(\Fbase)^\mathrm{full}$, properties AC3 and AC6 follow from AC4 and AC5 respectively.

Thus $\Cc$ and $\Cfull$ are both amalgamation categories. Let $M$ be the Fra\"iss\'e limit of $\Cfull$. If $A \in \Ccount$ then $\Afull \in \Cfullcount$, so as $M$ is $ \Cfullcount$-universal there is a strong embedding $\Afull \strong M$, which restricts to a strong embedding $A \strong M$. Hence $M$ is $\Ccount$-universal. 
Similarly, using Proposition~\ref{full extension} and the $\Cfullfg$-saturation of $M$ we can see that $M$ is also $\Cfg$-saturated. Hence $M$ is also the Fra\"iss\'e limit of $\Cfg$.
\end{proof}

\begin{notation}
We write $M(\Fbase)$ for the Fra\"iss\'e limit in $\Cc$.
\end{notation}

\subsection{$\Gamma$-algebraic extensions}

\begin{defn}
Let $A \strong B$ be a strong extension of $\Gamma$-fields. The extension is \emph{$\Gamma$-algebraic} if for all finite tuples $b$ from $\Gamma(B)$ there is a finite tuple $c \in \Gamma(B)$ containing $b$ such that $\delta(c/A) =0$.
\end{defn}
\begin{remark}\label{Galg predim}
From Lemma~\ref{dim and predim} we see that if $F$ is a full $\Gamma$-field such that $B \strong F$ then the extension $A \strong B$ is $\Gamma$-algebraic if and only if $B \subs \Gcl^F(A)$.
\end{remark}

Let $\Calg$ be the subcategory of $\Cc$ consisting of the $\Gamma$-algebraic extensions of $\Fbase$.

\begin{prop}\label{Calg is amalg cat}
$\Calg$ is an amalgamation category.
\end{prop}
\begin{proof}
The proof of Theorem~\ref{amalg category theorem} goes through, except we also have to show that the amalgam of $\Gamma$-algebraic extensions is $\Gamma$-algebraic. So suppose we have the amalgamation square
  
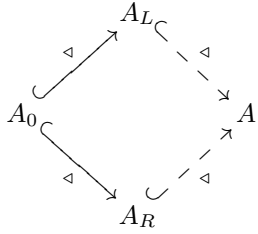
\[ \xymatrix{
      & A_L \hstrexar[dr] & \\
    A_0 \hstrar[ur] \lstrar[dr] &   & A \\
      & A_R \lstrexar[ur] &  \\
      }
      \]
as in Corollary~\ref{AP} with $A_L$ and $A_R$ both $\Gamma$-algebraic over $A_0$, $A'$ a full $\Gamma$-field and $A$ the $\Gamma$-subfield of $A'$ generated by $A_L \cup A_R$. Then by remark~\ref{Galg predim}, we have $A_L \cup A_R \subs \Gcl^{A'}(A_0)$ and so $A \subs \Gcl^{A'}(A_0)$, so $A_0 \strong A$ is $\Gamma$-algebraic.
\end{proof}

Write $M_0$ (or $M_0(\Fbase)$) for the \Fraisse\ limit of $\Calg$. 

\begin{defn}\label{alg-sat defn}
A $\Gamma$-field $F$ strongly extending $\Fbase$ is \emph{$\aleph_0$-saturated for $\Gamma$-algebraic extensions over $\Fbase$} if whenever $\Fbase \strong A \strong F$ with $A$ finitely generated over $\Fbase$ and $A \xrightarrow{\strong} B$ is a finitely generated $\Gamma$-algebraic extension then $B$ embeds (necessarily strongly) into $F$ over $A$.
\end{defn}

\begin{prop}\label{uniqueness of M_0}
$M_0(\Fbase)$ is the unique countable full $\Gamma$-field strongly extending $\Fbase$ which is $\Gamma$-algebraic over $\Fbase$ and $\aleph_0$-saturated for $\Gamma$-algebraic extensions.

\end{prop}
\begin{proof}
Immediate from the uniqueness part of the amalgamation theorem and Proposition~\ref{Calg is amalg cat}. 
\end{proof}

\subsection{Purely $\Gamma$-transcendental extensions}\label{purely G-trans section}

In  contrast with $\Gamma$-algebraic extensions are those we will call purely $\Gamma$-transcendental extensions. We discuss amalgamation of these which gives rise to some variant constructions.
\begin{defn}
Let $A \strong B$ be a strong extension of $\Gamma$-fields. The extension is \emph{purely $\Gamma$-transcendental} if for all tuples $b$ from $\Gamma(B)$, either $\delta(b/A) > 0$ or $b \subs \Gamma(A)$.
\end{defn}

\begin{remark}
If $A \strong B$ is an extension of full $\Gamma$-fields then it is purely $\Gamma$-transcendental if and only if $A$ is $\Gamma$-closed in $B$.
\end{remark}

\begin{defn}
When $\Fbase$ is a full countable $\Gamma$-field, we define $\Ctrans(\Fbase)$ (usually abbreviated to $\Ctrans$) to be  the full subcategory of $\Cc$ consisting of the strong purely $\Gamma$-transcendental extensions of $\Fbase$.
\end{defn}

\begin{lemma}\label{full transc}
If $A \in \Ctrans$ then $\Afull \in \Ctrans$.
\end{lemma}
\begin{proof}
Consider the case when $(a_1,a_2) \in \Gamma(\Afull)\minus \Gamma(A)$ with $a_1 \in G_1(\Afull)$ algebraic over $A$. Since $A \strong \Afull$ we have $\td(a_2/A) = d$. If $\delta((a_1,a_2)/\Fbase) \le 0$ then $\td(a_1,a_2/\Fbase) \le d$, which implies that $\td(a_1/\Fbase) = 0$. Since $\Fbase$ is full, that implies $(a_1,a_2) \in \Gamma(\Fbase)$, a contradiction.

Replacing $A$ by the $\Gamma$-subfield of $\Afull$ generated by $A \cup \{(a_1,a_2)\}$ and iterating appropriately, we see that $\Afull \in \Ctrans$.
\end{proof}

We will show that $\Ctrans$ is an amalgamation category by showing that the free amalgam of purely $\Gamma$-transcendental extensions is purely $\Gamma$-transcendental, using a lemma on stable groups.

\begin{lemma}\label{Ziegler lemma}
Let $H$ be a commutative algebraic group defined over an algebraically closed field $C$. Suppose $a_1,a_2,a_3 \in H$ are pairwise algebraically independent over $C$ and $a_1+a_2+a_3 = 0$. Then there is a connected algebraic subgroup $U$ of $H$ and cosets $c_i + U$ defined over $C$ such that $a_i$ is a generic point of $c_i + U$ over $C$, for each $i=1, 2, 3$. In particular, $\td(a_i/C) = \dim U$ for each $i$.
\end{lemma}
\begin{proof}
This is the special case for algebraic groups of a result about stable groups due to Ziegler \cite[Theorem~1]{Ziegler06}.
\end{proof}

\begin{theorem}\label{Gammatr is amalg cat}
If $\Fbase$ is a full countable $\Gamma$-field then $\Ctrans$ and $\Ctransfull$ are amalgamation categories.
\end{theorem}
\begin{notation}
We write $\Mtr(\Fbase)$ for the Fra\"iss\'e limit in $\Ctrans(\Fbase)$.
\end{notation}
\begin{proof}[Proof of Theorem~\ref{Gammatr is amalg cat}]
Axioms AC1, AC3 and AC4 follow immediately from the fact that $\Ctrans$ and $\Ctransfull$ are full subcategories of $\Cc$. AC2 and AC6 are also immediate. It remains to prove AC5, the amalgamation property. 

Using Lemma~\ref{full transc}, the same argument as for Corollary~\ref{AP} allows us to reduce the amalgamation property for $\Ctrans$ to the amalgamation property for $\Ctransfull$. So suppose we have full $\Gamma$-fields
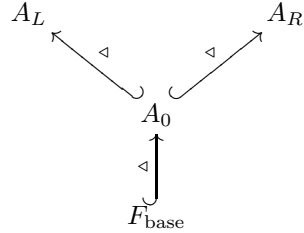
\[ \xymatrix{
    A_L & & A_R \\
    & A_0 \hstrarl[ul] \hstrar[ur] \\
    & \Fbase \hstrar[u] } \]
with $A_0$, $A_L$ and $A_R$ all purely $\Gamma$-transcendental extensions of $\Fbase$. Let $A_1$ be the free amalgam of $A_L$ and $A_R$ over $A_0$ as in the proof of Proposition~\ref{AAP}. We must show that $A_1$ is a purely $\Gamma$-transcendental extension of $\Fbase$.

So let $B$ be a $\Gamma$-subfield of $A_1$ properly containing $\Fbase$ and finitely generated over it. It remains to show that $\delta(B/\Fbase) \ge 1$. If $B \wedge A_R \neq \Fbase$ then we have 
\begin{eqnarray*}
\delta(B/\Fbase) & = & \delta(B/B\wedge A_R) + \delta(B\wedge A_R/\Fbase) \\
& \ge & \delta(B A_R / A_R) + \delta(B\wedge A_R/\Fbase) \qquad \text{ by submodularity} \\
& \ge & 0 + 1 = 1,
\end{eqnarray*}
the last line because $A_R \strong A_1$ and $A_R$ is purely $\Gamma$-transcendental over $\Fbase$. So in this case we are done, and similarly if $B \wedge A_L \neq \Fbase$.

So we may assume that $B \wedge A_R = B \wedge A_L = \Fbase$. Choose a basis $b = (b^1,\ldots,b^n)$ of  $B$ over $\Fbase$. For each $i$, choose $b^i_L \in \Gamma(A_L)$ and $b^i_R \in \Gamma(A_R)$ such that $b^i = b^i_L + b^i_R$. Let $B_L$ and $B_R$ be the $\Gamma$-field extensions of $\Fbase$ generated by $b_L \leteq (b^1_L,\ldots,b^n_L)$, and $b_R \leteq (b^1_R,\ldots,b^n_R)$ respectively.

We claim that $B_L \wedge A_R = \Fbase$ and that $A_L \wedge B_R = \Fbase$. To see this, suppose that $v \in \Gamma(B_L \wedge A_R) = \Gamma(B_L) \cap \Gamma(A_R)$. Since $v \in \Gamma(B_L)$ there are $s_i \in \kO$ and some $a \in \Gamma(\Fbase)$ such that $v = \sum_{i=1}^n s_i b^i_L + a$. Let $u_L = v-a = \sum_{i=1}^n s_i b^i_L$, let $u_R = \sum_{i=1}^n s_i b^i_R$, and let $u = u_L + u_R = \sum_{i=1}^n s_i b^i$. Then $v, a \in \Gamma(A_R)$, so $u_L \in \Gamma(A_R)$, and also $u_R \in \Gamma(A_R)$, hence  $u \in \Gamma(A_R)$. But $u \in \Gamma(B)$ and we have $B\wedge A_R = \Fbase$. So $u \in \Gamma(\Fbase)$ and so each $s_i = 0$, and thus $v \in \Gamma(\Fbase)$. So $B_L \wedge A_R = \Fbase$. The same argument shows that $A_L \wedge B_R = \Fbase$, and in particular $B_L \wedge B_R = \Fbase$.

Let $C$ be the $\Gamma$-subfield of $A_1$ generated by $B \cup B_L$, and note that it is also generated by $B_L \cup B_R$. We have $B \wedge B_L = \Fbase = B_L \wedge B_R$, so applying modularity of linear dimension to the squares
\[ \xymatrix{
    & C & & & & C\\
    B \hstrar[ur] & & B_L \hstrarl[ul] & \text{and} & B_L \hstrar[ur] & & B_R \hstrarl[ul] \\
    & \Fbase \hstrarl[ul] \hstrar[ur] & & & & \Fbase \hstrarl[ul] \hstrar[ur] } \]
we get 
\[\ldim_\kO(\Gamma(B_R)/\Gamma(\Fbase)) = \ldim_\kO(\Gamma(C)/\Gamma(B_L)) = \ldim_\kO(\Gamma(B)/\Gamma(\Fbase)) = n,\]
and so $b_R$ is $\kO$-linearly independent over $\Gamma(\Fbase)$, and hence, since $B_R \wedge A_0 = \Fbase$, over $\Gamma(A_0)$.
We have 
\[\td(b/\Fbase) \ge \td(b/A_0) \ge \td(b/A_0 b_L) = \td(b_R/A_0 b_L)   =  \td(b_R/A_0) \ge  dn\]
with the last three (in)equalities holding because $b=b_L + b_R$, $b_R$ is algebraically independent from $b_L$ over $A_0$, and because $A_0 \strong A_1$ and $b_R$ is $\kO$-linearly independent over $\Gamma(A_0)$. Similarly $\td(b/A_0) \ge \td(b_L/A_0) \ge dn$.

Suppose for a contradiction that $\delta(B/\Fbase) \le 0$. Then we must have 
\[\td(b/\Fbase) = \td(b/A_0) = \td(b_L/A_0) = \td(b_R/A_0) = dn,\]
and then we also have
\[\td(b,b_L/A_0) = \td(b,b_R/A_0)=\td(b_L,b_R/A_0) = \td(b,b_L,b_R/A_0) = 2dn\]
so $b,b_L,b_R$ are pairwise algebraically independent over $A_0$.

We apply Lemma~\ref{Ziegler lemma} with $H = G^n = G_1^n \cross G_2^n$, $a_1 = -b$, $a_2 = b_L$ and $a_3  = b_R$ to get a connected algebraic subgroup $U$ of $G^n$ of dimension $dn$ such that $b$ is in an $A_0$-coset of $U$. Since $\td(b/\Fbase) = \td(b/A_0) = \dim U$, and $\Fbase$ is an algebraically closed field, the coset is actually defined over $\Fbase$.

$G_1$ and $G_2$ are non-isogenous and so $U$ is of the form $U_1 \cross U_2$ where each $U_i$ is a connected subgroup of $G_i^n$. Since $\dim U = dn$, if $U_2 = G_2^n$ then $U_1$ is the trivial subgroup of $G_1^n$, so $\pi_1(b) \in G_1^n(A_0)$. But $A_0$ is a full $\Gamma$-field and so $b \in \Gamma(A_0)^n$ which contradicts $\td(b/A_0) = dn$ (and $n > 0$).  So $U_2$ must be a proper subgroup of $G_2^n$. Since $G_2$ is simple, it follows that $\pi_2(b)$ satisfies an $\OO$-linear equation $\sum_{i=1}^n s_i\pi_2(b^i) = c$ with $c \in G_2(\Fbase)$. Then, since $b \in \Gamma(B)^n$ and $\Fbase$ is a full $\Gamma$-field we have $\sum_{i=1}^n s_i b^i \in \Gamma(\Fbase)$, which contradicts $b$ being a basis for $B$ over $\Fbase$.

So we have $\delta(B/\Fbase) \ge 1$, and thus $A_1$ is a purely $\Gamma$-transcendental extension of $\Fbase$, as required.
\end{proof}


\section{Categoricity}\label{categoricity section}

\subsection{Quasiminimal pregeometry structures}

This definition of quasiminimal pregeometry structures comes from \cite{BHHKK14}.
\begin{defn}
  Let $M$ be an $L$-structure for a countable language $L$, equipped with a
  pregeometry $\cl$ (or $\cl_M$ if it is necessary to specify $M$). Write $\qftp$ for the quantifier-free $L$-type. We say that
  $M$ is a \emph{quasiminimal pregeometry structure} if the following hold:
  \begin{enumerate}[QM1.]
    \item The pregeometry is determined by the language. That is, if $a$, $a'$ are singletons, $b$, $b'$ are tuples,
      $\qftp(a,b) = \qftp(a',b')$, and $a \in \cl(b)$ then $a'
      \in \cl(b')$.
    \item $M$ is infinite-dimensional with respect to $\cl$.
    \item (Countable closure property) If $A \subs M$ is finite then $\cl(A)$
      is countable.
    \item (Uniqueness of the generic type) Suppose that $C, C' \subs M$ are
      countable closed subsets, enumerated such that $\qftp(C)=\qftp(C')$. If $a
      \in M \minus C$ and $a' \in M \minus C'$ then  $\qftp(C,a) = \qftp(C',a')$
      (with respect to the same enumerations for $C$ and $C'$).
    \item ($\aleph_0$-homogeneity over closed sets and the empty set)
      \ \\ Let $C, C' \subs M$ be countable closed subsets or empty,
      enumerated such that $\qftp(C)=\qftp(C')$, and let $b,b'$ be finite
      tuples from $M$ such that $\qftp(C,b) = \qftp(C',b')$, and let $a
      \in \cl(C,b)$. Then there is $a'\in M$ such that $\qftp(C,b,a) =
      \qftp(C',b', a')$.
  \end{enumerate}
  We say $M$ is a \emph{weakly quasiminimal pregeometry structure} if it
  satisfies all the axioms except possibly QM2.
\end{defn}

\begin{defn}
Given $M_1$ and $M_2$ both weakly quasiminimal pregeometry $L$-structures, we
say that an $L$-embedding $\theta: M_1 \into M_2$ is a \emph{closed embedding}
if for each $A \subs M_1$ we have $\theta(\cl_{M_1}(A))=\cl_{M_2}(\theta(A))$. In particular, $\theta(M_1)$ is closed in $M_2$ with respect to $\cl_{M_2}$. We write $M_1 \closed M_2$ for a closed embedding.
\end{defn}

\begin{defn}
Given a quasiminimal pregeometry structure $M$, let $\K(M)$ be the smallest class of $L$-structures which
contains $M$ and all its closed substructures and is closed under isomorphism and under taking unions of directed systems of closed embeddings. We call any class of the form $\K(M)$  a \emph{quasiminimal class}.
\end{defn}

The purpose of these definitions is the categoricity theorem, which is Theorem~2.3 in \cite{BHHKK14}.
\begin{fact}\label{cat theorem}
  If $\K$ is a quasiminimal class then every structure $A \in \K$ is a weakly
  quasiminimal pregeometry structure, and up to isomorphism there is exactly one structure in $\K$ of each
  cardinal dimension. In particular, $\K$ is uncountably categorical.
  Furthermore, $\K$ is the class of models of an $\Looq$ sentence.
\end{fact}

We will verify axioms QM1--QM5 for the \Fraisse\ limits we constructed. We first make some general observations which simplify what we have to verify.
\begin{prop}\label{new QM axioms}
Suppose that $M$ is a countable $L$-structure. Then it satisfies QM1--QM5 if and only if it satisfies the following axioms.
 \begin{enumerate}[\textup{QM1$'$.}]
    \item If $a$ and $b$ are finite tuples and $\qftp(a) = \qftp(b)$ then $\dim(a) = \dim(b)$.
    \item[\textup{QM2.}] $M$ is infinite-dimensional with respect to $\cl$.
    \item[\textup{QM4.}] (Uniqueness of the generic type) 
    \item[\textup{QM5a.}] ($\aleph_0$-homogeneity over the empty set)\\
If $a$ and $b$ are finite tuples from $M$ and $\qftp(a) = \qftp(b)$ then there is $\theta \in \Aut(M)$ such that $\theta(a) = b$.
    \item[\textup{QM5b.}] (Non-splitting over a finite set) \\
    If $C \closed M$ and $b \in M$ is a finite tuple then there is a finite tuple $c \in C$ such that $\qftp(b/C)$ does not split over $c$. That is, for all finite tuples $a,a' \in C$, if $\qftp(a/c) = \qftp(a'/c)$ then $\qftp(a/cb) = \qftp(a'/cb)$.    
  \end{enumerate} 
\end{prop}

\begin{proof}
QM1$'$ is equivalent to QM1, because $a \in \cl(b)$ if and only if $\dim(a,b) = \dim(b)$, so if quantifier-free types characterize the dimension they also characterize the closure operation, and vice versa.

QM3, the countable closure property, is immediate for a countable $M$.

Axiom QM5 with $C = \emptyset$ gives a back-and-forth condition which is equivalent to $\aleph_0$-homogeneity using the standard back-and-forth argument together with QM1 and QM4. Since $M$ is countable, the back-and-forth construction gives QM5a. The converse is immediate.

Then \cite[Corollary~5.3]{BHHKK14} shows the case of QM5 with $C$ closed is equivalent to QM5b.
\end{proof}

\begin{remark}
All the axioms refer to quantifier-free types with respect to a particular language, and from QM5a we get the conclusion that if two finite tuples from $M$ have the same quantifier-free type then they actually have the same complete type (even the same $L_{\infty,\omega}$-type, and furthermore they lie in the same automorphism orbit, that is, they have the same Galois-type). Since $M$ is not necessarily a saturated model of its first-order theory, it does not follow that every definable set is quantifier-free definable. Nonetheless, identifying the language which works allows us to understand the types which are realised in $M$.
\end{remark}

\subsection{Verification of the quasiminimal pregeometry axioms}

Recall that our language of $\Gamma$-fields is $L_\Gamma = \langle +,\cdot,-,\Gamma, (c_a)_{a\in \bk} \rangle$, where $\Gamma$ is a relation symbol of suitable arity to denote a subset of $G$. We start by defining the expansion \Lqe\ of $L_\Gamma$  in which we will have the form of quantifier-elimination described in the previous remark.

Let $W$ be any subvariety of $G^n \cross \A^r$ defined over $\Fbase$, for some $n, r \in \N$. (It suffices to consider those $W$ which are the graphs of rational maps $f : W' \to \A^r$, with $W' \subs G^n$.) Let $\phi_W(x,y)$ name the subset of $G(M)^n \cross M^r$ given by
\[(x,y) \in W  \quad\&\quad x \in \Gamma^n  \quad \&\quad x \text{ is $\OO$-linearly independent over } \Gamma(\Fbase)\]
and let $\psi_W(y)$ be the formula $\exists x \phi_W(x,y)$.

\begin{defn}
 We define \Lqe\ to be the expansion of $L_\Gamma$ by parameters for $\Fbase$ and relation symbols for all the formulas $\phi_W(x,y)$ and $\psi_W(y)$.
\end{defn}

\begin{remark}
Note that the formulas $\phi_W(x,y)$ are are always expressible in $\Loo(L_\Gamma)$ (with parameters in $\Fbase$), so a priori this is an expansion of $L_\Gamma(\Fbase)$ by $\Loo$-definitions. However, if the ring $\OO$ and its action on $G$ are definable and $\Gamma(\Fbase)$ is either of finite rank (which is true for example in pseudo-exponentiation) or is otherwise an $L_\Gamma$-definable set, then $\Lqe$ is just an expansion of $L_\Gamma(\Fbase)$ by first-order definitions.
\end{remark}

For the rest of this section we use tuples both from the field sort of a model $M$ and from $\Gamma(M)$, so to distinguish them we will use Latin letters for tuples from $M$ and Greek letters for tuples from $\Gamma(M)$.
\begin{theorem}\label{Fraisse limit is QPS}
Take $M$ to be either $M(\Fbase)$ or $\Mtr(\Fbase)$, the latter only if $\Fbase$ is a full $\Gamma$-field. Then, considered in the language $\Lqe$ and equipped with $\Gammacl$, $M$ is a quasiminimal pregeometry structure.
\end{theorem}
\begin{proof}
We verify the axioms from Proposition~\ref{new QM axioms}. The main difficulty is that the axioms refer to the field sort whereas the construction of $M$ was done in the sort $\Gamma$, and there is no canonical way to go from one sort to the other in either direction. However the sort $\Gamma$ has rank $d$ with respect to the pregeometry $\Gcl$, so as we want to include the case $d>1$ we have to verify the axioms with respect to the field sort.

\medskip

First we prove QM2. For any $n \in \N$, there is a strong $\Gamma$-field extension $A_n$ of $\Fbase$ generated by a tuple $\alpha \in \Gamma(A_n)^n$ such that $\alpha$ is generic in $G^n$ over $\Fbase$. Then $\delta(\alpha/\Fbase) = dn$. This $A_n$ embeds strongly in $M$ by the universality property of the Fra\"iss\'e limit, so $\Gammadim^M(\alpha) = dn$ by Lemma~\ref{dim and predim}. Hence $M$ is infinite-dimensional.

\medskip 

Now we prove QM1$'$ and QM5a together.

Suppose $a,b \in M^r$ with $\qftp_{\Lqe}(a) = \qftp_{\Lqe}(b)$. Choose a strong $\Gamma$-subfield $A \strong M$ which is a finitely-generated extension of $\Fbase$ such that $a \in A^r$ and $\delta(A/\Fbase)$ is minimal such. Let $\alpha \in \Gamma(A)^n$ be a good basis for $A$ over $\Fbase$, such that $a$ is in the field $\Fbase(\alpha)$, let $W = \loc(\alpha,a/\Fbase)$, and let $V = \loc(\alpha/\Fbase)$. Then $M \models \psi_W(a)$, so also $M \models \psi_W(b)$. So there is $\beta \in \Gamma(M)^n$ such that $M \models \phi_W(\beta,b)$. In particular, $\beta \in V \cap \Gamma(M)^n$, $\kO$-linearly independent over $\Gamma(\Fbase)$. We claim that $W = \loc(\beta,b/\Fbase)$. Suppose not, so $W' \leteq \loc(\beta,b/\Fbase)$ is a proper subvariety of $W$. We have $M \models \psi_{W'}(b)$, so since $a$ and $b$ have the same quantifier-free $\Lqe$-type, $M \models \psi_{W'}(a)$. So there is some $\alpha' \in \Gamma(M)^n$ such that $M \models \phi_{W'}(\alpha',a)$. $W$ is irreducible over $\Fbase$, so $\dim W' < \dim W$. Since $a$ is rational over $\Fbase(\alpha)$, $\dim W = \dim V$, and so $\alpha'$ lies in a subvariety $V'$ of $V$ with $\dim V' < \dim V$. But then setting $A' = \gen{\Fbase,\alpha'}$ we have $a \in A'^r$ and $\delta(A'/\Fbase) = \delta(\alpha'/\Fbase) = \dim V' - \ldim_\kO(\alpha'/\Gamma(\Fbase)) < \delta(A/\Fbase)$, which contradicts the choice of $A$. Thus $\loc(\beta,b/\Fbase) = W$, and in particular $\loc(\beta/\Fbase) = V$. Let $B = \gen{\Fbase,\beta}$. We further deduce that $B \strong M$, since if not the same proof would show that $\delta(A/\Fbase)$ were not minimal.

By Lemma~\ref{dim and predim}, we have $\Gdim(A) = \delta(A/\Fbase) = \delta(B/\Fbase) = \Gdim(B)$.
Now $\Gdim(a) = \Gdim(A)$ by the minimality of $\delta(A/\Fbase)$, since $A$ could be taken within $\Gcl(a)$,
and $\Gdim(b) \leq \Gdim(B)$, so $\Gdim(b) \leq \Gdim(a)$.
By symmetry, $\Gdim(a) = \Gdim(b)$, so QM1$'$ is proved.

Since $\alpha$ is a good basis, there is an isomorphism of $\Gamma$-fields $\theta_0: A \to B$ over $\Fbase$, with $\theta_0(\alpha) = \beta$. Then also $\theta_0(a) = b$. Since $M$ is $\Cfg$-homogeneous (or $\Ctrfg$-homogeneous), $\theta_0$ extends to an automorphism $\theta$ of $M$. That proves QM5a.

\medskip

QM4: Suppose that $C_1, C_2 \closed M$ with the same quantifier-free \Lqe-type according to some enumeration, and let $\theta : C_1 \iso C_2$ be the isomorphism given by the enumeration. Suppose also that $b_1 \in M \minus C_1$ and $b_2 \in M \minus C_2$.

Using Lemma~\ref{changing sorts} we get $\beta_1,\beta_2 \in \Gamma(M)$ such that $b_i \in C_i(\beta_i)$ and $\loc(\beta_1,b_1/C_1)$ is defined over $\bk$ and is equal to $\loc(\beta_2,b_2/C_2)$. Also, setting $B_i \leteq \gen{C_i,\beta_i}$ we have $B_i \strong M$ and $\beta_i$ is a good basis for $B_i$ over $C_i$. By the definition of a good basis, the isomorphism $\theta$ extends to $\theta_1: B_1 \iso B_2$ with $\theta_1(\beta_1) = \beta_2$ and hence $\theta_1(b_1) =b_2$.

Let $\Fbase \strong A_1 \strong C_1$ with $A_1$ finitely generated over $\Fbase$, and let $A_2 = \theta(A_1)$. Then $\theta_1$ restricts to an isomorphism $\theta_0: \gen{A_1\beta_1} \isom \gen{A_2\beta_2}$.  Also $\gen{A_i\beta_i} \strong  M$ since $\delta(\gen{A_i\beta_i}/A_i) = 1$
	and $\beta_i \notin \Gammacl(A_i)$. Since $M$ is $\Cfg$-homogeneous (or $\Ctrfg$-homogeneous), $\theta_0$ extends to an automorphism of $M$. So $\qftp_{\Lqe}(A_1 b_1) = \qftp_{\Lqe}(A_2 b_2)$ and thus, as $A_1$ ranges over strong $\Gamma$-subfields of $C_1$ finitely generated over $\Fbase$, we deduce that $\qftp_{\Lqe}(C_1 b_1) = \qftp_{\Lqe}(C_2 b_2)$ as required.
\medskip

\medskip

QM5b: Let $C \closed M$ and let $b\in M$ be a finite tuple. Let $B$ be a finitely generated $\Gamma$-field extension of $C$ such that $B \strong M$ and $b \in B$, and let $\beta \in \Gamma(B)^n$ be a good basis for $B$ over $C$ with $b \in C(\beta)$. 

Now choose a finitely generated $\Gamma$-field extension $C_0$ of $\Fbase$ in $C$ with $C_0 \strong C$, and a good
basis $\gamma$ for $C_0$, such that $\loc(\beta,b/C)$ is defined over $\Fbase(\gamma)$.

Suppose that finite tuples $a,a' \in C$ have $\qftp_{\Lqe}(a/\gamma) = \qftp_{\Lqe}(a'/\gamma)$. By QM5a, there is a $\Gamma$-field automorphism $\theta\in\Aut(M/\Fbase(\gamma))$ such that $\theta(a)=a'$. Let $A$ be a strong $\Gamma$-subfield of $\Gcl(C_0,a)$ which is finitely generated over $C_0$ and contains $a$, and let $A' = \theta(A)$. Then $A' \strong C$.

Let $V = \loc(\beta/C)$. Then $\loc(\beta/A) = \loc(\beta/A') = V$ because $V$ is defined over $\Fbase(\gamma)$.  So, since $\beta$ is a good basis, the isomorphism $\theta_0:A \iso A'$ extends to $\theta_1:\gen{A\beta} \iso \gen{A'\beta}$.

We claim that $\gen{A \beta} \strong B$. To see this, suppose that $X \subs B$ is a finitely generated extension of $\gen{A \beta}$, and let $X_0 = X \wedge C$. Let $\xi$ be a basis of $X_0$ over $A$. Then $\xi \cup \beta$ is a basis for $X$ over $A$, since $\beta$ is a basis for $B$ over $C$ and hence for $X$ over $X \wedge C = X_0$. So
\begin{eqnarray*}
\delta(X/A\beta) & = & \td(\xi/A\beta) - d \ldim_\kO(\xi/\Gamma(A),\beta) \\
 & = & \td(\xi/A) - d \ldim_\kO(\xi/\Gamma(A)) \\
 & = & \delta(\xi/A) \ge 0
\end{eqnarray*}
because $\beta$ is algebraically and linearly independent from $C$ over $A$ and $\xi \in C$. Since $B \strong M$ we have $\gen{A\beta} \strong M$. The same argument shows that $\gen{A'\beta} \strong M$.

Thus, since $M$ is $\Cfg$-homogeneous (or $\Ctrfg$-homogeneous), $\theta_1$ extends to an automorphism $\theta_2$ of $M$. Now $\theta_2$ fixes $b$ and $\gamma$ and $\theta_2(a) = a'$, so $\qftp_{\Lqe}(a/b \gamma) = \qftp_{\Lqe}(a'/b\gamma)$. Taking $c = \gamma$, considered as a tuple from the field sort of $C$, we see that $\tp(b/C)$ does not split over $c$, as required.
\end{proof}

\begin{remark}
A more complete analysis of splitting for pseudo-exponentiation was carried out in the PhD thesis of Robert Henderson \cite{Henderson14}.
\end{remark}

We conclude this section by showing that the $\Gamma$-algebraic types over finite tuples are isolated.
\begin{prop}\label{isozof types}
Suppose that $a, b$ are finite tuples in $M$ and that $b \in \Gcl^M(a)$. Then $\tp(b/a)$ is isolated by an $\Lqe$-formula.
\end{prop}
\begin{proof}
Choose a finitely generated $\Gamma$-field $B \strong M$ with $B \subs \Gcl(a)$, and a good basis $\beta$ for $B$ such that $a,b \in \Fbase(\beta)$. Let $W = \loc(\beta,a^\frown b/\Fbase)$.

Then $M \models \phi_W(\beta,a^\frown b)$ and $M \models \psi_W(a^\frown b)$. Suppose $M \models \psi_W(a^\frown c)$. Then there is a tuple $\gamma$ from $\Gamma(M)$ such that $M \models \phi_W(\gamma,a^\frown c)$. So $\loc(\gamma/\Fbase) \subs V$ but $\Fbase \strong M$ and $\ldim_\kO(\gamma/\Gamma(\Fbase)) = \dim V$ by the definition of $\phi_W$, so $\gamma$ is generic in $V$ over $\Fbase$. Thus $\loc(\gamma/\Fbase) = \loc(\beta/\Fbase)$ so, since $\beta$ is a good basis, the $\Gamma$-field $C$ generated by $\gamma$ is isomorphic to $B$ via an isomorphism $\theta : B \to C$ such that $\theta(\beta) = \gamma$, and then necessarily $\theta(a) = a$ and $\theta(b) = c$.

Using Lemma~\ref{dim and predim} repeatedly,
\[\delta(C/\Fbase) = \delta(B/\Fbase) = \Gdim(B) = \Gdim(a) \le \Gdim(C)\]
and so $\delta(C/\Fbase) = \Gdim(C)$, and $C \strong M$. Thus $\theta$ extends to an automorphism of $M$, so $\tp(c/a) = \tp(b/a)$, so the formula $\psi_W(a^\frown x)$ isolates $\tp(b/a)$.
\end{proof}


\section{Classification of strong extensions}\label{rotund section}

We next give a classification of the finitely generated strong extensions. In the next section we will use it to give axiomatizations of the classes of $\Gamma$-closed fields we have constructed, generalizing Zilber's axioms for pseudo-exponentiation.

Since $G$ is an $\OO$-module, each matrix $M \in \Mat_n(\OO)$ defines an $\OO$-module homomorphism $\ra{G^n}{M}{G^n}$ in the usual way. If $V \subs G^n$, we write $M\cdot V$
for its image. Note that if $V$ is a subvariety of $G^n$ then $M\cdot V$ is a constructible set, and since the $\OO$-module structure is defined over $\bk$, if $V$ is defined over $A$ then $M\cdot V$ is defined over $\bk \cup A$. If $V$ is irreducible then $M\cdot V$ is also irreducible.

We have $G^n  = (G_1 \cross G_2)^n$ and we write $x_1,\ldots,x_n$ for the coordinates in $G_1$ and $y_1,\ldots,y_n$ for the coordinates in $G_2$.
\begin{defn}
 Let $V$ be an irreducible subvariety of $G^n$.
 Then $V$ is \emph{$G_1$-free} if $V$ does not lie inside any subvariety defined by an equation $\sum_{j=1}^n r_jx_j = c$ for any $r_j \in \OO$, not all zero, and any $c \in G_1$.  We define \emph{$G_2$-free} the same way. We say $V$ is \emph{free} if it is both $G_1$-free and $G_2$-free.
 
$V$ is \emph{rotund} (for $G$ as an $\OO$-module) if for every matrix $M \in \Mat_n(\OO)$ we have
  \[\dim (M\cdot V) \ge d \rk M\]
where $\dim$ means dimension as an algebraic variety or constructible set, $\rk M$ is the rank of the matrix $M$, and recall that $d = \dim G_1$.

$V$ is \emph{strongly rotund} if for every \emph{non-zero} matrix $M \in \Mat_n(\OO)$ we have
  \[\dim (M\cdot V) > d \rk M.\]

A reducible subvariety $V$ of $G^n$ is defined to be free / rotund / strongly rotund if at least one of its (absolutely) irreducible components is free / rotund / strongly rotund respectively. If we say that such a $V$ is \emph{free and (strongly) rotund} then we mean that the same irreducible component is free and (strongly) rotund.

\end{defn}
So $V$ is free if it is ``free from $\OO$-linear dependencies'', and it is rotund if all its images under suitable homomorphisms are of large dimension.

\begin{lemma}\label{O-linear subgroups}
An irreducible subvariety $V \subs G^n$ is $G_2$-free if and only if $\pi_2(V)$ does not lie in a coset of a proper algebraic subgroup of $G_2^n$.
If $\OO = \End(G_1)$ then $V$ is $G_1$-free if and only if $\pi_1(V)$ does not lie in a coset of a proper algebraic subgroup of $G_1^n$.
\end{lemma}
\begin{proof}
  This is an immediate consequence of Lemma~\ref{subgroupsEndomorphisms}(i).
\end{proof}

\begin{prop}\label{rotund prop}
Suppose that $A$ is a full $\Gamma$-field, $A \subs B$ is a finitely generated extension of $\Gamma$-fields, preserving the kernels, and that $b \in \Gamma(B)^n$ is a basis for the extension. Let $V = \loc(b/A)$.

Then $V$ is free. Furthermore the extension is strong if and only if $V$ is rotund, and it is purely $\Gamma$-transcendental if and only if $V$ is strongly rotund.
\end{prop}
\begin{proof}
If $V$ is not $G_1$-free then, writing $b = (b_1^1,\ldots,b_1^n,b_2^1,\ldots,b_2^n) \in G_1^n \cross G_2^n$ we have $
\sum_{j=1}^n r_j b_1^j = c_1 \in G_1(A)$. Let $c_2 = \sum_{j=1}^n r_j b_2^j$. Then $(c_1,c_2) \in \Gamma(B)$ and since $A$ is full and the extension preserves the kernels we have $(c_1,c_2) \in \Gamma(A)$. That contradicts $b$ being a basis for the extension. So $V$ is $G_1$-free and, symmetrically, $G_2$-free.

For  $M \in \Mat_n(\OO)$ we have $M \cdot b \in \Gamma(B)^n$ with $\ldim_\kO (M\cdot b/\Gamma(A)) = \rk M$. Furthermore, every finite tuple from $\Gamma(B)$ generates the same $\Gamma$-field extension of $A$ as some tuple $M \cdot b$, because $b$ is a basis. Thus the extension is strong if and only if for all $M$ we have $\td(M \cdot b/A) \ge d \rk M$, if and only if for all $M$ we have $\dim (M\cdot V)^{\Zar} \ge d \rk M$, if and only if $V$ is rotund.

Similarly any finite tuple from $\Gamma(B)$ which is not in $\Gamma(A)$ generates the same extension of $A$ as a tuple $M \cdot b$ for some non-zero matrix $M$, and $V$ is strongly rotund  if and only if all such tuples have $\delta(M \cdot b/A) > 0$.
\end{proof}

\begin{corollary}\label{rotund cor}
Suppose that $A \strong B$ is a finitely generated strong extension of essentially finitary $\Gamma$-fields, that $\Afull \wedge B = A$, that $b \in \Gamma(B)^n$ is a basis for the extension, and that $V = \loc(b/A)$. Then $V$ is free and rotund, and it is strongly rotund if and only if $B$ is a purely $\Gamma$-transcendental extension of $A$.
\end{corollary}
\begin{proof}
First note that, since $A$ and $B$ are essentially finitary, by Theorem~\ref{full extension}, $\Bfull$ is uniquely determined up to isomorphism and $\Afull$ is uniquely determined as a $\Gamma$-subfield of $\Bfull$, so the condition that $\Afull \wedge B = A$ makes unambiguous sense. Now the proof of Proposition~\ref{rotund prop} goes through with this weaker condition in place of $A = \Afull$.
\end{proof}


\section{Axiomatization of $\Gamma$-closed fields}\label{axiomatization section}

Recall we have fixed a ring $\OO$, algebraic $\OO$-modules $G_1$ and $G_2$, both of dimension $d$, defined over a countable field $\bk$, we have $G = G_1 \cross G_2$, and we consider structures in the language  $L_\Gamma = \langle +,\cdot,-,\Gamma, (c_a)_{a\in \bk} \rangle$, where $\Gamma$ is a relation symbol of appropriate arity to denote a subset of $G$. We are given an essentially finitary $\Gamma$-field $\Fbase$ containing $\bk$, of type (EXP), (COR), or (DEQ). We add parameters for $\Fbase$ to the language to get a language $L_\Fbase$. We also have an expanded language $\Lqe$.

\begin{defn}
A model in the quasiminimal class $\K(M(\Fbase))$ will be called a \emph{$\Gamma$-closed field} (with the countable closure property, on the base $\Fbase$).
\end{defn}

\begin{theorem}\label{main theorem}
An $L_\Fbase$ structure $F$ is a  $\Gamma$-closed field if and only if it satisfies the following list of axioms which we denote by $\GCFccp(\Fbase)$. \begin{description}
 \item[1. Full $\Gamma$-field] $F$ is an algebraically closed field containing $\bk$ and $\Gamma(F)$ is an $\OO$-submodule of $G(F)$ such that the projections of $\Gamma(F)$ to $G_1(F)$ and $G_2(F)$ are surjective.
 
 \item[2. Base and kernels] $F$ satisfies the full atomic diagram of $\Fbase$. (In some examples we will discuss how this can be weakened.) Also $\ker_i(F) = \ker_i(\Fbase)$ for $i=1,2$.
 
 \item[3. Predimension inequality (generalised Schanuel property)]
 The \emph{predimension function} 
\[\delta(x/\Fbase) \leteq \td(x/\Fbase)- d \ldim_\kO(x/\Gamma(\Fbase))\]
satisfies $\delta(x/\Fbase) \ge 0$ for all tuples $x$ from $\Gamma(F)$. 

\item[4. Strong $\Gamma$-closedness]
For every irreducible subvariety $V$ of $G^n$ defined over $F$ and of dimension $dn$, which is free and rotund for the $\OO$-module structure on $G$, and every finite tuple $a$ from $\Gamma(F)$, there is $b \in V(F) \cap \Gamma(F)^n$ such that $b$ is $\kO$-linearly independent over $\Gamma(\Fbase) \cup a$ (that is, no non-zero $\kO$-linear combination of the $b_i$ lies in the $\kO$-linear span of $\Gamma(\Fbase) \cup a$). 

\item[5. Countable Closure Property] For each finite subset $X$ of $F$, the $\Gamma$-closure $\Gammacl^F(X)$ of $X$ in $F$ is countable. 
\end{description}
\end{theorem}

Observe that if $F$ is a $\Gamma$-field for the $\OO$-module $G$, and in particular if it is full so satisfies axiom 1, then it satisfies axioms 2 and 3 if and only if $\Fbase \strong F$.

\begin{lemma}\label{SGClosed lemma}
If $F$ satisfies axioms 1--3 then it also satisfies axiom 4 if and only if it is $\aleph_0$-saturated for $\Gamma$-algebraic extensions (in the sense of Definition~\ref{alg-sat defn}).
\end{lemma}
\begin{proof}
First assume that $F$ satisfies axioms 1--4. Suppose $A \strong F$ is finitely generated over $\Fbase$ and $A \strongembed B$ is a finitely generated $\Gamma$-algebraic extension. We have assumed that $\Fbase$ is essentially finitary, so $\Afull$ and $\Bfull$ are unique up to isomorphism as extensions of $A$ and $B$ respectively, so $\Afull$ embeds (strongly) into $\Bfull$. Choose an embedding. Since $\Afull$ embeds in $F$ we have $\Afull \wedge B$ embedding (strongly) into $F$, as summarised in the diagram below.
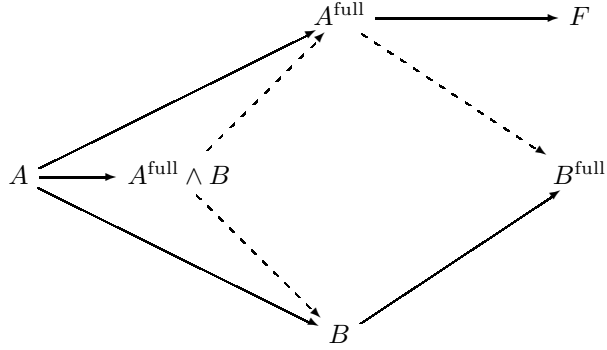
\[\begin{diagram}
&&&& \Afull & \rTo & & F\\
& &&\ruTo(4,2)  \ruDashto & & \rdDashto(3,2)\\
A & \rTo & \Afull \wedge B &&&&& \Bfull\\
& \rdTo(4,2) & & \rdDashto & & & \ruTo(3,2)\\
&&&& B
\end{diagram}\]

So it remains to embed $B$ in $F$ over $\Afull \wedge B$, so we may assume $A = \Afull \wedge B$. Let $a$ be a basis for $A$ over $\Fbase$ and let $b \in \Gamma(B)^n$ be a good basis for $B$ over $A$, which exists by Proposition~\ref{good bases exist}. Let $V = \loc(b/A)$, a subvariety of $G^n$. Then $V$ is free and rotund by Corollary~\ref{rotund cor}. Since $B$ is $\Gamma$-algebraic over $A$ we have $\delta(b/A) = 0$, so $\dim V = dn$. 

 Then, by axiom 4 applied to an absolutely irreducible component of $V$, there is $c \in \Gamma(F)^n \cap V(F)$, $\kO$-linearly independent
over $\Gamma(\Fbase) \cup \{a\}$. Since $A \strong F$ we have $\delta(c/A) \ge 0$, so $\td(c/A) = dn = \dim V$. Thus $c$ is generic in $V$ over $A$. Let $C$ be the $\Gamma$-subfield of $F$ generated by $A$ and $c$. Then $c$ is a good basis of $C$ over $A$ because this is a property of $\loc(c/A)$, that is, of $V$. So, by the definition of a good basis, $C$ is isomorphic to $B$ over $A$. So $F$ is $\aleph_0$-saturated for $\Gamma$-algebraic extensions over $\Fbase$.

For the converse, suppose that $F$ is $\aleph_0$-saturated for $\Gamma$-algebraic extensions over $\Fbase$. Let $V$ be a free and rotund absolutely irreducible subvariety of $G^n$ which is defined over $F$ and of dimension $dn$, and let $a$ be a finite tuple from $\Gamma(F)$. Extending $a$ if necessary, we may assume that $A = \gen{\Fbase, a} \strong F$ and that $V$ is defined over $a$. 

Consider a $\Gamma$-field extension $B$ of $A$, generated by a tuple $b \in \Gamma(B)$ such that $\loc(b/A) = V$. By Proposition~\ref{rotund prop} the extension is strong. Since $V$ is free, $\ldim_\kO(b/\Gamma(A)) = n$ and so $\delta(b/A) = \dim V - dn = 0$. So $B$ is a 
$\Gamma$-algebraic extension.  Thus $B$ embeds into $F$ over $A$ and so we have $b \in V(F) \cap \Gamma(F)^n$ which is $\kO$-linearly independent over $\Gamma(\Fbase) \cup \{a\}$ as required. 
\end{proof}

\begin{proof}[Proof of Theorem~\ref{main theorem}]

Suppose $F$ is a $\Gamma$-closed field. Then, by definition, $F$ is (isomorphic to) a closed substructure of the canonical model $M$ or 
is obtained from $M$ (and its closed substructures) as the union of a directed system of closed embeddings. If $F$ is a closed substructure of $M$ then certainly it is a full $\Gamma$-field strongly extending $\Fbase$, so it satisfies axioms 1--3, and it is countable so satisfies axiom 5. 

For axiom 4, suppose $A\strong F$ is finitely generated and $A \xrightarrow{\strong} B$ is a finitely generated and $\Gamma$-algebraic extension. Since $M$ is $\Cfg$-saturated, $B$ embeds strongly into $M$ over $A$ and since $F$ is closed in $M$, $B \subs F$. So by Lemma~\ref{SGClosed lemma}, $F$ satisfies axiom 4.

So closed substructures of $M$ satisfy axioms 1--5. Axioms 1--4 are preserved under unions of directed systems of strong embeddings, and all the axioms are preserved under unions of directed systems of closed embeddings, hence all $\Gamma$-closed fields satisfy all 5 axioms of $\GCFccp(\Fbase)$.

Suppose now that $F$ satisfies axioms 1--5. Since it satisfies axioms 1--3, we have the pregeometry $\Gammacl^F$ on $F$. If $F_0$ is a finite-dimensional substructure of $F$ then $F_0$ satisfies axioms 1--3 and 5 immediately and, using Lemma~\ref{SGClosed lemma}, also axiom 4.  Let $\abar$ be a $\Gcl^F$-basis for $F_0$. Using Lemma~\ref{changing sorts}, for each $a_i \in \abar$, choose $\alpha_i \in \Gamma(F_0)$, interalgebraic with $a_i$ over $\Fbase$. Let $C = \gen{\Fbase, \alpha_1,\ldots,\alpha_n}$. Then $F_0$ is $\Gamma$-algebraic over $C$ and is saturated for $\Gamma$-algebraic extensions so, by Proposition~\ref{uniqueness of M_0}, $F_0 \iso M_0(C)$. Now choose an embedding of $C$ into $M$ and note that $\Gcl^M(C)$ is also $\Gamma$-algebraic over $C$ and is saturated for $\Gamma$-algebraic extensions so is also isomorphic to $M_0(C)$. Hence $F_0$ is a $\Gamma$-closed field.


Now $F$ is the union of the directed system of all its finite-dimensional closed substructures, which by CCP are countable, and the class of $\Gamma$-closed fields is closed under such unions by definition, hence $F$ is a $\Gamma$-closed field.
\end{proof}

We can now prove Theorem~\ref{main Gamma theorem}.
\begin{proof}[Proof of Theorem~\ref{main Gamma theorem}]
By Theorem~\ref{Fraisse limit is QPS}, $M(\Fbase)$ is a quasiminimal pregeometry structure, so by Fact~\ref{cat theorem} the class $\K(M(\Fbase))$ is uncountably categorical and every model is quasiminimal. By Theorem~\ref{main theorem}, the list of axioms $\GCFccp(\Fbase)$ axiomatizes the class $\K(M(\Fbase))$.
\end{proof}

\begin{remarks}
\begin{enumerate}
\item It is easy to show that axioms 1--4 are $\Loo$-expressible, and axiom 5 is expressible as an $\Looq$-sentence.

\item  If we add an (\Loo-expressible) axiom stating that $F$ is infinite dimensional to axioms 1--4, the only countable model is $M$ and so we get an $\aleph_0$-categorical, and hence complete, $\Loo$-sentence.
\end{enumerate}
\end{remarks}


\section{Specific applications of the general construction}\label{applications section}
We list several instances of $\Gamma$-fields that are of interest, starting with the original example.

\subsection{Pseudo-exponentiation}

We take $\bk = \Q$, $G_1 = \Ga$ and $G_2 = \Gm$. Take $\OO = \Z$. Let $\tau$ be transcendental, and take $\Fbase$ to be the field $\Qab(\tau)$, where $\Qab$ is the extension of $\Q$ by all roots of unity. For each $m \in \N^+$, choose a primitive $m^{\mathrm{th}}$ root of unity $\omega_m$, such that for all $m,n \in \N^+$ we have $(\omega_{mn})^n = \omega_m$. We take $\Gamma(\Fbase)$ to be the graph of a homomorphism from the $\Q$-linear span of $\tau$ to the roots of unity such that $\tau/m \mapsto \omega_m$ for each $m \in \N^+$.
(This $\Fbase$ is called $\mathrm{SK}$ in the paper \cite{FPEF}.)

Then the construction gives a class of fields $F$ with a predicate $\Gamma(F)$ defining the graph of a surjective homomorphism from $\Ga(F)$ to $\Gm(F)$, with kernel $\tau \Z$, which we denote by $\exp$. The predimension inequality is precisely Schanuel's conjecture, and the strong existential closedness axiom is known as strong exponential-algebraic closedness. Thus we obtain a proof of Theorem~\ref{main exp theorem}, which we restate in explicit form.
\begin{theorem}
Up to isomorphism, there is exactly one model $\langle F;+,\cdot,\exp \rangle$ of each uncountable cardinality of the following list $\ECFskccp$ of axioms.
\begin{description}
 \item[1. ELA-field] $F$ is an algebraically closed field of characteristic zero, and $\exp$ is a surjective homomorphism from $\ga(F)$ to $\gm(F)$.

 \item[2. Standard kernel] the kernel of $\exp$ is an infinite cyclic group generated by a transcendental element $\tau$.

 \item[3. Schanuel Property] The \emph{predimension function} 
\[\delta(\xbar) \leteq \td(\xbar, \exp(\xbar))- \ldim_\Q(\xbar)\]
satisfies $\delta(\xbar) \ge 0$ for all tuples $\xbar$ from $F$. 

\item[4. Strong exponential-algebraic closedness] If $V$ is a rotund, additively and multiplicatively free subvariety of $\ga^n\cross \gm^n$ defined over $F$ and of dimension $n$, and $\abar$ is a finite tuple from $F$, then there is $\xbar$ in $F$ such that $(\xbar,e^\xbar) \in V$ and $\xbar$ is $\Q$-linearly independent over $\abar$ (that is, no non-zero $\Q$-linear combination of the $x_i$ lies in the $\Q$-linear span of the $a_i$).

\item[5. Countable Closure Property] For each finite subset $X$ of $F$, the exponential algebraic closure $\ecl^F(X)$ of $X$ in $F$ is countable. 
\end{description}
\end{theorem}
\begin{proof}
We apply Theorem~\ref{main Gamma theorem}, but note that axioms 2 and 3 are slightly different from the axioms given in the statement of Theorem~\ref{main theorem}. The Schanuel property holds on our choice of $\Fbase$ because $\tau$ is transcendental, and it follows from the addition property for $\delta$ that the two versions of axiom 3 are equivalent in this case. Since $\tau$ is transcendental and the kernel is standard,  it follows that $\Fbase$ embeds strongly in $F$, so the two versions of axiom 2 are also equivalent.
\end{proof}
We denote the canonical model of cardinality continuum by $\B$.

\subsection{Incorporating a counterexample to Schanuel's conjecture}

We proceed as in the previous example, except now we choose an irreducible polynomial $P(x,y) \in \Z[x,y]$ and take $(\epsilon, \tau)$ to be a generic zero of the polynomial $P(x,y)$. (We assume that $P$ is such that neither $\epsilon$ nor $\tau$ is zero.) Choose a division sequence $(\epsilon_m)$ for $\epsilon$, that is, numbers such that $\epsilon_1 = \epsilon$ and $(\epsilon_{mn})^n = \epsilon_m$ for all $m,n \in\N^+$. Now take $K$ to be the field $\Qab(\tau,(\epsilon_m)_{m \in \N^+})$, and define $\Gamma(K)$ to be the graph of a homomorphism from the $\Q$-linear span of $\tau$ and 1, with $\tau/m \mapsto \omega_m$ as above and $1/m \mapsto \epsilon_m$.

Now the construction gives us a canonical model $\B_P$, the unique model of cardinality continuum of almost the same list of axioms as those for $\B$, except that Schanuel's conjecture has this exception with the formal analogues $\epsilon$ and $\tau$ of $e$ and $2\pi i$ being algebraically dependent via the polynomial $P$. More precisely, the predimension axiom is replaced by an axiom scheme stating that $\exp(1)$ and $\tau$ are transcendental, that $P(\exp(1),\tau) = 0$, and the condition that for all tuples $\abar$, $\td(\abar,\exp(\abar)/\tau,\exp(1)) - \ldim_\Q(\abar/\tau,1) \ge 0$.

More generally, we can take any finitely generated partial exponential field with standard kernel (that is, a finitely generated $\Gamma$-field for the appropriate groups and kernels) as $\Fbase$ and do the same construction to build a quasiminimal exponential field $\MM(\Fbase)$ of size continuum with counterexamples to the Schanuel property within a finite-dimensional $\Q$-vector space, but with the Schanuel property holding over that vector space. Each $\MM(\Fbase)$ is unique up to isomorphism as a model of appropriate axioms, just as $\B$ is. One could conjecture that $\Cexp$ is isomorphic to one of these. Several people have asked us if it might be possible to prove Schanuel's conjecture easily by some method showing that $\Cexp$ must be isomorphic to \B, just because \B\ is categorical. Examples such as these show that soft methods which ignore transcendental number theory and analytic considerations cannot hope to work.




\subsection{Pseudo-Weierstrass $\wp$-functions}\label{pseudo-p section}


Let $E$ be an elliptic curve over a number field $K_0$. Choose a Weierstrass equation for $E$
\[Y^2 Z = 4X^3 - g_2 X Z^2 - g_3 Z^3\]
with $g_2,g_3 \in K_0$, which fixes an embedding of $E$ into projective space $\mathbb P^2$, with homogeneous coordinates $[X:Y:Z]$. Apart from the point $O = [0:1:0]$ at infinity, we can identify $E$ with its affine part, given by solutions of the equation
\[y^2 - 4x^3-g_2 x - g_3 = 0\]
in $\mathbb A^2$.

For our construction we take $G_1 = \Ga$ and $G_2 = E$, and we take $\OO = \End(E)$, so $\OO = \Z$ if $E$ does not have complex multiplication (CM) and $\OO = \Z[\tau]$ if $E$ has CM by the imaginary quadratic $\tau$. In the CM case, we assume that $\tau \in K_0$ (and adjoin it if not). Take $\omega_1$ transcendental over $K_0$ and $\omega_2$ transcendental over $K_0(\omega_1)$ if $E$ does not have CM, or $\omega_2 = \tau \omega_1$ if $E$ has CM by $\tau$.

As a field, we define $\Fbase = K_0(\Tor(E),\omega_1,\omega_2)$, where $\Tor(E)$ means the full torsion group of $E$, which is contained in $E(\Qalg)$. We define $\Gamma(\Fbase)$ to be the graph of a surjective $\OO$-module homomorphism from $\Q \omega_1 + \Q \omega_2$ to $\Tor(E)$, with kernel $\Lambda = \Z\omega_1 + \Z\omega_2$. While this may not specify $\Fbase$ up to isomorphism, we will see that Serre's open image theorem allows us to specify $\Fbase$ with only a finite amount of extra information.

In a model $M$, $\Gamma(M)$ will be the graph of a surjective homomorphism $\exp_{E,M} :\Ga(M) \to E(M)$ with kernel $\Lambda$. Using our chosen embedding of $E$ into $\mathbb P^2$, we can identify the components of the function $\exp_{E,M}$ with functions $\wp,\wp' : M \to M \cup \{\infty\}$, where $\exp_{E,M}(a) = [\wp(a) : \wp'(a) : 1]$. We call the function $\wp$ a ``pseudo-Weierstrass $\wp$-function''. 

Note that in our model $M$, $\Lambda$ is definable by the formula $\exp_{E,M}(x) = O$. In the non-CM case, $\Z$ is definable by the formula $\forall y[y \in \Lambda \to xy \in \Lambda]$, so $\Q$ is also definable as the field of fractions. In the CM case, these formulas define the rings $\Z[\tau]$ and $\Q[\tau]$.

Following \cite{Gav08}, we will apply the following version of Serre's open image theorem to show that only a finite amount of extra information is required to specify $\Fbase$ as a $\Gamma$-field.

\begin{fact}\label{fact:serre} Let $E$ be an elliptic curve defined over a number field $K_0$. Then there exists an $m \in \N$ such every $\End(E)$-module automorphism of the torsion $\Tor(E)$ which fixes the $m$-torsion $E[m]$ pointwise is induced by a field automorphism over $K_0$, that is, the natural homomorphism
    \[ \Gal(K_0^\alg/K_0(E[m])) \to  \Aut_{\End(E)}(\Tor(E)/E[m]) \]
    is surjective.
\end{fact}
\begin{proof}
    When $\End(E) \isom  \Z$, this is Serre's open image theorem
    \cite[Introduction (3)]{Serre72}. When $E$ has complex multiplication, 
    it is the analogous classical open image theorem \cite[Section 4.5,
    Corollaire]{Serre72}.
\end{proof}
Unfortunately the proof does not give an effective bound for $m$, so given an explicit $K_0$ and $E$ we do not know how to compute it.

We can now prove the first half of Theorem~\ref{main p theorem}, which we restate precisely.
\begin{theorem}\label{technical p theorem}
Let $E$ be an elliptic curve over a number field $K_0 \subs \C$. Up to isomorphism, there is exactly one model of each uncountable cardinality of the following list $\PCFskccp(E)$ of axioms, and these models are all quasiminimal.
\begin{description}
 \item[1. Full $\wp$-field] $M$ is an algebraically closed field of characteristic zero, and $\Gamma$ is the graph of a surjective homomorphism from $\ga(M)$ to $E(M)$, which we denote by $\exp_{E,M}$. We add parameters for the number field $K_0$.

 \item[2. Kernel and base (non-CM case)] There exist $\omega_1,\omega_2 \in \ga(M)$, $\Q$-linearly independent, such that the kernel of $\exp_{E,M}$ is of the form $\Z\omega_1 + \Z \omega_2$ and, for the number $m$ specified by Fact~\ref{fact:serre}, the algebraic type of the pair $(\exp_{E,M}(\omega_1/m), \exp_{E,M}(\omega_2/m)) \in E[m]^2$ over the parameters $K_0$ is specified.
 
 \item[2. Kernel and base (CM case)] There exists a non-zero $\omega_1 \in \ga(M)$ such that the kernel of $\exp_{E,M}$ is of the form $\Z[\tau]\omega_1$, and for the number $m$ specified by Fact~\ref{fact:serre}, the algebraic type of $\exp_{E,M}(\omega_1/m) \in E[m]$ over the parameters $K_0$ is specified.

\item[3. Predimension inequality] The \emph{predimension function} 
\[\delta(\xbar) \leteq \td(\xbar, \exp_{E,M}(\xbar))- \ldim_\kO(\xbar)\]
satisfies $\delta(\xbar) \ge 0$ for all tuples $\xbar$ from $M$ where $\kO = \Q$ or $\Q(\tau)$ as appropriate.
 
\item[4. Strong $\wp$-algebraic closedness] The specific case of strong $\Gamma$-closedness for this choice of $G$, $\OO$ and $\Fbase$, from Theorem~\ref{main theorem}.
\item[5. Countable Closure Property] as in Theorem~\ref{main theorem}.
\end{description}
\end{theorem}
\begin{proof}
Again we must show that these axioms are equivalent to those given in Theorem~\ref{main theorem}. As in the exponential case, we have the absolute form of the predimension inequality here, which is equivalent to the relative statement over the base together with the assertion that $\omega_1$ is transcendental and, in the non-CM case, that $\omega_2$ is transcendental over $K_0(\omega_1)$. It remains to show that the axioms here specify the atomic diagram of $\Fbase$.

So suppose that $M$ and $M'$ are both models of the axioms, and their bases, that is the $\Gamma$-subfields generated by the kernels, are $\Fbase$ and $\Fbase'$. We have $K_0$ as a common subfield, and the axioms give us kernel generators $(\omega_1,\omega_2) \in M^2$ and $(\omega_1',\omega_2') \in M'^2$ such that $(\alpha_1,\alpha_2) \leteq (\exp_{E,M}(\omega_1/m), \exp_{E,M}(\omega_2/m))$ and $(\alpha_1',\alpha_2') \leteq (\exp_{E,M'}(\omega_1'/m),\exp_{E,M'}(\omega_2'/m))$ have the same algebraic type over $K_0$. (In the CM case we define $\omega_2 = \tau \omega_1$ and $\omega_2' = \tau \omega_1'$ to treat the two cases at the same time.) The points $\alpha_i \in E$ generate the $m$-torsion subgroup $E[m]$ of $E$. So we can define a field isomorphism $\sigma_1 : K_0(E[m](M)) \to K_0(E[m](M'))$ by $\alpha_i \mapsto \alpha_i'$ for $i=1,2$. Then we extend $\sigma_1$ arbitrarily to a field isomorphism $\sigma_2: K_0(\Tor(E)(M)) \to K_0(\Tor(E)(M'))$.

Now define an $\End(E)$-module automorphism of $\Tor(E)(M)$ by
\[ \exp_{E,M}\left(\frac{\omega_1}{n_1} + \frac{\omega_2}{n_2}\right) \mapsto \sigma_2^{-1} \left( \exp_{E,M'}\left(\frac{\omega'_1}{n_1} + \frac{\omega'_2}{n_2}\right) \right) \]
for all $n_1,n_2 \in \Z$.
By construction of $\sigma_2$ this automorphism fixes $E[m]$ pointwise, so by Fact~\ref{fact:serre} it extends to a field automorphism $\sigma_3$ of $K_0(\Tor(E)(M))$. So defining $\sigma_4 = \sigma_2\circ\sigma_3$ we get a field isomorphism $\sigma_4 : K_0(\Tor(E)(M)) \to K_0(\Tor(E)(M'))$ such that $\sigma_4\left( \exp_{E,M}\left(\frac{\omega_1}{n_1} + \frac{\omega_2}{n_2}\right) \right) = \exp_{E,M'}\left(\frac{\omega'_1}{n_1} + \frac{\omega'_2}{n_2}\right)$ for all $n_1,n_2 \in \Z$.

The predimension inequality implies that $(\omega_1, \omega_2)$ and $(\omega_1',\omega_2')$ have the same field-theoretic type over $\Qalg$, so we can extend $\sigma_4$ to a field isomorphism $\sigma_5$ by defining $\sigma_5(\omega_i) = \omega_i'$ for $i=1,2$, and this $\sigma_5$ is a $\Gamma$-field isomorphism $\Fbase \to \Fbase'$ as required.
\end{proof}

Later in Proposition~\ref{wp predim prop} we will show that the predimension inequality above is the appropriate form of Schanuel's conjecture for the $\wp$-functions, and we will thereby complete the proof of Theorem~\ref{main p theorem}.

\subsection{Variants on $\wp$-functions.}
As in the exponential case, we can do variant constructions by changing the base field $\Fbase$ to a different finitely generated $\Gamma$-field, to incorporate some counterexamples to the predimension inequality. We can also do constructions of ``pseudo-analytic'' homomorphisms $\ga(M) \to E(M)$ which have no complex-analytic analogue. For example, choose an elliptic curve $E$ without complex multiplication and take the kernel lattice $\Lambda = \Z\omega_1 + \Z\omega_2$ for $\omega_1 / \omega_2 \in \R$, totally real (that is algebraic and such that it and all its conjugates are real), for example real quadratic. The construction still works perfectly well to produce a unique quasiminimal model, but no embedding of $\Lambda$ into $\C$ can be the kernel of a meromorphic homomorphism because it is dense on the line $\R \omega_1$.

\subsection{Exponential maps of simple abelian varieties}

This is the algebraic setup corresponding to the complex example described in Definition~\ref{analytic EXP defn}. Take $G_1 = \Ga^d$ and $G_2$ a simple abelian variety of dimension $d$, defined over a number field $K_0$. Take $\OO = \End(G_2)$, and suppose these endomorphisms are also defined over $K_0$. Fix an embedding of $K_0$ into $\C$.

Let $\omega_1,\ldots, \omega_{2d} \in \C^d$ be generators of a lattice $\Lambda$ such that $\C^d/\Lambda$ is isomorphic to $G_2(\C)$ as a complex $\OO$-module manifold.

We take $\Fbase$ to be the field generated by the $\omega_i$ together with $\Tor(G_2)$, and $\Gamma(\Fbase)$ to be the graph of an $\OO$-module homomorphism from $\Q \Lambda$ onto $\Tor(G_2)$.

For abelian varieties of dimension greater than 1 there is no non-conjectural analogue of Serre's open image theorem so we cannot be more specific about an axiomatization of the atomic diagram of the base. So we have no improvement on the statement of Theorem~\ref{main theorem} in this case.

\subsection{Factorisations via $\Gm$ of elliptic exponential maps}\label{theta factorization section}

The examples so far have all been of the exponential type, case (EXP). Here, we give an example in case (COR). Let $G_1 = \Gm$ and $G_2 = E$, an elliptic curve without complex multiplication, defined over a number field $K_0$. Let $\OO = \Z$.

Let $\omega$ be transcendental. As a field $\Fbase = K_0(\omega,\Tor(\Gm),\Tor(E))$ and we define $\Gamma(\Fbase)$ to be the graph of a surjective homomorphism from $\Q\omega + \Tor(\Gm)$ onto $\Tor(E)$ with kernel $\Z \omega$. Then for $\MM = \MM(\Fbase)$, $\Gamma(\MM)$ is the graph of a surjective homomorphism $\theta_\MM: \Gm(\MM) \to E(\MM)$.

In the complex case, the exponential map of $E$ factors through the exponential map of $\gm$ as
\[\begin{diagram}
\ga(\C) & \rTo^{[\wp:\wp':1]} & E(\C) \\
\dTo>{\exp} & \ruTo_{\theta} \\
\gm(\C)
\end{diagram}\]
and this pseudo-analytic map $\theta_\MM$ is an analogue of the complex map $\theta$. Since $E\times\Gm$ is not simple, the methods of this paper do not suffice to build a field $F$ equipped with a map $\theta$ and pseudo-exponential maps of $\gm$ and $E$ together, in which the analogue of the above commutative diagram would hold together with a suitable predimension inequality and a categoricity theorem for a reasonable axiomatization. However, it seems likely that this is achievable by combining the methods of this paper with those of \cite{TEDESV}.

\begin{question}
The main obstacle to stability for the first-order theory of $\B$ is the kernel. In this case the kernel is just a cyclic subgroup of $\gm$, and it is known that $\gm$ equipped with such a group is superstable. So it is natural to ask whether the first-order theory of $\MM$ in this case is actually superstable. One could even ask if any construction of type (B), say with finite rank kernels, produces a structure with a superstable first-order theory.
\end{question}

\subsection{Differential equations}

We now give an example of type \DE. Let $K_0$ be a countable field of characteristic 0, let $G_2$ be any simple semiabelian variety of dimension $d$ defined over $K_0$ and let $G_1 = \ga^d$. Let $\OO = \End(G_2)$. Let $C$ be a countable algebraically closed field extending $K_0$, and define $\Gamma(C) = G(C)$.

Now consider the amalgamation construction using $C$ as the base but considering only purely $\Gamma$-transcendental extensions of $C$, that is using the category $\Ctrans(C)$ in place of $\Cc(C)$. Theorem~\ref{Fraisse limit is QPS} shows we have a quasiminimal pregeometry structure, and hence a canonical model in each uncountable cardinality. The models we obtain are quasiminimal and $\Gcl(\emptyset) = C$.

This construction is also considered in \cite{TEDESV} where it is shown that the first-order theory of these models is $\aleph_0$-stable. 
In that paper it is also shown that if $\langle F; + ,\cdot,D \rangle$ is a differentially closed field and we define $\Gamma(F)$ to be the solution set to the exponential differential equation for $G_2$ then the reduct $\langle F; + ,\cdot,\Gamma \rangle$ is a model of the same first-order theory, and $C$ is the field of constants for the differential field.
The paper \cite{TEDESV} also considers the situation with several different groups $\Gamma_S$ relating to the exponential differential equations of different semiabelian varieties $S$, which do not have to be simple, but they do have to be defined over the constant field $C$.


\section{Comparison with the analytic models}\label{analytic comparison}

Zilber conjectured that $\Cexp \iso \B$ and it seems reasonable to extend the conjecture to the exponential map of any simple complex abelian variety, and indeed to other analytic functions such as the function $\theta$ from section~\ref{theta factorization section}. In each example, axioms 1 and 2 are set up to describe properties we know about these analytic functions, so verifying the conjecture amounts to verifying the other three axioms. We consider the progress towards each of the axioms in turn.

\subsection{The predimension inequalities}

For the usual exponential function, the predimension inequality states that for all tuples $a$ from $\C$, $\td(a,e^a) \ge \ldim_\Q(a)$. This is precisely Schanuel's conjecture.

In the case of an elliptic curve $E$ defined over a number field, the predimension inequality states for all tuples $a$ from $\C$,
$\delta_E(a) = \td(a,\exp_{E,\C}(a)) - \ldim_{\kO}(a) \ge 0$. 

\begin{prop}\label{wp predim prop}
The predimension inequality above for the exponential map of an elliptic curve follows from the Andr\'e-Grothendieck conjecture on the periods of 1-motives.
\end{prop}
The following proof was explained to us by Juan Diego Caycedo, and follows the proof of a related statement in section~3 of \cite{CZ14}.
\begin{proof}
By Th\'eor\`eme~1.2 of \cite{Bertolin02}, with $s=0$ and $n=1$, a special case of Andr\'e's conjecture (building on Grothendieck's earlier conjecture) states that if $j(E)$ is the $j$-invariant of $E$, $\omega_1$ and $\omega_2$ are the periods of $E$, $\eta_1$ and $\eta_2$ are the quasiperiods of $E$, $P_1,\ldots,P_n \in E(\C)$, $a_i$ is the integral of the first kind associated with $P_i$, and $d_i$ is the integral of the second kind associated with $P_i$ then:
\begin{equation}
\td(2\pi i, j(E),\omega_1,\omega_2,\eta_1,\eta_2,\bar{P}, a,d) \ge 2\ldim_\kO(a/\omega_1,\omega_2) + 4[\kO:\Q]^{-1} .
\end{equation}
In this case, we have that $P_i = [\wp(a_i) : \wp'(a_i) : 1] = \exp_{E,\C}(a_i)$. Since our $E$ is defined over a number field, $j(E)$ is algebraic. The Legendre relation states $\omega_1\eta_2 - \omega_2\eta_1 = 2 \pi i$, so $j(E)$ and $2 \pi i$ do not contribute to the above inequality.

If we assume that $a_1,\ldots,a_n \in \C$ are $\kO$-linearly independent over $\omega_1,\omega_2$ we can discard the integrals of the second kind to get the bound
\begin{equation}\label{Bert general}
\td(\omega_1,\omega_2,\eta_1,\eta_2, a,\exp_{E,\C}(a)) \ge \ldim_\kO(a/\omega_1,\omega_2) + 4[\kO:\Q]^{-1}.
\end{equation}

Consider the case where there is no CM, so $\kO = \Q$. Throwing away $\eta_1$ and $\eta_2$ we get
\begin{equation}\label{Bert noCM}
\td(\omega_1,\omega_2,a,\exp_{E,\C}(a)) \ge \ldim_\kO(a/\omega_1,\omega_2) + 2.
\end{equation}
From the case $n=0$ we see that $\td(\omega_1,\omega_2)=2$ and since $\omega_1$ and $\omega_2$ are $\Q$-linearly independent we have $\delta_E(\omega_1,\omega_2) = 0$. Then (\ref{Bert noCM}) implies for any $a$ we have $\delta_E(a/\omega_1,\omega_2) \ge 0$, and putting these two statements together we deduce that $\delta_E(a) \ge 0$.

Where $E$ does have CM (and is defined over a number field) Chudnovsky's theorem \cite[Theorem~1 and Corollary~2]{Chudnovsky80} gives us $\td(\omega_1,\omega_2,\eta_1,\eta_2) = \td(\omega_1,\pi) = 2$, so in particular $\td(\omega_1) = 1$. We also have $\kO = \Q(\tau)$ with $[\kO:\Q] = 2$ and $\omega_2 = \omega_1\tau$, so we can discard $\omega_2, \eta_1$ and $\eta_2$ from (\ref{Bert general}) to obtain
\begin{equation}\label{Bert CM}
\td(\omega_1,a,\exp_{E,\C}(a)) \ge \ldim_\kO(a/\omega_1) + 1.
\end{equation}
The same argument now shows that $\delta_E(a) \ge 0$ for any tuple $a$.
\end{proof}

\begin{proof}[Proof of Theorem~\ref{main p theorem}]
Theorem~\ref{technical p theorem} shows that the axioms $\PCFskccp(E)$ are uncountably categorical and that every model is quasiminimal. The analytic structure $\C_\wp$ is a model of the first two axioms by construction. Proposition~\ref{wp predim prop} shows that the predimension inequality given in axiom 3 is the appropriate analogue of Schanuel's conjecture for $\wp$-functions. Axiom 5, the countable closure property was proved in this case in \cite{JKS16}.
\end{proof}
We will give another proof of the countable closure property in Theorem~\ref{CCP for analytic Gamma-fields} in this paper.

Our understanding of the periods conjecture uses Bertolin's translation to remove the motives, which she did only in the cases of elliptic curves and $\Gm$. For abelian varieties of dimension greater than 1 we suspect that the predimension inequality axiom again follows from the Andr\'e-Grothendieck periods conjecture, but there are more complications because the Mumford-Tate group plays a role and so we have not been able to verify it.

\subsection{Strong $\Gamma$-closedness}

In the case of the usual exponentiation for $\gm$, Mantova \cite{Mantova14} currently has the best result towards proving the Strong $\Gamma$-closedness in the complex case. He only considers the case of a variety $V \subs G^n$ where $n=1$. A free and rotund $V \subs G^1$ is just the solution set of an irreducible polynomial $p(x,y) \in \C[x,y]$ which depends on both $x$ and $y$, that is, the partial derivatives $\frac{\partial p}{\partial x}$ and $\frac{\partial p}{\partial y}$ are both non-zero.
\begin{fact}
Suppose $p(x,y)\in \C[x,y]$ depends on both $x$ and $y$. Then there are infinitely many points $x \in \C$ such that $p(x,e^x) = 0$. Furthermore suppose Schanuel's conjecture is true and let $a$ be a finite tuple from $\C$. Then there is $x \in \C$ such that $(x,e^x)$ is a generic zero of $p$ over $a$.
\end{fact}
The observation that there are infinitely many solutions and the whole statement in the case that $p$ is defined over $\Q^\alg$ is due to Marker \cite{Marker06}. The general case stated above is due to Mantova {\cite[Theorem~1.2]{Mantova14}}.


\subsection{$\Gamma$-closedness}
In section~\ref{GGC section} we will see that for some purposes strong $\Gamma$-closedness can be weakened to $\Gamma$-closedness.
\begin{defn}\label{G-closed defn}
A full $\Gamma$-field $F$ is \emph{$\Gamma$-closed} if for every irreducible subvariety $V$ of $G^n$ defined over $F$ and of dimension $dn$, which is free and rotund, $V(F) \cap \Gamma(F)^n$ is Zariski-dense in $V$.
\end{defn}
Using the classical Rabinovich trick, one can easily show this axiom scheme is equivalent to the existence of a single point $\beta \in V(F) \cap \Gamma(F)^n$, for every such $V$. For the usual exponentiation, $\Gamma$-closedness is known as exponential-algebraic closedness. In this direction, Brownawell and Masser \cite[Proposition~2]{BM16} have the following.
\begin{fact}
If $V \subs (\ga \cross \gm)^n(\C)$ is an algebraic subvariety of dimension $n$ which projects dominantly to $\ga^n$ then there is $a \in \C^n$ such that $(a,e^a) \in V$.
\end{fact}
In this case $V$ can be taken free without loss of generality, and the condition of projecting dominantly to $\ga^n$ implies rotundity. However it is much stronger than rotundity. Another exposition of this theorem is given in \cite{DAFT16}.

\subsection{The pregeometry and the countable closure property}

To compare the pregeometry of our constructions such as $\B$ with the complex analytic models such as $\Cexp$ we have to define the appropriate pregeometry on the complex field. Given a $\Gamma$-field $F$, we defined a $\Gamma$-subfield $A$ of $F$ to be $\Gamma$-closed in $F$ if whenever $A \subs B \subs F$ with $\delta(B/A) \le 0$ then $B \subs A$. One can construct $\Gamma$-fields with no proper $\Gamma$-closed subfields. Fortunately we are able to show unconditionally that there is a countable $\Gamma$-subfield of $\C$ which is $\Gamma$-closed in $\C$.

For $\Cexp$, this was done in \cite{EAEF} by adapting the proof of Ax's differential forms version of Schanuel's conjecture. A similar proof was given in \cite{JKS16} for elliptic curves. The same method ought to work for any semiabelian varieties, but here we give a different approach, applying the main result of \cite{Ax72a} directly to generalise a theorem of Zilber in the exponential case \cite[Theorem~5.12]{Zilber05peACF0}.

Let $\C_\Gamma$ be an analytic $\Gamma$-field, which recall means a $\Gamma$-field from Definition~\ref{analytic EXP defn} or~\ref{analytic COR defn}. Then $\Gamma$ is a complex Lie subgroup of $G(\C)$, and $L\Gamma \leq LG(\C)$ is the graph of a $\C$-linear isomorphism between $LG_1(\C)$ and $LG_2(\C)$.
Thus $\Gamma^n$ is a complex Lie subgroup of $G^n(\C)$, so has a complex topology. It might not be a closed subgroup, so the topology might not be the subspace topology. 

\begin{defn}
Given an algebraic subvariety $V\subs G^n$, write $V^\mathrm{isol}$ for the set of all isolated points of $V(\C) \cap \Gamma^n$ with respect to the complex topology on $\Gamma^n$.
For any subset $A \subs \C$ we define $\Gcl'(A)$ to be
the subfield of $\C$ generated by the union of all $V^\mathrm{isol}$ where $V$ ranges over the algebraic subvarieties $V\subs G^n$ which are defined over $\bk(A)$.
We consider $\Gcl'(A)$ as a $\Gamma$-subfield of $\C$ by defining
$\Gamma(\Gcl'(A)) := \Gamma \cap G(\Gcl'(A))$.
\end{defn}
\begin{lemma}\label{CCP for Gcl'}
$\Gcl'$ is a closure operator on $\C$ and for any $A \subs \C$,
$\Gcl'(A)$ is a full $\Gamma$-subfield of $\C$ of cardinality $|A| + \aleph_0$.
\end{lemma}
\begin{proof}
  For transitivity, suppose $x$ is a finite tuple from $\Gcl'(A)$ and $y\in\Gcl'(Ax)$. We may reduce to the case that
  $\alpha\in W(\C)\cap\Gamma^n$ is isolated and the tuple $x$ lists the co-ordinates of
  $\alpha$, and $\beta\in V(\C)\cap\Gamma^m$ is isolated and $y$ is a
  co-ordinate of $\beta$, with $W$ defined over $K_0(A)$ and $V$ over
  $K_0(Ax)$. Then $V$ can be written as a fibre $V'(\alpha)$ of a subvariety
  $V'\subs G^{m+n}$ over $K_0(A)$ projecting to $W \subs G^n$. Then $\beta\alpha
  \in V'\cap\Gamma^{m+n}$ is an isolated point, so $\beta\in\Gcl'(A)$.
  
  Now let $A\subs\C$, and set $A' := \acl^\C(K_0(A))$.
  Let $\alpha_0 \in G_1(A')$, and let $V_0 \subs G_1$ be the set of its
  conjugates, i.e.\ the 0-dimensional locus of $\alpha_0$ over $K_0(A)$, and let
  $V := V_0 \times G_2 \subs G$. Then $V\cap\Gamma$ is non-empty, since $\pi_1 :
  \Gamma \onto G_1(\C)$ is surjective, and it consists of isolated points since
  $\ker(\pi_1)$ does. This shows that $A' \subs \Gcl'(A)$, and hence in
  particular that $\Gcl'$ is a closure operator, and furthermore this argument 
  shows that $\Gcl'(A)$ is full.

  Finally, for the cardinality calculation, note there are only $(|A| + \aleph_0)$-many algebraic varieties $V$ defined over $A$ and for each there can be only countably many isolated points in $V(\C) \cap \Gamma^n$.
\end{proof}

\begin{prop}\label{Gcl=Gcl'}
For an analytic $\Gamma$-field $\C_\Gamma$, the closure operators $\Gcl$ and $\Gcl'$ are the same. In particular, $\Gcl'$ is a pregeometry on $\C$.
\end{prop}

To prove this we will use a lemma and Ax's theorem on the transversality of intersections between analytic subgroups and algebraic varieties.
\begin{lemma}\label{algSubgroupsAndGamma}
  If $H \leq  G^n$ is a connected algebraic subgroup which is free then the analytic subgroup
  $H(\C) + \Gamma^n$ is equal to $G^n(\C)$.
\end{lemma}
\begin{proof}
Since $G_2$ is simple and not isogenous to $G_1$, every algebraic subgroup of $G^n$ is of the form $H_1 \cross H_2$ with $H_i$ a subgroup of $G_i^n$, and since $H$ is free it is $G_2$-free. Since $\OO = \End(G_2)$ it follows that $H$ is of the form $H_1 \cross G_2^n$. Now since $\pi_1(\Gamma^n) = G_1^n(\C)$ we see $H(\C) + \Gamma^n = G^n(\C)$. 
\end{proof}

\begin{fact}[{\cite[Corollary~1]{Ax72a}}]\label{Ax72 fact}
Suppose that $\mathcal G$ is a complex algebraic group, $A$ is a connected analytic subgroup of $\mathcal G$, $U$ is open in $\mathcal G$ and $X$ is an irreducible analytic subvariety of $U$ such that $X \subs A$, $X^{\mathrm{Zar}}$ is the Zariski closure of $X$ and $H$ is the smallest algebraic subgroup of $\mathcal {G}$ containing $X$. Then
\[ \dim(H+A) \le \dim X^{\mathrm{Zar}} + \dim A - \dim X .\]
\end{fact}

\begin{proof}[Proof of Proposition~\ref{Gcl=Gcl'}]
Suppose that $A \subs \C$ is $\Gcl$-closed. Let $\alpha \in V^\mathrm{isol}$ for some $V$ defined over $\bk(A)$. Replacing $\alpha$ with a subtuple if necessary, we may assume $\alpha$ is $\kO$-linearly independent over $\Gamma(A)$. Then $\alpha$ is a smooth point of $V$ and of $\Gamma^n$, so by considering tangent spaces we see that $\td(\alpha/A) \le \dim V \le nd$, and hence $\delta(\alpha/A) \le 0$. So $\alpha \in \Gamma^n(A)$ and hence $A$ is $\Gcl'$-closed.
 
Now suppose $A$ is $\Gcl'$-closed, and that $A \subs B \subs \C$ is a proper finitely generated $\Gamma$-field extension in $\C_\Gamma$. Let $b \in \Gamma^n(B)$ be a basis for the extension and let $V = \loc(b/A)$. We will show that $\delta(b/A) > 0$.

Let $X$ be an irreducible analytic component containing $b$ of the analytic subset $V(\C) \cap \Gamma^n$ of the complex Lie group $\Gamma^n$. Since $A$ is $\Gcl'$-closed and $b \notin A$, $\dim X \ge 1$.
 
We claim that $X$ has some point in $A$. To see this, let $e$ be a smooth point of $X$, and take regular local co-ordinates $\eta_i$ at $e$ in $G^n$ such that $X$ is locally the graph of a function from the first $\dim X$ co-ordinates to the rest. $A$ is algebraically closed as a field, so is topologically dense in $\C$. So there is a point $a \in X$ close to $e$ such that the first $\dim X$ co-ordinates are in $A$. Let $W$ be the intersection of $V$ with $\eta_i=a_i$ for $i = 1,...,\dim X$. Then $W$ is defined over $A$ and $a$ is an isolated point of $W(\C) \cap \Gamma^n$, hence $a$ is in $G^n(A)$ as required.

Suppose $X^{\mathrm{Zar}}$ is not $G_1$-free, so say $(x,y)\in X^{\mathrm{Zar}}$ implies a non-trivial
  $\er$-linear equation $\sum_{j=1}^n r_jx_j = c$. Then this equation holds for $\pi_1(a)$, so $c \in G_1(A)$. But then since $b\in X$, already $(x,y)\in V$ implies this equation, so $V$ is not $G_1$-free, a contradiction since $V$ is the locus of a basis over $A$. The same proof shows that $X^{\mathrm{Zar}}$ is $G_2$-free, so it is free.
  
Let $H$ be the algebraic subgroup of $G^n$ generated by $X^{\mathrm{Zar}}(\C)-b$. Then $X^{\mathrm{Zar}}(\C)-b$ is free, so $H$ is free. So by Lemma~\ref{algSubgroupsAndGamma}, the subgroup $H + \Gamma^n$ is equal to $G^n$.

Applying Fact~\ref{Ax72 fact} we get
 \[  \dim (H + \Gamma^n)  \le  \dim (X^{\mathrm{Zar}}-b) + \dim \Gamma^n - \dim (X-b) \]
which gives
\[  	2dn 			      \le  \dim  X^{\mathrm{Zar}} + dn - \dim X\]
but $\dim X >0 $ and $X^{\mathrm{Zar}} \subs V$ so we deduce that $\dim V > nd$. Thus $\delta(b/A) >0$. So $A$ is $\Gcl$-closed, as required.
\end{proof}

\begin{proof}[Proof of Theorem~\ref{CCP for analytic Gamma-fields}]
Proposition~\ref{Gcl=Gcl'} shows that $\Gcl = \Gcl'$, and so Lemma~\ref{CCP for Gcl'} shows that $\Gcl$ has the countable closure property.
\end{proof}

\begin{remark}
In \cite{EAEF}, an algebraic version of the isolated points closure $\Gcl'$ was given, using the fact that a solution to a system of $2n$ equations of analytic functions in $2n$ variables is isolated if and only if a certain Jacobian matrix does not vanish at the point. So this gives a definition of a closure operator $\ecl$ which makes sense on any exponential field, and it was shown in \cite{EAEF} that $\ecl$-closed sets are strong and agree with the $\Gcl$-closed sets as we have defined them here, and in particular that $\ecl$ is always a pregeometry. This algebraic definition of the closure operator was suggested by Macintyre \cite{Macintyre96} although it had previously been used in the real and complex cases by Khovanskii and by Wilkie.
\end{remark}


\section{Generically $\Gamma$-closed fields}\label{GGC section}

In this section we consider $\Gamma$-fields which may not be strongly $\Gamma$-closed but are generically strongly $\Gamma$-closed. 
Using the variant of the amalgamation construction from Section~\ref{purely G-trans section}, we show that such $\Gamma$-fields are also quasiminimal and that the \emph{strong} part of strong $\Gamma$-closedness becomes redundant in this generic case. 

Let $K$ be a full $\Gamma$-field. Recall from section~\ref{purely G-trans section} that an extension $K \subs A$ of $K$ is \emph{purely $\Gamma$-transcendental} if and only if $K \closed A$, if and only if for all tuples $b$ from $\Gamma(A)$, either $\delta(b/K) > 0$ or $b \subs \Gamma(K)$. Clearly an extension $A$ of $K$ is purely $\Gamma$-transcendental if and only if all of its finitely generated sub-extensions are.

\subsection{Generic $\Gamma$-closedness}

Recall from Theorem~\ref{main theorem} that a full $\Gamma$-field $F$ is \emph{strongly $\Gamma$-closed} if for every irreducible subvariety $V$ of $G^n$ defined over $F$ and of dimension $dn$, which is free and rotund for the $\OO$-module structure on $G$, and every finite tuple $a$ from $\Gamma(F)$, there is $b \in V(F) \cap \Gamma(F)^n$ such that $b$ is $\kO$-linearly independent over $\Gamma(\Fbase) \cup a$.

Assuming the Schanuel property, this is equivalent to requiring that $b$ is generic in $V$ over $\Fbase(a)$. Recall also the weaker form of $\Gamma$-closedness from Definition~\ref{G-closed defn}:

$F$ is \emph{$\Gamma$-closed} if for every irreducible subvariety $V$ of $G^n$ defined over $F$ and of dimension $dn$, which is free and rotund, $V(F) \cap \Gamma(F)^n$ is Zariski-dense in $V$.
\medskip

For the concept of generic $\Gamma$-closedness with respect to a subfield $K$, we need to consider extensions of the form $K \closed A \strong B$ where $A$ and $B$ are finitely generated as extensions of the full $\Gamma$-field $K$. Say $\alpha$ is a basis of $A$ over $K$ and $\beta$ is a basis of $B$ over $A$, and that $V \leteq \loc(\beta/A)$ and $W \leteq \loc(\alpha,\beta/K)$. We also assume that $\Afull \wedge B = A$. Then by Corollary~\ref{rotund cor} both $V$ and $W$ are free, $V$ is rotund, and $W$ is strongly rotund.

\begin{defn}[Generic $\Gamma$-closedness]\label{GSGC defn}
Let $F$ be a full $\Gamma$-field and $K \closed F$, $K \neq F$. Suppose $V \subs G^n$ is free and rotund, irreducible and of dimension $dn$. Suppose also that there is $\alpha \in \Gamma(F)^r$, $\kO$-linearly independent over $\Gamma(K)$ such that $V$ is defined over $K(\alpha)$ and for $\beta \in V$, generic over $K(\alpha)$, the variety $W \leteq \loc(\alpha,\beta/K)$ is strongly rotund. 

We say that $F$ is \emph{generically $\Gamma$-closed over $K$} (G$\Gamma$C over $K$) if, for all such $V$, we have that $V(F) \cap \Gamma^n(F)$ is Zariski-dense in $V$.

$F$ is \emph{generically strongly $\Gamma$-closed over $K$} (GS$\Gamma$C over $K$) if, whenever $V$ and $\alpha$ are as above, there is $\gamma \in V(F) \cap \Gamma^n(F)$, $\kO$-linearly independent over $\Gamma(K) \cup \alpha$.

We say $F$ is G$\Gamma$C or GS$\Gamma$C without reference to $K$ to mean G(S)$\Gamma$C over $\Gcl^F(\emptyset)$.
\end{defn}

\begin{prop}\label{GSGC equiv saturation}
Suppose $F$ is a full $\Gamma$-field and $K \closed F$, $K \neq F$. Then $F$ is GS$\varGamma$C over $K$ if and only if $F$ is $\aleph_0$-saturated for $\Gamma$-algebraic extensions which are purely $\Gamma$-transcendental over $K$.
\end{prop}
\begin{proof}
This is essentially the same as the proof of Lemma~\ref{SGClosed lemma}.
\end{proof}

It is immediate from the definitions that GS$\Gamma$C over $K$ implies G$\Gamma$C over $K$. In \cite{ECFCIT} it was shown that if  the Conjecture on Intersections with Tori (CIT, now also known as the multiplicative case of the Zilber-Pink conjecture) is true, then any exponential field satisfying the Schanuel property which is exponentially-algebraically closed is also strongly exponentially algebraically closed. We use similar ideas here to prove that G$\Gamma$C over $K$ implies GS$\Gamma$C over $K$. The Schanuel property is replaced by the assumption that $K$ is $\Gamma$-closed in $F$, and instead of the Zilber-Pink conjecture it is enough to use the weak version which is a theorem. 

Given any variety $S$ and subvarieties $W, V$ of $S$, the typical dimension of $W \cap V$ is $\dim W + \dim V - \dim S$. If $X$ is an irreducible component of $W \cap V$ it is said to have \emph{atypical dimension} (for the intersection) if 
\[\dim X > \dim W + \dim V - \dim S.\]
We also say that $X$ is an \emph{atypical component} of the intersection. 
For the multiplicative group, the \emph{weak Zilber-Pink} theorem is known as \emph{weak CIT} and is due to Zilber~\cite[Corollary~3]{Zilber02esesc}.
The semiabelian form of the statement is the following theorem \cite[Theorem~4.6]{TEDESV}.

\begin{fact}[``Semiabelian weak Zilber-Pink'', basic version]\label{wCIT}
Let $S$ be a semiabelian variety, defined over an algebraically closed field of characteristic 0. Let $(W_b)_{b \in B}$ be a constructible family of constructible subsets of $S$. That is, $B$ is a constructible set and $W$ is a constructible subset of $B \cross S$, with $W_b$ the obvious projection of a fibre. Then there is a finite set $\mathcal{H}_W$ of connected proper algebraic subgroups of $S$ such that for any $b \in B$ and any coset $c{+}J$ of any connected algebraic subgroup $J$ of $S$, if $X$ is an irreducible component of $W_b \cap  c{+}J$ of atypical dimension with $c \in X$ then there is $H \in \mathcal{H}_W$ such that $X \subs c{+}H$.
\end{fact}

We also need a version for subvarieties not of $S$ but of varieties of the form $U \cross S$, which is sometimes called a ``horizontal'' family of semiabelian varieties.
\begin{theorem}[``Horizontal semiabelian weak Zilber-Pink'']\label{hwCIT}
Let $S$ be a semiabelian variety and let $U$ be any variety. Let $( W_b)_{b \in B}$ be a constructible family of constructible subsets of $U \cross S$. Then there is a finite set $\mathcal{H}_W$ of connected proper algebraic subgroups of $S$ such that for any $b \in B$ and any coset $c{+}J$ of any connected algebraic subgroup $J$ of $S$, if $X$ is an irreducible component of $W_b \cap (U \cross c{+}J)$ of atypical dimension with $c \in X$ then there is $H \in \mathcal{H}_W$ such that $X \subs U \cross c{+}H$. Furthermore $H$ can be chosen such that we have
\[\quad (*) \hspace{3em} \dim X \le \dim \big(W_b \cap (U \cross c{+}H)\big) + \dim (H \cap J) - \dim H . \hspace{3em} {}\]
\end{theorem}
\begin{proof}
  First suppose $U$ is a point, so $U \cross S \iso S$. The main part of the statement is then Fact~\ref{wCIT}. For the ``furthermore'' part, suppose $(*)$ does not hold for the $H$ we chose from $\mathcal{H}_W$. Then rename $\mathcal{H}_W$ as $\mathcal{H}_W^1$. We give an inductive argument to find a new $\mathcal{H}_W$ which suffices. We have the irreducible $X$ as a component of the intersection $(W_b \cap c{+}H) \cap c{+}(H \cap J)$, and the failure of $(*)$ says that $X$ is atypical as a component of this intersection considered as an intersection of subvarieties of $c{+}H$. 
  Translating everything by $c$, we get that $X{-}c$ is an atypical component of the intersection $(W_b{-}c \cap H) \cap (H \cap J)$ inside $H$. Now apply Fact~\ref{wCIT} again with $H$ in place of $S$ to get a proper connected algebraic subgroup $H'$ of $H$ from the finite set $\mathcal{H}_W^2 \leteq \mathcal{H}_W^1 \cup \bigcup_{H \in \mathcal{H}_W^1} \mathcal{H}_{(W_b{-}c \cap H)_{b,c}}$ such that $X \subs c{+}H'$. If necessary we can iterate this construction and since $\dim H' < \dim H$ it stops after finitely many steps, and the $\mathcal{H}_W$ we eventually find is still finite.

  Now consider arbitrary $U$. Suppose first that all fibres of the co-ordinate
  projection $\pi: W_b \to S$ have the same dimension $k$, constant with
  respect to $b$.
  Take $\mathcal{H}_W$ to be the finite set $\mathcal{H}_{(\pi( W_b))_{b\in B}}$ given by this theorem with $U$ a point. We will see that
  this works as required.

  Indeed, let $c{+}J$ be a coset in $S$, let $X$ be an atypical component of $ W_b \cap (U\times c {+} J)$,
  let $Y$ be any irreducible component of $\pi( W_b) \cap c {+} J$ containing $\pi(X)$, and let $H \in \mathcal{H}_{(\pi( W_b))_{b\in B}}$ be as given by the theorem.

  Then by considering dimensions of fibres, we have
  \begin{eqnarray*}
   \dim X  & = & \dim(\pi(X)) + k \\
    &\le &\dim Y  + k\\
    &\le & \dim (\pi( W_b) \cap c {+} H)  + \dim (H \cap J) - \dim H  + k\\
     &=  &\dim \big( W_b \cap (U\times c {+} H) \big) + \dim (H \cap J) - \dim H 
  \end{eqnarray*}

  Now for a general family $ W \subseteq B \cross U\times S$, write $\pi : U \cross S \to S$ for the projection and define
\[W^k = \class{(b,u,s) \in W}{\dim (W_b \cap \pi^{-1}(s)) = k}\]
for each $k=0,\ldots,\dim W$. By the definability of dimension, these $W^k$ are all constructible subsets of $W$, partitioning it, and each $W^k$ satisfies the above constancy condition on fibres. For any $c {+} J$, any component $X$ of $ W_b \cap (U\times c {+} J)$ contains a dense constructible subset which lies in some piece $ W^k_b$. So we can take $\mathcal{H}_W$ to be $\bigcup_k \mathcal{H}_{ W^k}$.
\end{proof}

Now we can prove that GS$\Gamma$C and G$\Gamma$C are equivalent.
\begin{prop}\label{GSGC implies GGC}
Suppose $F$ is a full $\Gamma$-field and $K \closed F$, $K \neq F$. Then $F$ is GS$\varGamma$C over $K$ if and only if it is G$\varGamma$C over $K$.
\end{prop}
\newcommand{\Vadep}{\ensuremath{V_{\alpha,\mathrm{dep}}}}
\begin{proof}
As remarked earlier, it is immediate that GS$\Gamma$C over $K$ implies G$\Gamma$C over $K$. So assume $F$ is G$\Gamma$C over $K$. Let $V$, $\alpha$, $\beta$, and $W$ be as given in Definition~\ref{GSGC defn}. Let $V_{\alpha,\mathrm{dep}}$ be the set of points of $V(F)$ which are $\kO$-linearly dependent over $\Gamma(K) \cup \alpha$. We shall find a proper Zariski-closed subset of $V$ containing $\Vadep$.

We first work in case (EXP), so $G_2$ is a simple semiabelian variety of dimension $d$ and $G_1 = \ga^d$, which we identify with the Lie algebra $LG_2$ of $G_2$.

For a $d(r+n)$-square matrix $M$ and an $d(r+n)$-column vector $c$, let $\Lambda_{M,c} \subs \ga^{d(r+n)}$ be given by $x \in \Lambda_{M,c}$ if and only if $Mx = c$. So as $M$ and $c$ vary, we get the family of all possible affine linear subspaces. Let $U_{M,c} = W \cap  (\Lambda_{M,c} \cross G_2^{r+n})$.

Now suppose $\xi \in \Vadep \cap \Gamma(F)^n$. Let $\zeta =(\alpha,\xi) \in \Gamma(F)^{r+n}$. We write $\zeta$ also as $\zeta = (\zeta_1,\zeta_2) \in G_1^{r+n} \cross G_2^{r+n}$. Let $J$ be the smallest algebraic subgroup of  $G_2^{r+n}$ such that $\zeta_2$ lies in a $K$-coset of $J$,
say $\zeta_2 \in c_2'+J$ with $c_2' \in G_2(K)^{r+n}$.
Since $\xi \in \Vadep$, we see that $J$ is a proper algebraic subgroup of $G_2^{r+n}$.

By Lemma~\ref{subgroupsEndomorphisms},
there is $M \in \Mat_{r+n}(\OO)$ such that $J = (\ker(M))^o$ and $LJ = \ker(M)$.
Since $K$ is a full $\Gamma$-field,
there is $c_1' \in G_1^{r+n}(K)$ such that $c'\leteq (c_1',c_2') \in \Gamma(K)^{r+n}$.
Then $M(\zeta-c') \in \Gamma(K)^{r+n}$ since $K \leq F$ preserves the kernels.
Now $\Gamma(K)$ is a $k_\OO$-subspace of $\Gamma(F)$,
so in particular is existentially closed as an $\OO$-submodule,
so there exists $c'' \in \Gamma(K)^{r+n}$ such that $Mc'' = M(\zeta-c')$.
Set $c = (c_1,c_2) := c'+c'' \in \Gamma(K)^{r+n}$,
so $M(\zeta-c) = 0$.
Then $\zeta-c$ is divisible in $\ker(M)$,
since $\Gamma(F)$ is divisible and torsion-free,
so $\zeta-c \in LJ \times J$.

Now we have $\zeta \in U_{M,c_1} \cap (G_1^{r+n} \cross c_2 {+} J)$. Let $X$ be the irreducible component of this intersection containing $\zeta$.

We next show that $X$ has atypical dimension for the intersection.
From the definition of the predimension $\delta$ and its relationship with $\Gdim$ we have 
\begin{eqnarray}\label{ggc3}
\dim X \ge  \td(\zeta/K) & = & \delta(\zeta/K) + d \ldim_\OO(\zeta/\Gamma(K)) \nonumber\\
 & = & \delta(\zeta/K) + \dim J \nonumber\\
 & \ge  & \Gdim^F(\zeta/K) + \dim J \nonumber\\
\dim X & \ge  & \Gdim^F(\alpha/K) + \dim J.
\end{eqnarray}
Since $\alpha \in \Gamma(F)^r$ was chosen $\kO$-linearly independent over $\Gamma(K)$ and such that $\gen{K,\alpha} \strong F$, we have 
\begin{equation}\label{ggc5}
\Gdim^F(\alpha/K) = \delta(\alpha/K) = \td(\alpha/K) - d\ldim_\OO(\alpha/\Gamma(K)) = \td(\alpha/K) - dr. 
\end{equation}
Since $W= \loc(\alpha,\beta/K)$, and $V = \loc(\beta/K(\alpha))$ has dimension $dn$, using (\ref{ggc5}) we have
\begin{eqnarray}\label{ggc7}
\dim W &=& \dim V + \td(\alpha/K) \nonumber\\
& =  & d(r+n) + \Gdim^F(\alpha/K) 
\end{eqnarray}
From (\ref{ggc3}) and (\ref{ggc7}), 
\begin{eqnarray}
\dim X &\ge & \dim W + \dim J - d(r+n) \nonumber\\
 &= & \dim W + (\dim J + d(r+n)) - 2d(r+n) \nonumber\\
 & = & \dim W + \dim(G_1^{r+n} \cross c_2 {+} J) - \dim (G^{r+n}) \nonumber
\end{eqnarray}
but $W$ is free, so $\dim U_{M,c_1} < \dim W$ and so
\begin{equation}
\dim X  > \dim U_{M,c_1} + \dim(G_1^{r+n} \cross c_2 {+} J) - \dim (G^{r+n}).
\end{equation}
So $X$ has atypical dimension.

Applying Theorem~\ref{hwCIT} there is a proper algebraic subgroup $H$ of $G_2^{r+n}$ from the finite set $\mathcal{H}_U$ such that $X \subs G_1^{r+n} \cross c{+}H$. We have $\zeta \in X$, so $\zeta_2 \in c{+}H$. $J$ was chosen as the smallest algebraic subgroup of $S$ such that $\zeta_2$ lies in a $K$-coset of $J$, so $J \subs H$ and hence $H \cap J = J$. So, from the ``furthermore'' clause of Theorem~\ref{hwCIT} we have
\begin{equation}\label{ggc1}
\dim X \le \dim\bigg( U_{M,c_1} \cap (G_1^{r+n} \cross c_2 {+} J)\bigg) + \dim J - \dim H.
\end{equation}

We write $TJ = LJ \cross J$ and $TH = LH \cross H$, thinking of them as the tangent bundles. Then we have 
\begin{eqnarray*}
 U_{M,c_1} \cap (G_1^{r+n} \cross c_2 {+} J) & = & W \cap (c_1 {+} LJ \cross G_2^{r+n}) \cap (G_1^{r+n} \cross c_2 {+} J) \\
 & = & W \cap c {+} TJ\\
 & = & W \cap \zeta {+} TJ\\
 & \subs & W \cap \zeta {+} TH
\end{eqnarray*}
so
\begin{equation}\label{ggc2}
\dim \bigg( U_{M,c_1} \cap (G_1^{r+n} \cross c_2 {+} J)\bigg) \le \dim \big(W \cap \zeta {+} TH\big).
\end{equation}

Combining (\ref{ggc1}), (\ref{ggc2}), and (\ref{ggc3}) we get
\begin{eqnarray}\label{ggc4}
 \Gdim^F(\alpha/K) + \dim J & \le &  \dim \big(W \cap \zeta {+} TH\big) + \dim J - \dim H \nonumber\\
 \Gdim^F(\alpha/K) + \dim H & \le &  \dim \big(W \cap \zeta {+} TH\big) .
\end{eqnarray}

By Lemma~\ref{subgroupsEndomorphisms},
there is $M \in \Mat_{r+n}(\OO)$ such that $H = \ker(M)^o$ and $LH = \ker(M)$,
so $TH = \ker(M)^o$ for the action of $\OO$ on $G$.
So $M : G^{r+n} \to G^{r+n}$ factors as $G^{r+n} \onto G^{r+n} / TH \to G^{r+n}$,
where the first homomorphism is the quotient map,
and the kernel of the second homomorphism is the finite group $\ker(M) / \ker(M)^o$.
Now let $\theta_H : W \onto W / TH$ be the restriction of the quotient map $G^{r+n} \onto G^{r+n} / TH$.
Then $\dim(W/TH) = \dim(M\cdot W)$.
Now $H$ is a proper subgroup of $G_2^{r+n}$,
and $\dim(H) = \dim(\ker(M)) = d(r+n - \rk M)$,
and so $M$ is non-zero.
Then since $W$ is strongly rotund,
$\dim(W/TH) = \dim(M\cdot W) > d \rk M = d(r+n) - \dim H$.

So using the fibre dimension theorem,
the dimension of a typical fibre of $\theta_H$ is
\begin{eqnarray}
\dim (\text{typical fibre}) & = &\dim W - \dim(W/TH) \nonumber \\
& < & \big( d(r+n) + \Gdim^F(\alpha/K) \big) - \big( d(r+n) - \dim H \big) \nonumber\\
& = & \Gdim^F(\alpha/K) + \dim H
\end{eqnarray}
The fibre of $\theta_H$ in which $\zeta$ lies is $W \cap \zeta {+} TH$, so (\ref{ggc4}) says exactly that $\zeta$ lies in a fibre of $\theta_H$ of atypical dimension. By the fibre dimension theorem, there is a proper Zariski-closed subset $W_H$ of $W$, defined over $K$, containing all the fibres of $\theta_H$ of atypical dimension.

Since $\alpha$ is generic in the projection of $W$, and hence of $W_H$, the subset $V_{H,\alpha} \leteq \class{y\in V}{(\alpha,y) \in W_H}$ is proper Zariski-closed in $V$. Let $V_\alpha \leteq \bigcup_{H \in \mathcal{H}_W} H_{H,\alpha}$. Then $V_\alpha$ is also a proper Zariski-closed subset of $V$, and we have shown that $\Vadep \subs V_\alpha$.

So since $F$ is G$\Gamma$C over $K$, there is a point $\beta \in \Gamma(F)^n \cap V(F) \minus V_\alpha(F)$. Since $\beta \notin \Vadep$, $\beta$ is $\kO$-linearly independent over $\Gamma(K) \cup \alpha$. Hence $F$ is  GS$\Gamma$C over $K$ as required.

The proof for case (COR) is very similar, but instead of $\zeta \in (c_1 {+} LJ) \cross (c_2 {+} J)$ we have subgroups $J_1 \subs G_1^{r+n}$ and $J_2 \subs G_2^{r+n}$ which correspond to each other in the sense that they are connected components of solutions to the same system of $\OO$-linear equations. So we get $\zeta \in (c_1 {+} J_1) \cross (c_2 {+} J_2)$ with $\dim J_1 = \dim J_2 = \ldim_\kO(\zeta/\Gamma(K)) < r+n$. Then a similar calculation shows that $\zeta$ lies in a component of the intersection $W \cap c {+} (J_1 \cross J_2)$ of atypical dimension, and we apply the weak Zilber-Pink for the semiabelian variety $G^{r+n}$ and proceed as in case (EXP).
\end{proof}

\subsection{Sufficient conditions for quasiminimality}

\begin{theorem}
Suppose $F$ is a full $\Gamma$-field with the countable closure property which is generically $\Gamma$-closed over some countable $K \closed F$. Then $F$ is quasiminimal.
\end{theorem}
\begin{proof}
We take $\Fbase = K$ and consider the category $\Ctrans(K)$. By Theorem~\ref{Gammatr is amalg cat} it is an amalgamation category so we have a \Fraisse\ limit $\Mtr(K)$. By Theorem~\ref{Fraisse limit is QPS}, $\Mtr(K)$ is a quasiminimal pregeometry structure, so defines a quasiminimal class $\K(\Mtr(K))$. Substituting Proposition~\ref{GSGC equiv saturation} for Lemma~\ref{SGClosed lemma}, the proof of Theorem~\ref{main theorem} shows that the models in this class are precisely the full $\Gamma$-fields which are purely $\Gamma$-transcendental extensions of $K$, are $\aleph_0$-saturated with respect to the $\Gamma$-algebraic extensions which are purely $\Gamma$-transcendental over $K$, and satisfy the the countable closure property. Hence by Propositions~\ref{GSGC equiv saturation} and~\ref{GSGC implies GGC}, $F$ is in $\K(\Mtr(K))$ and hence is quasiminimal.
\end{proof}

If $F$ is the complex field, in practice it might be difficult or impossible to identify a countable $\Gamma$-closed $K$ and prove directly that 
$F$ is generically $\Gamma$-closed over $K$. Thus the following corollaries may be more useful.

\begin{corollary}
Suppose $F$ is a full $\Gamma$-field with the countable closure property which is \emph{$\Gamma$-closed}. Then $F$ is quasiminimal.
\end{corollary}
\begin{proof}
Clearly $\Gamma$-closedness implies generic $\Gamma$-closedness.
\end{proof}
Since $\Cexp$ has the countable closure property by Proposition~\ref{Gcl=Gcl'}, this completes the proof of Theorem~\ref{eac implies qm}. We can do slightly better.
\begin{corollary}\label{CCP + almost Gcl implies qm}
Suppose $F$ is a full $\Gamma$-field with the countable closure property which is \emph{almost $\Gamma$-closed}. That is, for all but countably many free and rotund, irreducible subvarieties $V \subs G^n$ of dimension $dn$, $\Gamma^n(F) \cap V(F)$ is Zariski-dense in $V$. Then $F$ is quasiminimal.
\end{corollary}
\begin{proof}
Suppose $F$ is almost $\Gamma$-closed, and take $K_1$ to be a countable subfield of $F$ over which all the countably many exceptional varieties $V$ are defined. Take $K = \Gcl^F(K_1)$. Then $F$ is generically $\Gamma$-closed over $K$.
\end{proof}

Overall we have proved the following generalization of Theorem~\ref{eac implies qm}, which applies to the exponential function, the Weierstrass $\wp$-functions, the exponential maps of simple abelian varieties, and more.
\begin{theorem}
Let $\C_\Gamma$ be an analytic $\Gamma$-field. If $\C_\Gamma$ is almost $\Gamma$-closed then it is quasiminimal. 
\end{theorem}
\begin{proof}
Combine~ \ref{CCP + almost Gcl implies qm} with Proposition~\ref{Gcl=Gcl'} which gives CCP.
\end{proof}
Since being $\Gamma$-closed implies being almost $\Gamma$-closed, Theorem~\ref{Gcl implies qm} follows. Theorem~\ref{eac implies qm} is a special case.

\begin{remark}
We do not know if almost $\Gamma$-closedness is a necessary condition for quasiminimality. For example in the exponential case, is it possible to build an uncountable quasiminimal exponential field $F$ with a definable family $(V_p)_{p \in P}$ of rotund and free varieties such that for only countably many $p$ (perhaps none) there is $(\xbar,e^\xbar) \in V_p(F)$?
\end{remark}




\end{document}